\documentclass[11pt]{amsart}
\usepackage{lmodern}
\usepackage{amsmath, amsthm, amssymb, amsfonts,enumitem,multicol, tasks}
\usepackage[normalem]{ulem}
\usepackage{hyperref}
\usepackage{mathtools}

\hypersetup{
    colorlinks=true,
    linkcolor=black,
    filecolor=magenta,      
    urlcolor=cyan,
    citecolor=magenta,
    pdftitle={Overleaf Example},
    pdfpagemode=FullScreen,
    }
    
\usepackage{mathrsfs}
\usepackage{verbatim} 
\usepackage{longtable}
\usepackage{stmaryrd}

\usepackage{mathtools}

\usepackage{tikz}
\usetikzlibrary{decorations.pathmorphing}
\tikzset{snake it/.style={decorate, decoration=snake}}

\usepackage{caption}

\usepackage{tikz-cd}
\usetikzlibrary{arrows}

\theoremstyle{plain}
\newtheorem{thm}{Theorem}[section]
\newtheorem{problem}[thm]{Problem}
\newtheorem{cor}[thm]{Corollary}
\newtheorem{lem}[thm]{Lemma}
\newtheorem{prop}[thm]{Proposition}
\newtheorem{conj}[thm]{Conjecture}
\newtheorem{question}[thm]{Question}

\theoremstyle{definition}
\newtheorem{defn}[thm]{Definition}

\newtheorem{example}[thm]{Example}

\theoremstyle{remark}
\newtheorem{rmk}[thm]{Remark}

\newcommand{\BA}{{\mathbb{A}}}

\newcommand{\BC}{{\mathbb{C}}}
\newcommand{\BD}{{\mathbb{D}}}

\newcommand{\BF}{{\mathbb{F}}}
\newcommand{\BG}{{\mathbb{G}}}

\newcommand{\BL}{{\mathbb{L}}}

\newcommand{\BP}{{\mathbb{P}}}
\newcommand{\BQ}{{\mathbb{Q}}}

\newcommand{\BT}{{\mathbb{T}}}

\newcommand{\BZ}{{\mathbb{Z}}}

\newcommand{\CA}{{\mathcal A}}

\newcommand{\CE}{{\mathcal E}}
\newcommand{\CF}{{\mathcal F}}

\newcommand{\CH}{{\mathcal H}}

\newcommand{\CO}{{\mathcal O}}

\newcommand{\CS}{{\mathcal S}}

\newcommand{\CV}{{\mathcal V}}

\newcommand{\FM}{{\mathfrak{M}}}
\newcommand{\FX}{{\mathfrak{X}}}

\newcommand{\FI}{{\mathfrak{I}}}

\newcommand{\Sym}{{\textnormal{Sym}}}

\newcommand{\pt}{{\mathsf{p}}}
\newcommand{\ch}{{\mathrm{ch}}}
\newcommand{\td}{{\mathrm{td}}}

\DeclareFontFamily{OT1}{rsfs}{}
\DeclareFontShape{OT1}{rsfs}{n}{it}{<-> rsfs10}{}
\DeclareMathAlphabet{\curly}{OT1}{rsfs}{n}{it}

\newcommand\Ext{\operatorname{Ext}}
\newcommand\Hom{\operatorname{Hom}}

\newcommand\id{\textup{id}}

\newcommand{\inv}{{\mathrm{wt}_0}}
\newcommand{\inva}{\mathrm{inv}}
\newcommand{\GL}{\mathrm{GL}}

\newcommand{\Gr}{\mathrm{Gr}}
\newcommand{\MHM}{\mathrm{MHM}}
\newcommand{\MMHM}{\mathrm{MMHM}}
\newcommand{\vir}{\mathrm{vir}}
\newcommand{\Vir}{\mathsf{Vir}}
\newcommand{\IC}{\mathrm{IC}}
\newcommand{\PE}{\mathrm{PE}}

\newcommand{\bL}{{\mathsf{L}}}
\newcommand{\bR}{{\mathsf{R}}}
\newcommand{\bT}{{\mathsf{T}}}

\newcommand{\bF}{{\mathsf{F}}}

\newcommand{\RHom}{R\CH\kern -1pt\mathit{om}}
\newcommand{\HHom}{\CH\kern -1pt\mathit{om}}

\newcommand{\PT}{\mathsf{PT}}

\newcommand{\mr}{{\mathsf{MR}}}
\newcommand{\gmr}{{\mathsf{GMR}}}
\newcommand{\br}{{\mathsf{BR}}}

\newcommand{\geo}{{\mathsf{Geom}}}

\usepackage{lmodern}
\usetikzlibrary{decorations.pathmorphing}

\begin{document}

\title[Cohomology of the moduli of 1-dim sheaves on $\mathbb{P}^2$]{Cohomology rings of the moduli of one-dimensional sheaves on the projective plane}
\date{\today}

\newcommand\blfootnote[1]{%
  \begingroup
  \renewcommand\thefootnote{}\footnote{#1}%
  \addtocounter{footnote}{-1}%
  \endgroup
}

\author[Y. Kononov]{Yakov Kononov}
\address{Yale University, Department of Mathematics}
\email{ya.kononoff@gmail.com}

\author[W. Lim]{Woonam Lim}
\address{Utrecht University, Department of Mathematics}
\email{w.lim@uu.nl}

\author[M. Moreira]{Miguel Moreira}
\address{Massachusetts Institute of Technology, Department of Mathematics}
\email{miguel73@mit.edu}

\author[W. Pi]{Weite Pi}
\address{Yale University, Department of Mathematics}
\email{weite.pi@yale.edu}

%\subjclass[2020]{14C15, 14D22, 14F06}

\keywords{Cohomology rings and relations, Virasoro representations, BPS integrality, perverse filtrations, moduli spaces of sheaves.}

\begin{abstract}
     We initiate a systematic study on the cohomology rings of the moduli stack $\mathfrak{M}_{d,\chi}$ of semistable one-dimensional sheaves on the projective plane. We introduce a set of tautological relations of geometric origin, including Mumford-type relations, and prove that their ideal is generated by certain primitive relations via the Virasoro operators. Using BPS integrality and the computational efficiency of Virasoro operators, we show that our geometric relations completely determine the cohomology rings of the moduli stacks up to degree $5$.

     As an application, we verify the refined Gopakumar--Vafa/Pandharipande--Thomas correspondence for local $\mathbb{P}^2$ in degree $5$. Furthermore, we propose a substantially strengthened version of the $P=C$ conjecture, originally introduced by Shen and two of the authors. This can be viewed as an analogue of the $P=W$ conjecture in a compact and Fano setting. 
     
     %Furthermore, we propose a global $P=C$ conjecture that substantially strengthens a previous local version of Shen and two of the authors. This can be viewed as an analogue of the $P=W$ conjecture in a compact and Fano setting. 
     
     %and  will imply, in particular, that the perverse filtration on the good moduli space $M_{d,\chi}$ is multiplicative. 
     
\end{abstract}

\maketitle

\setcounter{tocdepth}{1} 

\tableofcontents
\setcounter{section}{-1}

\section{Introduction}

\subsection{Overview and main results.} The primary purpose of this paper is twofold. First, we initiate a systematic study on the cohomology rings of the moduli stack $\mathfrak{M}_{d,\chi}$ of semistable one-dimensional sheaves on $\mathbb{P}^2$, combining a variety of ideas and tools including geometric relations, Virasoro representations, and BPS integrality. Our methods produce explicit presentations of a number of new cohomology rings in degrees $d\leq 5$. Second, we propose a substantially strengthened version of \textit{the $P=C$ conjecture} introduced by Shen and two of the authors \cite{KPS}, extending it to all cohomological degrees. This can be viewed as an analogue of the $P=W$ conjecture in a compact, Fano setting, and has various structural implications on both sides of the equality. We verify this conjecture, among other predictions, on the cohomology rings we obtain. Throughout, we work over the complex numbers $\mathbb{C}$.

\medskip

Fix two integers $d$ and $\chi$ with $d\geq 1$. We consider the moduli stack of sheaves
\[
\mathfrak{M}_{d,\chi}:=\{\textrm{semistable } {F} \in \mathrm{Coh}(\mathbb{P}^2)\mid [\mathrm{supp}({F})] = d\cdot H, \;\; \chi({F}) = \chi\},
\]
where $H$ is the class of a line, $\mathrm{supp}({F})$ denotes the Fitting support, and the stability condition is with respect to the slope
\[
\mu({F}) = \frac{\chi({F})}{c_1({F})\cdot H} \in \mathbb{Q}.
\]
This moduli stack is smooth and admits a good moduli space $M_{d,\chi}$ parametrizing semistable sheaves up to S-equivalence. When $d$ and $\chi$ are coprime, the good moduli map $\pi: \mathfrak{M}_{d,\chi} \to M_{d,\chi}$ is a trivial $B\mathbb{G}_m$-gerbe and induces an isomorphism of rings
\begin{equation}
    \label{Gm_gerbe}
    H^*(\FM_{d,\chi}, \mathbb{Q}) \simeq H^*(M_{d,\chi}, \mathbb{Q})\otimes H^*(B\mathbb{G}_m, \mathbb{Q}). 
\end{equation}
Thus, study of the cohomology rings of $\FM_{d,\chi}$ subsumes that of the coprime moduli spaces $M_{d,\chi}$. The latter has applications to enumerative geometry which we explain below. We remark that our approach is inductive and understanding the cohomology rings of $M_{d,\chi}$ makes use of the moduli stacks $\FM_{d',\chi'}$ with non-coprime $(d',\chi')$ as well. 

\smallskip

%\wocomment{Remove COHA here.} Our study of $\FM_{d,\chi}$ and its cohomology is motivated in part by cohomological Hall algebras. Cohomological Hall algebras (CoHA in short) are rich algebraic structures coming from various moduli stacks, see Kontsevich--Soibelman \cite{KS}. Indeed, for a fixed slope $\mu\in \BQ$ and the corresponding monoid of topological types $\Lambda_\mu$, the direct sum of the cohomology groups$$\bigoplus_{(d,\chi)\in \Lambda_\mu}H^*(\FM_{d,\chi},\BQ)$$
%underlies a CoHA whose structure is largely unknown. Therefore, studying $H^*(\FM_{d,\chi},\BQ)$ may provide the first step towards understanding this CoHA. Furthermore, when $d$ and $\chi$ are coprime, cohomology ring of the moduli stack recovers that of the moduli space via the isomorphism 
%\begin{equation}
%    \label{Gm_gerbe}
%    H^*(\FM_{d,\chi}, \mathbb{Q}) \simeq H^*(M_{d,\chi}, \mathbb{Q})\otimes H^*(B\mathbb{G}_m, \mathbb{Q})
%\end{equation}
%induced from the good moduli map $\pi: \mathfrak{M}_{d,\chi} \to M_{d,\chi}$, which is a trivial $B\mathbb{G}_m$-gerbe in this case. Thus our study on the cohomology of $\FM_{d,\chi}$ subsumes the coprime moduli spaces $M_{d,\chi}$ as particularly interesting cases, which we motivate below. In particular, we remark that our approach is inductive and understanding the cohomology of $M_{d,\chi}$ makes use of the moduli stacks $\FM_{d',\chi'}$ with non-coprime $(d',\chi')$ as well. \smallskip

The study of the good moduli spaces $M_{d,\chi}$ dates back to the works of Simpson \cite{simp94} and Le Potier \cite{lep93}. We refer to \cite[Introduction]{PS} for a brief overview of various directions of study toward this moduli space. One of the motivations to investigate $M_{d,\chi}$ comes from enumerative geometry. Briefly speaking, a certain \textit{perverse filtration} on the intersection cohomology
%\footnote{When $d$ and $\chi$ are not coprime, we need to use intersection cohomology. In the following, we will assume coprimality implicitly when we discuss the moduli \textit{space} $M_{d,\chi}$, unless stated otherwise.}
$\mathit{IH}^*(M_{d,\chi})$ defines interesting curve counting invariants called \textit{refined Gopakumar--Vafa invariants} for local $\BP^2$, i.e., the Calabi--Yau threefold $\mathrm{Tot}(K_{\mathbb{P}^2})$. These invariants are expected to recover curve counting invariants from other enumerative theories \cite{PT14}. We refer to \cite{KPS} for a detailed account on the history and development of this proposal. %and to \cite[Section 1.1]{KPS} for generalities concerning the perverse filtration. 
When $d$ and $\chi$ are coprime, the moduli space $M_{d,\chi}$ is a smooth projective variety and its intersection cohomology coincides with singular cohomology. Furthermore, the perverse filtration in this case is characterized by the cup product with a base ample class, cf. Section \ref{sec: perverse filtration}, which is in turn determined by the cohomology ring structure. 

\medskip

The above discussion motivates the following problem: 
%\footnote{See \cite[Introduction]{PS} for another perspective from the cohomology of Higgs moduli spaces.}

\begin{problem}
\label{question: coh ring}
    Study the cohomology rings of the moduli stacks $\FM_{d,\chi}$.
\end{problem}

We remark again that this formulation subsumes the coprime moduli spaces $M_{d,\chi}$; our approach in this paper is to treat them uniformly
%In this paper, we shall approach Problem \ref{question: coh ring} uniformly 
in the language of stacks. Indeed, this is a more general and natural setting for various constructions, and relates to the intersection cohomology of (not necessarily coprime) moduli spaces via the BPS integrality in Section \ref{sec: BPS integrality}.

%In this paper, we shall approach Problem \ref{question: coh ring} uniformly in the language of moduli stacks because they have advantage of being smooth and admitting a unique universal family. Nevertheless, they can be used to study moduli spaces $M_{d,\chi}$. When $(d,\chi)=1$, the cohomology ring of $\FM_{d,\chi}$ recovers the cohomology ring of $M_{d,\chi}$ by taking the kernel of certain explicit operator, called $\bR_{-1}$. Even when $d$ and $\chi$ are not necessarily coprime, BPS integrality formula relates the cohomology group of $\FM_{d,\chi}$ in terms of intersection cohomology groups of $M_{d',\chi'}$ where $(d,\chi)=k(d',\chi')$ for some $k\geq 1$. 

%As we will see, this is the more general setting for various constructions in this paper, and subsumes the moduli spaces $M_{d,\chi}$ in a natural way. 

\smallskip

The stack $\mathfrak{M}_{d,\chi}$ admits tautological classes coming from its universal family. We will prove that these classes generate the rational cohomology ring, cf. Theorem \ref{thm: tautgenerated}. These are the ring generators of $H^*(\mathfrak{M}_{d,\chi}, \mathbb{Q})$ we use throughout the paper. Furthermore, we formulate in Section \ref{sec: relations} certain geometric relations, including notably Mumford-type relations, among these classes. The main result of this paper, along the lines of Problem \ref{question: coh ring}, is the following. 

%we refer to Section \ref{sec: rings} for the presentations of the cohomology rings.

\begin{thm}
\label{thm: rings}
    The geometric relations in Section \ref{sec: relations} completely determine the cohomology rings of the moduli stacks $\mathfrak{M}_{d,\chi}$ for $d\leq 4$ and the smooth (i.e. coprime) moduli spaces $M_{5,\chi}$. %In particular, we obtain explicit presentations of these rings. \wecomment{Say they are on the website below}
\end{thm} 

We refer to Section \ref{sec: rings} for some essential information on these cohomology rings, e.g. their Poincaré series and degrees of the relations.\footnote{As a first indictaion of their complexity, we note that $\dim H^*(M_{5,\chi}) = 1695$ for $\chi$ coprime to 5.} Explicit presentations of the rings are available on the fourth author's website

\centerline{\url{https://github.com/Weite-Pi/weitepi.github.io}}

\noindent in a folder named \texttt{cohomology rings.}

% \footnote{As a first indication of their complexity, we note that $\dim H^*(M_{5,\chi}) = 1695$ for $\chi$ coprime to 5.} \wocomment{Moved this footnote here, what do you think?}

\medskip

Theorem \ref{thm: rings} allows us to verify several predictions in the range $d\leq 5$. These include the refined Gopakumar--Vafa/Pandharipande--Thomas correspondence, cf. Section \ref{sec: GV/PT}. 

\begin{cor}
The refined GV/PT correspondence for local $\mathbb P^2$ holds up to degree $5$. 
\end{cor}

In \cite{CM}, Chung--Moon studied the holomorphic Euler characteristics of certain line bundles on $M_{4,1}$ motivated by strange duality. Based on this, the authors formulated a conjectural formula for general degrees, cf. Conjecture \ref{conj: CM}. We verify this conjecture in degree $5$ by computing the Euler characteristics using the Hirzebruch--Riemann--Roch formula. %See Section \ref{sec: CHung-Moon} for the precise statement. 

\begin{cor}
    Chung--Moon's conjecture holds in degree $5$. 
\end{cor}

\begin{rmk} We make a few remarks on our main result. 

% \wocomment{I shortened (i).}
\begin{enumerate}
    \item [(i)] For $d\leq 4$ coprime to $\chi$, the cohomology rings of the stacks recover those of the moduli spaces via the isomorphism \eqref{Gm_gerbe}; the latter were previously known via classical or birational methods \cite{lep93, CM}, but our approach recovers them independently. All other cohomology rings in Theorem \ref{thm: rings}, i.e., those of
\[
\FM_{2,0},\; \FM_{3,0},\; \FM_{4,2},\; \FM_{4,0},\;
M_{5,1},\;
M_{5,2}
\]
are new to the best of our knowledge.\footnote{We restrict to $0\leq \chi \leq d/2$ here since these are the essential cases by symmetries, cf. Section \ref{sec: computations of the rings}.} We do not fully treat $\FM_{5,0}$ and higher degree stacks for a purely computational reason, but our approach applies readily to these cases; see Section \ref{sec: M50} for some partial results.

\smallskip

    \item [(ii)] A particularly interesting aspect of Theorem \ref{thm: rings} consists of the two moduli spaces $M_{5,1}$ and $M_{5,2}$. Indeed, they are the first non-trivial pair of (coprime) moduli spaces $M_{d,\chi}$ with the same Poincaré polynomials but \textit{different} ring structure; see \cite[Theorem 0.1]{MS_chi-indep} and \cite[Theorem 1.2]{LMP}. As such, they give a highly non-trivial testing ground for conjectures around enumerative geometry and $\chi$-(in)dependence, as exemplified by the two corollaries above.
    
    % The main result of \cite[Theorem \ref{chi-dep}]{LMP} states that the cohomology rings of $M_{d,\chi}$ with $\gcd(d,\chi)=1$ are in general $\chi$-dependent. Thus for example we obtain two non-isomorphic cohomology rings for $M_{5,1}$ and $M_{5,2}$, respectively.

    \smallskip
    
    \item [(iii)] According to \cite{ES,Mar07} and Theorem \ref{thm: tautological generation} in the case of stacks, the cycle class maps are isomorphisms for the moduli stack $\FM_{d,\chi}$ and the coprime moduli space $M_{d,\chi}$. Thus Theorem \ref{thm: rings}, and more generally all structural results for cohomology rings in this paper, hold for the Chow rings as well.
\end{enumerate}

\smallskip

For the proof of Theorem \ref{thm: rings}, we need two more ingredients called \textit{Virasoro representation} and \textit{BPS integrality} in addition to geometric relations. Roughly speaking, Virasoro representation allows us to produce geometric relations efficiently by applying certain explicit operators to `primitive' relations, cf. Theorem \ref{thm: Virasoro_intro}, while BPS integrality computes the Poincaré series of the stacks $\FM_{d,\chi}$, thus giving a completeness criterion for our relations. These three ingredients lead to a systematic strategy to compute the rings; see Section \ref{sec: strategy}. Readers might find this section helpful as a guide for reading this paper.

Finally, we remark that our strategy is applicable to the study of moduli stacks of one-dimensional sheaves on other del Pezzo surfaces. More generally, similar ideas should have applications to other smooth moduli stacks parametrizing objects in some abelian category, such as semistable torsion-free sheaves on del Pezzo surfaces, semistable bundles on curves, and semistable representations of quivers. Analogues of our three main ingredients \cite{PS, lmquivers, MozRei} are available in all the contexts mentioned. We hope this paper serves as a proof of concept for the power of this package of techniques.
    
\end{rmk}

\subsubsection{The $P=C$ conjecture.} We turn to the second goal of the paper and introduce the $P=C$ conjecture next. 
% I read this sentence again and feel it is more natural to put `next' back
The moduli spaces $M_{d,\chi}$ admit a proper map to the complete linear system of degree $d$ curves on $\BP^2$, called the Hilbert--Chow map
$$h: M_{d,\chi}\longrightarrow |d\cdot H|
$$
sending a sheaf to its Fitting support. Associated to such a proper map, there is a filtration $P_\bullet \mathit{IH}^\ast(M_{d, \chi})$ called the \emph{perverse filtration}, cf. \cite{BBDG}, which encodes interesting information about the topology of the fibration $h$. As mentioned before, numerical data of this filtration (more precisely, dimensions of $\textup{gr}^P_i \mathit{IH}^{i+j}(M_{d,\chi})$) defines the refined Gopakumar--Vafa invariants for local $\BP^2$.

On the other hand, when $d$ and $\chi$ are coprime, there is a natural multiplicative filtration $C_\bullet H^\ast(M_{d,\chi})$ called the \emph{Chern filtration}, defined using the (normalized) tautological classes $c_k(j)$, which serve as generators for our rings; see \eqref{eq: normalized class} and \eqref{eqn: global Chern filtration} for precise definitions. In \cite{KPS}, Shen and two of the authors proposed that the perverse and Chern filtrations should match up to cohomological degree $2d-4$. In Section \ref{sec: P=C}, we shall upgrade this `local' version into a statement on the entire cohomology ring as follows.

% In Section \ref{sec: P=C}, we strengthen the (local) $P=C$ conjecture of \cite{KPS} to all cohomological degrees. The $P=C$ conjecture below relates two filtrations of highly different nature on the cohomology of $M_{d,\chi}$ with $\gcd(d,\chi)=1$: the perverse filtration $P_\bullet H^*(M_{d,\chi},\BQ)$ is defined in terms of the topology of the fibration $h$  whose generic fiber is an abelian variety; the Chern filtration $C_\bullet H^*(M_{d,\chi},\BQ)$, on the other hand, is defined explicitly via the tautological classes. 

%The last section is somewhat independent from the others, in which we strengthen the (local) $P=C$ conjecture of \cite{KPS} to all cohomological degrees, and verify a number of non-trivial predictions on the cohomology rings. The $P=C$ conjecture relates two filtrations of highly different nature on the cohomology of $M_{d,\chi}$ with $\gcd(d,\chi)=1$: the perverse filtration $P_\bullet H^*(M_{d,\chi},\BQ)$ is defined in terms of the topology of a fibration from $M_{d,\chi}$ to the linear system $|d\cdot H|$ of degree $d$ planar curves, whose generic fiber is an abelian variety; the Chern filtration $C_\bullet H^*(M_{d,\chi},\BQ)$, on the other hand, is defined explicitly via the tautological classes. 

%Our new formulation below (cf. Conjecture \ref{conj: global P=C}) strengthens the local version \cite[Conjecture 0.3]{KPS} which only concerns the cohomology groups in low degrees.

\begin{conj} 
\label{conj: global P=C (intro)}
For coprime $d\geq 1$ and $\chi\in \BZ$, we have $P_\bullet H^*(M_{d,\chi}, \mathbb{Q}) = C_\bullet H^*(M_{d,\chi}, \mathbb{Q})$.
\end{conj}

We refer readers to \cite{MSY} for a discussion on the history and development of various $P=C$ phenomena. In particular, it plays a key role in the recently proved $P=W$ conjecture of de Cataldo--Hausel--Migliorini \cite{dCHM, MS_P=W, HMMS, MSY} by serving as an intermediate step $P=C=W$; the equality $P=C$ is considered the most difficult one. Note that, by the spectral correspondence \cite{BNR}, Higgs bundles on a curve $\Sigma$ are in correspondence with one-dimensional sheaves on the non-compact and symplectic surface $T^* \Sigma$. Thus Conjecture \ref{conj: global P=C (intro)} can be viewed as an analogue of the $P=W$ conjecture for the compact and Fano surface $\mathbb P^2$. One important different feature, however, is that we do not have a multiplicative \textit{splitting} of the perverse filtration on $M_{d,\chi}$ as on the moduli of Higgs bundles; this fails already for $d=3$.

As a consequence of this conjecture, structures of either the perverse or the Chern filtration can be transported to the other side. In particular, Conjecture \ref{conj: global P=C (intro)} would imply the multiplicativity of the perverse filtration and a curious Hard Lefschetz symmetry for the Chern filtration. See Remark \ref{rmk: P=C} for a detailed discussion. 

\smallskip

Again, Theorem \ref{thm: rings} allows us to verify this conjecture in low degrees. 

\begin{cor}
    The $P=C$ conjecture holds up to degree $5$. 
\end{cor}

%\smallskip Since Theorem \ref{thm: rings} determines the cohomology rings of  the coprime moduli spaces up to degree $5$, we are able to verify numerous predictions in this range: the (strengthened) $P=C$ conjecture, the (refined) Gopakumar--Vafa/Pandharipande--Thomas correspondence, and a conjecture of Chung--Moon \cite{CM}. See Section \ref{sec: conjectures} for the precise statements. 

%\wocomment{Moved the generality of our strategy to here as we discussed.} \wecomment{I will probably add a few sentence to this later}

In the rest of the introduction, we explain our main ingredients for the proof of Theorem~\ref{thm: rings} in some more detail.

%since the three ingredients we use \cite{PS, lmquivers, MozRei} apply more generally \wecomment{also mention other contexts this works, e.g. quivers, curves, torsion-free sheaves on del Pezzo surfaces}. In what follows, we explain our ingredients in some more detail.\footnote{For Virasoro representations and BPS integrality, less technically-inclined readers might find the brief ideas explained above adequate for a first reading.} \wecomment{Do you think we should include this footnote? I wrote it mainly in response to the referee's criticism...}\mcomment{I personally don't like the use of ``less technically-inclined readers'', sounds a bit patronizing}

\subsection{Mumford and geometric relations.}
\label{sec: intro_relations}

The original Mumford relations are formulated over $M_{r,d}(\Sigma)$, the moduli space of semistable holomorphic bundles of rank $r$ and degree $d$ over a fixed algebraic curve $\Sigma$ of genus $g\geq 2$. Mumford conjectured that such relations form a complete set in the rank two case \cite{AB83}. A similar but more general construction of the relations called \textit{generalized Mumford relations} is carried out in \cite{EK04}. 

\smallskip

Roughly speaking, the Mumford-type relations are obtained by constructing certain vector bundles built from the universal family over the moduli spaces, and taking the (vanishing) Chern classes of degrees beyond the rank of the bundle. Grothendieck--Riemann--Roch theorem expresses the Chern classes in terms of tautological classes on the moduli space, thus giving relations in the cohomology ring. More precisely, the constructions in \cite{AB83} and \cite{EK04} rely on the following vanishing results:

\begin{enumerate}[label=(\roman*)]
    \item Higher cohomology of the bundles due to positivity.

    \item Hom (and Ext) groups between semistable bundles due to slope differences.
\end{enumerate}

In the following, we reserve the name \textit{Mumford relations} ($\mathsf{MR}$) for the first type, and call the second type \textit{generalized Mumford relations} ($\mathsf{GMR}$) in accordance with the original terminology. For the moduli space $M_{r,d}(\Sigma)$, Mumford's conjecture in the rank two case was first proved by Kirwan \cite{Kir92}, while Earl--Kirwan \cite{EK04} showed that the generalized Mumford relations are complete for general rank coprime to the degree.

\smallskip

The study of Mumford-type relations for sheaves on (del Pezzo) \textit{surfaces} was initiated in \cite{PS} and played a key role in the proof of \cite[Theorem 1.2]{LMP}; see Remark \ref{rmk: GMR_11}. One primary goal of the current paper is to pursue this idea more systematically. In Section \ref{sec: relations}, we carry out general constructions of the two Mumford-type relations for the moduli stack $\mathfrak{M}_{d,\chi}$. In addition, there is a third type of relations, which we name \textit{base relations} ($\mathsf{BR}$), coming from the fibration structure of $\mathfrak{M}_{d,\chi}$ over the base $|d \cdot H|$. We shall refer to all of these as \textit{geometric relations.} In light of the results in \cite{Kir92, EK04}, it is natural to ask the following question:

\begin{question}
\label{completeness}
    Are the geometric relations complete for general $\mathfrak{M}_{d,\chi}$?
\end{question}

The answer to Question \ref{completeness} turns out to be positive for all moduli stacks (and spaces) in Theorem \ref{thm: rings}, \textit{except for} $M_{5,1}$. In fact, we will see that Mumford and generalized Mumford relations are already complete for these moduli stacks. The situation for $M_{5,1}$ is more subtle: the geometric relations are, somewhat surprisingly,  not complete in the literal sense, but they are still sufficient to determine the entire ring structure when combined with Poincaré duality. We refer to Section \ref{sec: M51} for a detailed explanation.

% Original formulation on $M_{r,d}(C)$, Earl--Kirwan's completeness results, counterpart for $\mathfrak{M}_{d,\chi}$.

\subsection{Virasoro representations.} 

The Virasoro constraints predict a rich set of relations among tautological invariants of moduli spaces in various settings. It is named so because the predicted relations are described by certain representations of the half of the Virasoro algebra $\Vir_{\geq -1}:=\textup{span}\{\bL_n\,|\,n\geq -1\}$, where $\mathsf{L}_n$ are certain operators (cf. Section \ref{sec: virasoro operators} for the definition of these operators in our setting) satisfying the Virasoro relations 
\[
[\mathsf{L}_n, \mathsf{L}_m] = (m - n)\, \mathsf{L}_{n+m}.
\]
The Witten--Kontsevich theorem \cite{Witten, Kontsevich} was the first instance of the Virasoro constraints. This was conjecturally generalized to Gromov--Witten theory \cite{EHX}. After being transported to the sheaf side via the GW/PT correspondence \cite{MOOP}, there have been rapid developments in the sheaf-theoretic Virasoro constraints \cite{M, vanBree, BLM, Arkadij_quiver, lmquivers}. 

For the moduli spaces of our interest, namely $M_{d,\chi}$, the Virasoro constraint was first conjectured in \cite{BLM} and proven for the coprime cases in \cite{lmquivers} using quiver representations. The Virasoro constraints is a priori a statement about top degree relations. However, it was observed in \cite{lmquivers} that the constraints imply that a representation of $\Vir_{\geq -1}$ on the descendent algebra $\BD_{d,\chi}$ (see Section \ref{sec: descendent algebra}) factors through the surjective homomorphism\footnote{There are two representations of $\Vir_{\geq -1}$ on $\BD_{d,\chi}$, one from $\bL_n$ operators and the other from $\bR_n$ operators, cf. Section \ref{sec: virasoro operators}. When we discuss the factorization property of the $\Vir_{\geq -1}$-representation through $H^*(\FM_{d,\chi})$, we always consider the one coming from $\bR_n$ for computational simplicity.}
$$\BD_{d,\chi}\twoheadrightarrow H^*(\FM_{d,\chi},\BQ).
$$
In other words, the representation preserves the kernel of this surjection, which we call the ideal of tautological relations. In fact, this factorization property was proven for any $\FM_{d,\chi}$ with $d$ and $\chi$ not necessarily coprime. Since we study the ideal of tautological relations using geometric relations, the following question naturally arises. 
\begin{question}
    Does the representation of $\Vir_{\geq -1}$ on $\mathbb{D}_{d,\chi}$ interact nicely with the ideals of geometric relations?
\end{question}

This question is interesting because there has been very little geometric understanding on the Virasoro constraints. In Section \ref{sec: Virasoro rep}, we shall give a positive answer to this question in the form of the following theorem.

%By using the combinatorial identities from the appendix, we prove how $\Vir_{\geq -1}$-representation interacts with the ideal of geometric relations of each type.

\begin{thm}
\label{thm: Virasoro_intro}
The $\Vir_{\geq -1}$-representation on $\BD_{d,\chi}$ preserves the ideal of geometric relations of each type (MR, GMR, and BR). Furthermore, in each case the entire ideal is generated by certain primitive relations, cf. Definition \ref{def: primitive relations}, as a $\BD_{d,\chi}\otimes U(\Vir_{\geq -1})$-module. 
\end{thm}

Among the entire set of generators for the geometric relations, the primitive ones form a significantly smaller subset. This brings considerable computational efficiency when we work with the ideal of geometric relations using a computer program. %cf. Section \ref{sec: strategy}. 

\subsection{BPS integrality.}
\label{sec: BPS intro}
One of the most subtle and important aspects in the study of tautological rings is whether a given set of relations is complete, as inquired in Question \ref{completeness}. When the moduli space is a smooth projective variety, Poincaré duality provides an easy criterion for the completeness of a set of relations. But unless we are in such a case, it is in general difficult to determine whether the relations we have found are complete. 

One practical method is to compute the Poincaré series of the moduli stacks (and spaces) in question. In this paper, we use the BPS integrality formula of Mozgovoy--Reineke \cite{MozRei} to obtain the Poincaré series of the moduli stacks $\FM_{d,\chi}$.
% before knowing its cohomology ring in our inductive approach. 
Roughly speaking, the BPS integrality formula relates the virtual Poincaré series of the moduli stacks and the virtual intersection Poincaré polynomials of the moduli spaces for a fixed slope $\mu\in \BQ$. In order to derive an equality of the actual Poincaré series of the moduli stacks, we need the following result, cf. Theorem \ref{thm: tautological generation}.

\begin{thm}
    The cohomology ring $H^*(\FM_{d,\chi},\BQ)$ is tautologically generated and pure. Furthermore, the cycle class map from Chow to cohomology is an isomorphism.
\end{thm}

This theorem shows in particular that $H^*(\FM_{d,\chi},\BQ)$ is of $(p,p)$-type. The corresponding property for $\mathit{\mathit{IH}}^*(M_{d,\chi},\BQ)$ was proven by Bousseau \cite{Bou}. Therefore, the BPS integrality formula implies an equality between the shifted (intersection) Poincaré series
\begin{equation}\label{eq: BPS introduction}
    \sum_{(d,\chi)\in \Lambda_\mu}\overline{E}(\FM_{d,\chi},q)\cdot e^{(d,\chi)}=\PE\left(\frac{-q^{1/2}}{1-q}\cdot \sum_{(d,\chi)\in \Lambda_\mu\backslash\{(0,0)\}}\overline{\mathit{IE}}(M_{d,\chi},q)\cdot e^{(d,\chi)}
\right) 
\end{equation}
where $\PE(-)$ denotes the plethystic exponential. Note that on the right hand side we have moduli spaces $M_{d,\chi}$ that are not necessarily smooth, whose (intersection) cohomology is a priori more mysterious. Nonetheless, the $\chi$-independence theorem of Maulik--Shen \cite[Theorem 0.1]{MS_chi-indep} gives an isomorphism of graded vector spaces
    \begin{equation}\label{eqn: MS_chi_indep}
    \mathit{IH}^*(M_{d,\chi}) \simeq \mathit{IH}^*(M_{d,\chi'})
    \end{equation}
for any $\chi, \chi' \in \mathbb{Z}$ not necessarily coprime to $d$. Thanks to this isomorphism, we can use Equation \eqref{eq: BPS introduction} to effectively compute the Poincaré series of $\FM_{d,\chi}$ in terms of the Poincaré polynomials of the smooth moduli spaces $M_{d',1}$ where $(d,\chi)=k(d',\chi')$ for some $k\geq 1$. The latter can be computed directly (at least for $d \leq 5$) using geometric relations and the Poincaré duality criterion; see Section \ref{sec: rings}. Knowing the Poincaré series then allows us to check whether our geometric relations are complete for the moduli stacks.

\subsection{Further directions.} We outline a few directions of future research.

\begin{enumerate}[label=(\roman*)]
    \item The relations in $H^*(M_{5,1}, \mathbb{Q})$ that cannot be obtained as geometric relations, cf. Section \ref{sec: M51}, seem rather mysterious to us. It is natural to ask if one can obtain these relations by other geometric means, with the aim of proving a result similar to \cite{EK04}.

    \smallskip

    \item The $P=C$ conjecture gives various predictions on the perverse and the Chern filtrations on $M_{d,\chi}$, see Remark \ref{rmk: P=C}. Confirming these predictions independently would in turn give more evidence to the conjecture. In particular, Remark \ref{rmk: P=C} (ii) and (iv) are closely related to the theme of finding relations in this paper. We will prove in Section \ref{sec: br} a subset of the vanishing Chern monomials predicted by Remark \ref{rmk: P=C} (iv), and we expect similar methods to establish more results along this line.

    \smallskip

    \item In the setting of Higgs bundles on curves, there is a generalized version of the $P=W$ conjecture for stacks by \cite[Section 1.3]{Dav}. Motivated by this, we provide strong numerical evidence toward a \textit{stacky} $P=C$ conjecture in our setting. More precisely, it is verified that the graded dimensions of the perverse and the Chern filtrations match in all degrees for the moduli stacks we compute, cf. Theorem \ref{thm: stacky P=C}. A systematic study of the stacky $P=C$ conjecture would likely require a better understanding of the interaction between BPS integrality and the Chern filtration. %We plan to explore this direction in the future. 
\end{enumerate}

\subsection{Notations and conventions} All cohomology rings $H^*(-)$ take $\mathbb{Q}$-coefficients if unspecfied, and will be considered only for smooth spaces. For possibly singular spaces, we use the intersection cohomology $\mathit{IH}^*(-)$. The shifted (intersection) virtual Poincaré series, denoted by $\overline{E}(-,q)$ or $\overline{\mathit{IE}}(-,q)$, is defined as the non-shifted one multiplied by $(-q)^{-\dim(-)/2}$.

\smallskip

Here are some notations we use throughout the paper:

\smallskip

\begin{tasks}[label = {}, column-sep=-17em, item-indent=1.5em](2)
        \task $\alpha,\:\! \alpha'$
        
        \task Topological type of one-dimensional sheaves

        \task $\mathbb{D}, \:\!\mathbb{D}_{\mathrm{wt}_0}$

        \task Descendent algebra and the weight zero subalgebra

        \task $\mathbb{D}_\alpha,\:\! \mathbb{D}_{\alpha,\mathrm{wt}_0}$
        
        \task Descendent algebra for topological type $\alpha$

        \task $\mathrm{ch}_i(\gamma)$

        \task Formal descendent symbols

        \task $c_k(j)$

        \task Normalized tautological classes\footnotemark
    
        \task $\FI_{d,\chi}, \:\! I_{d,\chi}$

        \task Ideal of tautological relations for $\FM_{d,\chi}$ and $M_{d,\chi}$

        \task $\FI^\bullet_{d,\chi},\:\! I^\bullet_{d,\chi}$

        \task Ideal of geometric relations, where  $\bullet \in \{\mathsf{MR, GMR, BR, Geom}\}$

        %\task $\mathfrak{P}^\bullet_{d,\chi}$

        %\task Primitive relations, where  $\bullet \in \{\mathsf{MR, GMR, BR}\}$

        \task $\mathsf{L}_n, \:\!\mathsf{R}_n$
        
        \task Virasoro operators acting on $\mathbb{D}$

        \task $\mathsf{L}_n^\delta, \:\!\mathsf{R}_n^\delta$

        \task Normalized Virasoro operators acting on $\mathbb{D}_\mathrm{wt_0}$

        \task $\Lambda,\:\! \Lambda_\mu$

        \task Monoid of topological types (and of fixed slope $\mu$)

        \task $P_\bullet, \:\! C_\bullet$

        \task Perverse and Chern filtrations on cohomology
    
\end{tasks}

\footnotetext{{We refer to Section \ref{sec: relations between classes} for the precise relations between $\mathrm{ch}_i(\gamma)$ and $c_k(j)$.}}

\subsection{Acknowledgements.} We would like to thank Ben Davison, Andres Fernandez Herrero, Davesh Maulik, Sebastian Schlegel Mejia, Rahul Pandharipande, Junliang Shen, and Ruijie Yang for helpful conversations and comments. WL is supported by SNF-200020-182181 and ERC Consolidator Grant FourSurf 101087365. MM was supported during part of the project by ERC-2017-AdG-786580-MACI. The project received funding from the European Research Council (ERC) under the European Union Horizon 2020 research and innovation programme (grant agreement 786580).

\section{Moduli of sheaves and descendent algebras}

%\begin{enumerate}
 %   \item A Few words about stack and unique universal sheaf
  %  \item Definition of descendent algebra and $c_i(j)$ and $ch_k(\gamma)$ and $R_{-1}$
   % \item Definition of tautological relations
    %\item GMR, MR, base dimension relations on stack
%\end{enumerate}

%Notation: $S=\BP^2$, $M=M_{d, \chi}$ for some $(d, \chi)$, $\FM=\FM_{d,\chi}$. There is canonical universal sheaf $\CF$ on $S\times \FM$ and non-canonical universal sheaves $\BF$ on $S\times M$. Maybe add some words about what we mean by cohomology of stacks. Maps $p$ and $q$ from product to moduli space and $S$, respectively.

\subsection{Descendent algebra}\label{sec: descendent algebra}

Our goal in this paper is to study the cohomology rings $H^*(\FM_{d,\chi})$ for general $(d,\chi)$ and $H^*(M_{d,\chi})$ for $\gcd(d,\chi)=1$ in terms of generators and relations. A natural set of generators can be constructed using the Chern classes of universal sheaves. It is useful to introduce the free algebra generated by the formal descendent symbols. We follow the notation and presentation in \cite{BLM}, and later we establish the connection to the normalized generators $c_k(j)$ used in \cite{PS,KPS,LMP}.

    \begin{defn}[Descendent algebra]
Let $\BD$ be the free commutative and unital $\BQ$-algebra generated by the formal symbols
\[\ch_i(\gamma)\,\textup{ for }\,i\geq 0, \;\gamma\in H^*(\mathbb{P}^2, \mathbb{Q})\]
modulo the linearity relations
\[\ch_i(\lambda_1\gamma_1+\lambda_2\gamma_2)=\lambda_1\ch_i(\gamma_1)+\lambda_2\ch_i(\gamma_2)\textup{\, for }\gamma_1, \gamma_2\in H^*(\mathbb{P}^2), \,\lambda_1, \lambda_2\in \BQ\,.\]
Given $\alpha=(d,\chi)$, we let $\BD_\alpha$ be the quotient of $\BD$ by the relations\footnote{Here, $\ch(\alpha)$ denotes the Chern character of any one-dimensional sheaf of type $\alpha$.}
\[\ch_0(\gamma)=\int_{\mathbb{P}^2} \ch(\alpha)\cdot \gamma, \quad \gamma \in H^*(\mathbb{P}^2).\]

%Given $\alpha\in K(\mathbb{P}^2)$, for example $\alpha=(d, \chi)$, we let $\BD_\alpha$ be the quotient of $\BD$ by the relations\wo{Here the meaning of taking $\ch(\alpha)$ is slightly ambiguous since it is defined as a tuple $(d,\chi)$. We can mention in a footnote that $\ch(\alpha)$ means $\ch(F)$ for any one dimensional sheaf $F$ of type $(d,\chi)$ or explicitly given by $(0,d,\chi-\frac{3d}{2})\in H^*(\BP^2)$?}
%\[\ch_0(\gamma)=\int_{\mathbb{P}^2} \ch(\alpha)\cdot \gamma, \quad \gamma \in H^*(\mathbb{P}^2).\]
    \end{defn}
We call $\BD$ the descendent algebra. Using the universal sheaf $\CF$ on $\FM_{d,\chi}\times \mathbb{P}^2$, we can realize descendents in the cohomology on moduli stacks. Explicitly, denoting $p$ and $q$ the respective projections to $\FM_{d,\chi}$ and to $\mathbb{P}^2$, we have a map of algebras
\[\xi\colon \BD\to H^*(\FM_{d,\chi})\]
given by
\[\ch_i(\gamma)\mapsto \Big[p_\ast\big(\ch(\CF)\cdot q^\ast \gamma)\Big]_{2i}\in H^{2i}(\FM_{d,\chi})\,\]
where the notation  $[-]_{2i}$ means that we extract the degree $2i$ part of a non-homogeneous cohomology class. In particular, we have
\[\ch_i(H^j)\mapsto p_\ast\big(\ch_{i+2-j}(\CF)\cdot q^\ast H^j\big)\in H^{2i}(\FM_{d,\chi})\,.\]
Note that for a given $\alpha = (d,\chi)$, the map $\xi$ factors through a map $\BD_\alpha\to  H^*(\FM_{d,\chi})$, which we still denote by $\xi$.

\smallskip

For the good moduli space $M_{d,\chi}$ with $\gcd(d,\chi)=1$, we can only define a map
\[\xi_\BF\colon \BD\to H^*(M_{d,\chi})\]
after we fix the choice of a universal sheaf $\BF$, which exists by \cite[Corollary 4.6.7]{HL}. There is, however, a subalgebra $\BD_\inv\subset \BD$ on which a canonical map is defined.

\begin{defn}\label{def: R -1}
    Let $\bR_{-1}: \BD\to \BD$ be a derivation defined on generators by
    \[\bR_{-1}(\ch_i(\gamma))=\ch_{i-1}(\gamma)\,,\]
    where we set $\ch_{-1}(\gamma)= 0$. We define
    \[\BD_\inv=\ker \bR_{-1}\subset \BD.\]
    Similarly, we can define $\BD_{\alpha, \inv}\subset \BD_{\alpha}$.
\end{defn}
As explained in \cite[Section 2.4]{BLM}, the restriction of $\xi_{\BF}$ to $\BD_\inv$ does not depend on the choice of $\BF$, so there is a canonical map
\[\BD_\inv\to H^*(M_{d,\chi})\,.\]

\medskip

Descendents indeed generate the cohomology of moduli spaces and moduli stacks. 

\begin{thm}\label{thm: tautgenerated}Let $d, \chi\in \BZ$ with $d>0$.%\wocomment{Minor point: putting stack for the second part may give impression to reader (not us) that gcd=1 is assumed also for the stack. What about either switching the order or emphasizing in (ii) that (d,chi) is general?}
\begin{enumerate}
    \item[\textup{(i)}] For coprime $(d,\chi)$, the map $\BD_{\alpha, \inv}\to H^*(M_{d, \chi})$ is surjective.
    \item[\textup{(ii)}] For general $(d,\chi)$, the map $\BD_\alpha\to H^*(\FM_{d, \chi})$ is surjective.
\end{enumerate}
\end{thm}
The statement (i) for the moduli spaces follows from \cite[Theorem 0.2 (a)]{PS} together with the description of the generators $c_k(j)$ in terms of $\BD_{\alpha,\inv}$ that we will explain in Section \ref{sec: normalized class}. The statement (ii) for the stacks is new and is the content of Theorem \ref{thm: tautological generation}. 

\begin{defn}

Let $I_{d, \chi}$ and $\FI_{d, \chi}$ be the kernels of the maps 
    \[\BD_{\alpha,\inv}\to H^*(M_{d, \chi})\textup{ \,and\, }\BD_\alpha\to H^*(\FM_{d, \chi})\,,\]
    respectively.
\end{defn}

Studying these ideals is the main goal of this paper. We start with a basic property of the ideals on the stack by using the geometric interpretation of the operator $\bR_{-1}$.

\begin{prop}\label{prop: R-1geom}
    The operator $\bR_{-1}$ preserves the ideal $\FI_{d, \chi}$. Thus it descends to $H^*(\FM_{d, \chi})$, i.e. there exists a morphism completing the diagram
    \begin{center}
        \begin{tikzcd}\label{diagram: R-1geometric}
            \BD_\alpha\arrow[r, "\bR_{-1}"] \arrow[d, "\xi"]& \BD_\alpha\arrow[d, "\xi"]\\
            H^*(\FM_{d, \chi})\arrow[r, dashed]&H^*(\FM_{d, \chi}).
        \end{tikzcd}
    \end{center}
\end{prop}
\begin{proof}
    We show this by constructing the operator on the cohomology of the stack that makes the diagram commute. There is a $B\BG_m$-action 
    \[\mathrm{act}\colon \FM_{d, \chi}\times B\BG_m\to  \FM_{d, \chi}\]
    induced by the $\BG_m$-automorphisms of sheaves. Then we define the map
    \[H^*(\FM_{d, \chi})\xlongrightarrow{\mathrm{act}^\ast} H^*(\FM_{d, \chi}\times B\BG_m)\cong H^*(\FM_{d,\chi})[\zeta]\xlongrightarrow{[\zeta^1]} H^*(\FM_{d, \chi}) \]
    where the last map takes the $\zeta^1$ coefficient of a polynomial in $\zeta$. The argument in \cite[Lemma 4.9]{BLM} shows that this makes the diagram in Proposition \ref{diagram: R-1geometric} commute. \qedhere
\end{proof}

\subsection{Normalized universal sheaves and generators.}

\label{sec: normalized class}

Although the algebra of weight zero descendents is a very useful object from a conceptual point of view, it is not so easy to work with in concrete computations, for instance to describe explicitly the ideal $I_{d, \chi}$ for small $d$. 

Instead, we can normalize the universal sheaf on $M_{d,\chi} \times \mathbb{P}^2$, which is the approach taken in \cite{PS}. There, the authors take an arbitrary universal sheaf $\BF$ and modify it by
\[\BF^{\textup{norm}}=\BF\otimes e^{A_\BF}\]
where $A_\BF\in H^2(M_{d,\chi}\times \mathbb{P}^2, \BQ)$ is chosen in a way such that $\BF^{\textup{norm}}$ becomes independent of $\BF$. More precisely, we have
\begin{defn}\label{defprof: cij}
    Given a universal sheaf $\BF$, let $A_\BF\in H^2(M_{d,\chi}\times \mathbb{P}^2, \BQ)$ be the class in \cite[Lemma 2.5]{KPS} and define the \textit{normalized classes}
    \begin{equation}
    \label{eq: normalized class}
    c_k(j):=p_\ast \big(\ch_{k+1}(\BF^{\textup{norm}})\cdot q^\ast H^j\big)\in H^{2k+2j-2}(M_{d, \chi}), \quad j \in \{0,1,2\}
    \end{equation}
    where $\ch_{k+1}(\BF^{\textup{norm}})$ is the degree $2k+2$ part of $\ch(\BF)\cdot e^{A_\BF}$. 
    \end{defn}

    The class $A_\mathbb{F}$ is the unique class such that the following holds \cite[Proposition 1.3]{PS}:
    \begin{equation}\label{eq: normalization}c_1(0)=0\in H^0(M_{d, \chi})\textup{ \,and\, }c_1(1)=0\in H^2(M_{d, \chi})\,.\end{equation}
    Moreover, the normalized universal sheaf $\BF^\mathrm{norm}$, and hence the classes $c_k(j)$, does not depend on the choice of $\BF$. 

    \smallskip
    
    These normalized tautological classes have nice properties established in \cite{PS, KPS, Yuan2}. Before stating them, we introduce two more structures on the moduli space $M_{d,\chi}$. First, it admits a proper and flat \textit{Hilbert--Chow map} 
    \begin{equation}
    \label{hilbert-chow}
    h: M_{d,\chi} \to |d\cdot H| = \mathbb{P}^b, \quad \CF \mapsto \mathrm{supp}(\CF)
    \end{equation}
    sending a sheaf to its Fitting support, where 
    \[
    b = \dim H^0(\mathbb{P}^2, \mathcal{O}(d))-1 = \frac{1}{2}d(d+3).
    \]
    This is a weakly abelian fibration in the sense of \cite{MS_chi-indep}, and a general fiber of $h$ is a Picard variety whose dimension is $g=\frac{1}{2}(d-1)(d-2)$. Second, the moduli space $M_{d,\chi}$ (and similarly the stacks $\FM_{d,\chi}$) admits two types of symmetry:  
    \begin{enumerate}
    \item[(a)] The first type is given by the isomorphism
    \[
    \psi_1 : M_{d,\chi} \xrightarrow{\sim} M_{d,\chi+d}, \;\; \mathcal{F} \mapsto \mathcal{F} \otimes \mathcal{O}_{\mathbb{P}^2}(1).
    \]
    \item[(b)] The second type \cite[Theorem 13]{Mai10} is given by the duality isomorphism
    \[
    \psi_2 : M_{d,\chi} \xrightarrow{\sim} M_{d,-\chi}, \;\; \mathcal{F} \mapsto \CE \kern -1.5 pt \mathit{xt}^1(\CF, \omega_{\BP^2}).
    \]
    \end{enumerate}

The following $\chi$-dependence theorem is proved in \cite{LMP}. It provides context to our main results and contrast with the $\chi$-independence of the (intersection) Poincaré polynomials of $M_{d,\chi}$ as in \eqref{eqn: MS_chi_indep}.

\begin{thm}\label{chi-dep}
    For $d\geq 1$ and $\chi, \chi'$ coprime to $d$, there is an isomorphism of graded $\mathbb{Q}$-algebras
    \[
    H^*(M_{d,\chi}) \simeq H^*(M_{d,\chi'})
    \]
    if and only if $\chi$ and $\chi'$ are related by the two symmetries above.
\end{thm}

\smallskip

Now we return to the normalized tautological classes $c_k(j)$.

\begin{prop}[{\cite[Section 1.2]{PS}}] With the above notations,

\label{prop: c_kj_property}

    \begin{enumerate}
    \item[\textup{(i)}] The class $c_0(2)$ is the pull-back of the hyperplane class via $h$, and $c_2(0)$ is relatively ample. These two classes span $H^2(M_{d,\chi}, \mathbb{Q})$ for $d\geq 3$.
    \item[\textup{(ii)}] We have
    \[
    \psi_1^* c_k(j) = c_k(j), \quad \psi_2^* c_k(j) = (-1)^kc_k(j).
    \] 
    \end{enumerate}
\end{prop}

\begin{thm}[{\cite{PS, Yuan2}}]\label{generation} Assume $d \geq 3$. We have:
\begin{enumerate}
    \item[\textup{(i)}] $H^*(M_{d,\chi})$ is generated as a $\BQ$-algebra by the $3d-7$ classes\footnote{We call these $3d-7$ classes \textit{normalized tautological generators.}} of degrees $\leq 2d-4$:
    \[
   c_0(2),\, c_2(0) \in H^2(M_{d,\chi}), \;\;  c_{k}(0),\, c_{k-1}(1),\, c_{k-2}(2) \in H^{2k-2}(M_{d,\chi}),\;\; 3\leq k \leq d-1.
    \]
    \item[\textup{(ii)}] There are no relations among these $3d-7$ classes in degrees $\leq 2d-2$, and exactly three linearly independent relations in degree $2d$ if $d\geq 5$.
    
\end{enumerate}

\end{thm}

\begin{rmk}
\label{rmk: normalized_class}
    One reason we introduce the normalized classes $c_k(j)$ is that the vanishing of $c_1(0)$ and $c_1(1)$ simplifies computations and we will use them to present our cohomology rings. A more important motivation is that this choice of normalization gives the correct formulation of the local $P=C$ conjecture; see \cite[Proposition 1.2]{KPS}. %\wocomment{"Notation received influence" doesn't read naturally. What about "In fact, definition of the normalization $c_k(j)$ was influenced by literature on $P=W$ conjecture where similar normalized classes on the Dolbeault moduli spaces appear, see for example \cite{ MS_P=W}."} 
    In fact, the choice of normalization in the definition of $c_k(j)$ has been influenced by the $P=W$ literature, where similar normalized classes on the Dolbeault  moduli spaces are denoted by $c_k(\gamma)$, see for example \cite{ MS_P=W}. A more canonical and normalization-free definition of the Chern filtration will be given in Section \ref{sec: chern filtration}.
    
    %\we{I think I agree now that the specific normalization is not crucial, following Woonam's previous note on defining Chern filtration using $\mathbb{D}_{\mathrm{wt}_0}$. Also not sure if we discussed this explicitly, but I think it is indeed true that
    %\[
    %C_\bullet H^*(M_{d,\chi}):= \mathrm{Im}\big(\mathbb{D}_\mathrm{\alpha, wt_0}\cap C_\bullet \mathbb{D}_{d,\chi} \to H^*(M_{d,\chi})\big)
    %\]
    %gives the same Chern filtration as we have been using. This is implicit in Miguel's identification $\mathbb{D}_\mathrm{\alpha,wt_0}\simeq \mathbb{Q}[c_0(2),c_2(0), \cdots]$.}
    
    %A very rough argument is as follows: first, the Chern filtration (defined by tautological classes $c_k(j)$) is not affected by twisting from $\mathbb{P}^2$. This is because twisting from $\mathbb{P}^2$ only modifies a given tautological class by classes with lower Chern grading, for example it preserves $c_0(2)$, changes $c_2(0)$ to $c_2(0) + \lambda_1 c_1(1) + \lambda_2 c_0(2)$. So essentially we can ignore twisting from $\mathbb{P}^2$, and we are left with twisting from line bundles on $M_{d,\chi}$. But by definition of $\mathbb{D}_\mathrm{wt_0}$, the realization map does not depend on this, so in particular we can choose the normalization that recovers $c_k(j)$ that we've been using. This shows the match between two definitions.
    
    % Does this look correct to you? If so, we can modify the remark above, and maybe add a remark on this canonical definition of Chern filtration in the $P=C$ section.}
\end{rmk}

We now proceed to understand the generators $c_k(j)$ in terms of the canonical map $\BD_{\alpha,\inv}\to H^*(M_{d, \chi})$. The class $A_\BF$ can be decomposed as a sum of classes in $H^2(M_{d, \chi})$ and $H^2(\mathbb{P}^2)$; we study the effect of twisting by a line bundle on $M_{d, \chi}$ and on $\mathbb{P}^2$ separately, following ideas in \cite[Sections 2.5 and 2.7]{BLM}.

\subsubsection{Twist by line bundles on $M_{d,\chi}$.} Let $\alpha = (d,\chi) \in K(\BP^2)$ and let $\delta=H/d\in H^*(\mathbb{P}^2)$. This choice has the property that
\[\ch_0(\delta)=\int_{\BP^2} \ch(\alpha)\cdot \delta=1 \in \BD_\alpha\,.\]
We make the following definition:
\begin{defn}
    For any universal sheaf $\BF$, we let 
    \[\BF_\delta=\BF\otimes e^{-\xi_\BF(\ch_1(\delta))}\,.\]
    We call $\BF_\delta$ the $\delta$-normalized universal sheaf, which is characterized by the property that
    \[\xi_{\BF_\delta}(\ch_1(\delta))=0\,.\]
\end{defn}
Note that $\BF_\delta$ does not depend on the initial choice of $\BF$. Indeed, if $\BF'=\BF\otimes p^\ast L$ then 
\[\BF'\otimes e^{-\xi_{\BF'}(\ch_1(\delta))}=\BF'\otimes  e^{-\xi_{\BF}(\ch_1(\delta))-c_1(L)\xi_{\BF}(\ch_0(\delta))}=\BF\otimes e^{-\xi_{\BF}(\ch_1(\delta))}\,.\]
The effect of normalizing on descendents is controlled by the operator 
\[\eta=\eta_\delta\colon\, \BD_\alpha\to \BD_\alpha\]
given by
\[\eta=\sum_{j\geq 0}\frac{(-1)^j}{j!}\ch_1(\delta)^j\bR_{-1}^j\,.\]
\begin{prop}\label{prop: propertieseta}
    Let $\alpha\in K(\mathbb{P}^2)$ and $\delta\in H^*(\mathbb{P}^2)$ be such that $\int_{\mathbb{P}^2} \ch(\alpha)\cdot \delta=1$, and let $\eta=\eta_\delta$ be the operator above. We have the following:
    \begin{enumerate}
        \item[\textup{(i)}] $\bR_{-1}\circ \eta=0$, so
\[\eta\colon \BD_\alpha\to \BD_{\alpha, \inv}\,.\]
\item[\textup{(ii)}] $\eta$ is an algebra homomorphism.
\item[\textup{(iii)}] $\eta$ is  the identity when restricted to $\BD_{\alpha, \inv}\subset \BD_\alpha$, and $\eta(\ch_1(\delta))=0$.
\item[\textup{(iv)}] The composition 
\[\BD_\alpha\xlongrightarrow{\eta} \BD_{\alpha, \inv}\longrightarrow H^*(M)\]
is precisely $\xi_{\BF_\delta}$. 
    \end{enumerate}
\end{prop}
\begin{proof}
\begin{enumerate}
    \item[(i)] We have
    \[[\bR_{-1}, \ch_1(\delta)^j]=j[\bR_{-1}, \ch_1(\delta)]\ch_1(\delta)^{j-1}=j\ch_0(\delta)\ch_1(\delta)^{j-1}=j\ch_1(\delta)^{j-1}\]
    where $[-,-]$ denotes commutator of operators and we write $\ch_1(\delta)^j$ for the operator of multiplication by $\ch_1(\delta)^j\in \BD_\alpha$. After applying the commutator above, the sum $\bR_{-1}\circ \eta$ telescopes and we obtain zero.
    \item[(ii)] This is a formal consequence of the fact that $\bR_{-1}$ is a derivation. Indeed, by the general Leibniz rule,
    \begin{align*}
        \eta(D_1D_2)=\sum_{j\geq 0}\frac{(-1)^j}{j!}\ch_1(\delta)^j\sum_{s+t=j}\binom{j}{s}\bR_{-1}^s(D_1)\bR_{-1}^t(D_2)=\eta(D_1)\eta(D_2)\,.
    \end{align*}
  
    \item[(iii)] Both claims are immediate. %since $\ch_0(\delta)=1$ in $\BD_\alpha$.
    \item[(iv)] Apply \cite[Lemma 2.8]{BLM}, specializing to $\zeta=-\xi_{\BF}(\ch_1(\delta))$.\qedhere
\end{enumerate}
\end{proof}
The idea of $\delta$-normalization and the operator $\eta$ allow for a more concrete description of $\BD_{\alpha, \inv}$. By (i)-(iii) in Proposition \ref{prop: propertieseta}, $\eta$ induces an algebra homomorphism 
\[\eta\colon \BD_{\alpha}/\langle\ch_1(\delta)\rangle\to \BD_{\alpha, \inv}\,.\]
We show that this map is an isomorphism in the next proposition. For this, we define algebra homomorphisms
\[\phi\colon \BD_{\alpha, \inv}[u]\to \BD_\alpha\textup{ and } \varphi\colon \BD_\alpha\to \BD_{\alpha, \inv}[u]\]
as follows: $\phi$ restricted to $\BD_{\alpha, \inv}$ is the inclusion $\BD_{\alpha, \inv}\hookrightarrow \BD_{\alpha}$ and $\phi(u)=\ch_1(\delta)$, while
\[\varphi=\sum_{i\geq 0}\frac{u^i}{i!}\eta\circ \bR_{-1}^i\,.\]
The same proof that $\eta$ is an algebra homomorphism shows that $\varphi$ is as well.

\begin{prop}\label{prop: isowt0}
The homomorphisms $\phi$ and $\varphi$ are inverse isomorphisms. In particular, 
\[\eta\colon \BD_{\alpha}/\langle\ch_1(\delta)\rangle\to \BD_{\alpha, \inv}\,\]
is an isomorphism.
\end{prop}
\begin{proof}
    We need to show that both $\varphi\circ \phi$ and $\phi\circ \varphi$ are the identity. By Proposition \ref{prop: propertieseta} (iii), $\varphi\circ \phi$ is the identity when restricted to $\BD_{\alpha, \inv}$. Moreover,
    \[\varphi(\phi(u))=\varphi(\ch_1(\delta))=\eta(\ch_1(\delta))+u\cdot \eta(\ch_0(\delta))=u\,.\]
    For $\phi\circ \varphi$ we have
    \[\phi\circ \varphi=\sum_{i\geq 0}\sum_{j\geq 0}\frac{\ch_1(\delta)^i}{i!}\frac{(-1)^j\ch_1(\delta)^j}{j!}\bR_{-1}^{i+j}=\textup{id}\]
    by the binomial formula. The statement for $\eta$ follows from the commutative diagram
    \begin{center}
        \begin{tikzcd}
          \BD_\alpha \arrow[d, two heads]  \arrow[r, "\varphi"] &\BD_{\alpha, \inv}[u]\arrow[d, two heads]\\
            \BD_\alpha/\langle\ch_1(\delta)\rangle\arrow[r, "\eta"]& \BD_{\alpha, \inv}\,
        \end{tikzcd}
    \end{center}
    where the arrow on the right sends $u$ to zero.
\end{proof}

\subsubsection{Twist by line bundles on $\mathbb{P}^2$.}\label{sec: twist by P2}
Given $\rho\in H^2(\mathbb{P}^2)$, let $\bF_\rho\colon \BD\to \BD$ be the algebra isomorphism given by
\[\bF_\rho\big(\ch_i(\gamma)\big)=\ch_i(e^\rho \gamma)\,.\]
Given any universal sheaf $\BF$ on $M_{d,\chi}\times \mathbb{P}^2$, we have
\[\xi_{\BF\otimes e^\rho}=\xi_\BF\circ \bF_\rho\,.\]
Note that $\bF_\rho$ commutes with $\bR_{-1}$, so it is also well-defined as a map $\BD_\inv\to \BD_\inv$. It can also be easily checked that $\bF_\rho$ induces 
\[\bF_\rho\colon \BD_{e^{\rho}\cdot\alpha}\to \BD_{\alpha}\textup{ \,and\, }\bF_\rho\colon \BD_{e^{\rho}\cdot\alpha,\inv}\to \BD_{\alpha,\inv}\,.\]
Indeed, 
\[\bF_{\rho}\left(\ch_0(\gamma)-\int_S \ch(e^\rho\cdot\alpha)\gamma\right)=\ch_0(e^\rho\gamma)-\int_S \ch(\alpha)e^\rho\gamma\]
is zero in $\BD_{\alpha}$.

\subsubsection{Normalized classes in terms of descendent algebra}
\label{sec: relations between classes}
We can now finally explain how the classes $c_k(j)$ can be described in terms of the descendent algebra. Let 
\[\delta=H/d\textup{ \;and\; }\rho=\left(\frac{3}{2}-\frac{\chi}{d}\right)H\,.\]
Note that the class $\rho$ satisfies
\[\int_{\BP^2} \ch(\alpha) \cdot e^\rho=0\,.\]
Consider the composition
\begin{center}
    \begin{tikzcd}
    \label{eq: compositioncij}
    \BD_{\alpha \cdot e^\rho}\arrow[r, "\eta_{\delta}"]&  \BD_{\alpha \cdot e^\rho, \inv}\arrow[r, "\bF_\rho"]&  \BD_{\alpha, \inv}\arrow[r, "\xi"]&H^*(M_{d,\chi})\,.
        \end{tikzcd}
    \end{center}
Then the normalized class $c_k(j)\in H^*(M_{d,\chi})$ can be obtained \textit{as the image of $\ch_{k+j-1}(H^j) \in \BD_{\alpha \cdot e^\rho}$ under the composition $\xi\circ \mathsf{F}_\rho \circ \eta_\delta$.}

To justify this, we check that the normalization (\ref{eq: normalization}) is satisfied. Indeed, $\ch_0(1)=0$ in $\BD_{\alpha\cdot e^\rho}$ due to our choice of $\rho$, and the composition annihilates $\ch_1(H)$ by Proposition \ref{prop: propertieseta} (iv). We will also write $c_k(j)$ for the image of $\ch_{k+j-1}(H^j)$ in $\BD_{\alpha, \inv}$, and Proposition \ref{prop: isowt0} identifies $\BD_{\alpha, \inv}$ with a polynomial ring in infinitely many variables:
\[\BD_{\alpha, \inv}\cong \BQ[c_0(2), c_2(0), c_1(2), \cdots]\,.\]
The variables are those in \eqref{eq: normalized class} with positive cohomological degrees, but excluding $c_1(1)$. In particular, we can identify the ideal $I_{d,\chi}$ with an ideal in this polynomial ring.

\smallskip

When working with the moduli stacks $\FM_{d,\chi}$, we have a canonical universal sheaf and there is no need to consider the $\delta$-normalized sheaf. However, to get a description of tautological classes and generators that is uniform under the isomorphism $\FM_{d, \chi}\cong \FM_{d, \chi+d}$, cf. Proposition \ref{prop: c_kj_property} (ii), we shall still use the twist by $\bF_\rho$ and denote by $c_k(j)\in H^*(\FM_{d,\chi})$ the image of $\ch_{k+j-1}(H^j)$ after the composition
\begin{center}
    \begin{tikzcd}
    \label{eq: compositioncij_stack}
    \BD_{\alpha \cdot e^\rho}\arrow[r, "\bF_\rho"]&  \BD_{\alpha}\arrow[r, "\xi"]&H^*(\FM_{d,\chi})\,.
        \end{tikzcd}
    \end{center}
As in the case of moduli spaces, we use the isomorphism
\[\BD_{\alpha}\cong \BQ[c_0(2), c_1(1), c_2(0), c_1(2), \cdots]\]
to identify the ideal $\FI_{d, \chi}$ with an ideal of the above polynomial ring in infinitely many variables. Note that we have an extra variable $c_1(1)$ now
%\[\begin{tikzcd}[ampersand replacement=\&]
	%{\mathbb{D}_{\alpha\cdot e^\rho}} \&\& {\mathbb{D}_\alpha} \&\& {H^*(\mathfrak{M}_{d,\chi})} \\
	%{\mathbb{D}_{\alpha\cdot e^\rho}} \& {\mathbb{D}_{\alpha\cdot e^\rho, \mathrm{wt}_0}} \& {\mathbb{D}_{\alpha,\mathrm{wt}_0}} \&\& {H^*(M_{d,\chi})}
	%\arrow["{\eta_\delta}", from=2-1, to=2-2]
	%\arrow["{\mathsf{F}_\rho}", from=2-2, to=2-3]
	%\arrow["\xi", two heads, from=2-3, to=2-5]
	%\arrow["\xi", two heads, from=1-3, to=1-5]
	%\arrow["{\mathsf{F}_\rho}", from=1-1, to=1-3]
	%\arrow["{\mathrm{id}}"', from=1-1, to=2-1]
	%\arrow[dashed, from=1-3, to=2-3]
%\end{tikzcd}\]
and the map $\eta\colon \BD_{\alpha}\to \BD_{\alpha, \inv}$ becomes identified with the quotient map
\begin{equation}
\label{eq: eta_identification}
    \BQ[c_0(2), c_1(1), c_2(0), c_1(2),\cdots]\to \BQ[c_0(2), c_2(0), c_1(2), \cdots]
\end{equation}
that sets $c_1(1)$ to 0.

\begin{rmk}
    In Section \ref{sec: chern filtration}, we will use $\mathbb{D}_{\alpha,\mathrm{wt_0}}$ and the identifications above to give a canonical (and equivalent) definition of the Chern filtration on $H^*(M_{d,\chi})$. Previously, the Chern filtration is only defined in \cite{KPS} using the \textit{normalized} classes, which is very explicit but conceptually less than ideal. We expect this new definition to shed more light on the structure of the Chern filtration in various $P=C$ phenomena.
\end{rmk}

\begin{rmk}
By Proposition \ref{prop: c_kj_property} (i), for $d\geq 3$ the second cohomology $H^2(M_{d, \chi})$ has dimension 2 and basis $\{c_0(2), c_2(0)\}$, while $H^2(\FM_{d, \chi})$ has dimension 3 and basis $\{c_0(2), c_1(1), c_2(0)\}$.
\end{rmk}

\subsection{Descending tautological relations from the stacks.}
When we write down and study tautological relations, it will be more natural to do so for the moduli stacks where canonical universal sheaves exist. When $\gcd(d,\chi)=1$, the good moduli map
\[\pi\colon \FM_{d,\chi}\to M_{d, \chi}\]
is a trivial $\BG_m$-gerbe, so in particular we have an isomorphism
\begin{equation}
\label{eq: BG_m iso}
H^*(\FM_{d, \chi})\cong H^*(M_{d, \chi}\times B\BG_m)=H^*(M_{d, \chi})[u]\,
\end{equation}
where we write $u$ for the generator of $H^2(B\BG_m)$. However, such an isomorphism is non-canonical as it depends on a choice of universal sheaf $\BF$ on $M_{d,\chi}\times \mathbb{P}^2$. Indeed, given $\BF$, by the universal property of $\FM_{d,\chi}$ there is a section $s_\BF\colon M_{d,\chi}\to \FM_{d,\chi}$ such that $s_\BF^\ast \CF=\BF$, and this section gives a trivialization 
\[\mathrm{act}\circ (s_\BF\times \textup{id}_{B\BG_m})\colon M_{d,\chi}\times B\BG_m\to \FM_{d,\chi}\,\]
where $\mathrm{act}$ is the $B\BG_m$-action on $\FM_{d,\chi}$, cf. the proof of Proposition \ref{prop: R-1geom}. In light of the identifications explained above, it is natural to consider the isomorphism associated to the $\delta$-normalized universal sheaf; even if $\BF_\delta$ might not be represented by an actual universal sheaf, it still induces an isomorphism \eqref{eq: BG_m iso} with rational coefficients.

\begin{lem}\label{lem: idealcoarse}
    Suppose that $\gcd(d,\chi)=1$. There is an isomorphism $H^*(\FM_{d, \chi})\cong H^*(M_{d, \chi})[u]$ that makes the following diagram commute:
    \begin{center}
        \begin{tikzcd}
            \BD_\alpha\arrow[r, "\varphi"]\arrow[d]& \BD_{\alpha, \inv}[u]\arrow[d]\\
            H^*(\FM_{d,\chi})\arrow[r]&H^*(M_{d,\chi})[u]
        \end{tikzcd}
    \end{center}
    In particular, we have $I_{d,\chi}=\eta\big(\FI_{d,\chi}\big)\,$, or more explicitly 
    \[
    \FI_{d,\chi}|_{c_1(1)=0} = I_{d,\chi}
    \]
    via the identification \eqref{eq: eta_identification}.
\end{lem}

\begin{proof}
    The idea of the proof is essentially the same as in \cite[Proposition 3.24]{Jo18}, so we will just sketch it. Let $\BF$ be any universal sheaf; there is an $m\in \BZ_+$ and a $B\BZ_m$-fibration $f\colon T\to M_{d,\chi}$ such that $T$ has a line bundle $L$ with $c_1(L)=-f^\ast\xi_{\BF}\big(\ch_1(\delta)\big).$ Then $f^\ast \BF\otimes L$ is a sheaf on $T\times \mathbb{P}^2$ with the property that 
    \[\ch(f^\ast \BF\otimes L)=f^\ast \ch(\BF_\delta)\,.\]
    By the universal property, we get a map $s\colon T\to \FM_{d,\chi}$. A $B\BZ_m$-fibration induces an isomorphism $f^\ast$ on rational cohomology. Thus, we obtain the required isomorphism by composing
    \[H^*(\FM_{d,\chi})\xlongrightarrow{\mathrm{act}^\ast}H^*(\FM_{d,\chi}\times B\BG_m)\xlongrightarrow{s^\ast\otimes \textup{id}}H^*(T\times B\BG_m) \xlongleftarrow{f^\ast\otimes \textup{id}} H^*(M_{d,\chi}\times B\BG_m)\,.\]
    The diagram is easily shown to commute using Proposition \ref{prop: propertieseta}. The last statement follows from the fact that $\eta$ is obtained by applying $\varphi$ and then setting $u\mapsto 0$. 
\end{proof}

\section{Geometric relations}
\label{sec: relations}

We will now focus on proving three families of relations --- we call these families \textit{Mumford relations} ($\mr$), \textit{generalized Mumford relations} ($\gmr$) and \textit{base relations} ($\br$) ---  among the generators of the descendent algebra. In other words, we will construct explicit ideals $\FI^{\mr}_{d, \chi}$, $\FI^{\gmr}_{d, \chi}$ and $\FI^{\br}_{d, \chi}$, and show that
\[\FI^{\mr}_{d, \chi}, \FI^{\gmr}_{d, \chi}, \FI^{\br}_{d, \chi}\subset \FI_{d,\chi}\,.\]

We will denote by $\FI^\geo_{d,\chi}$ the ideal generated by $\FI^{\mr}_{d, \chi}, \FI^{\gmr}_{d, \chi}, \FI^{\br}_{d, \chi}$. As we will explain, the determination of $\FI^{\gmr}_{d, \chi}$ requires knowledge of the cohomology rings of $\FM_{d', \chi'}$ for `smaller' values $(d', \chi')$, leading to an inductive approach to understanding this ideal and more generally the cohomology ring structure.

\medskip

Having defined these ideals, the natural question is the following (cf. Question \ref{completeness}):
\begin{question}\label{completeness2}
    Do we have $\FI^\geo_{d,\chi}=\FI_{d,\chi}$?
\end{question}

If the answer to this question were positive, we would have a recursive algorithm to compute the cohomology rings of $\FM_{d,\chi}$. We will look into Question \ref{completeness2} in Section \ref{sec: rings} for small $d$; it turns out that the answer is positive for $d\leq 4$ and $(d,\chi)=(5,2)$, but negative for $(d, \chi)=(5,1)$.

\medskip

We work at the level of stacks where universal sheaves exist, but this also gives relations for the good moduli spaces $M_{d,\chi}$ when $\gcd(d, \chi)=1$ by Lemma \ref{lem: idealcoarse}. More precisely, if we define 
\[I^\bullet_{d,\chi}=\eta(\FI^\bullet_{d,\chi})\quad\textup{for \,} \bullet=\mr, \gmr, \br, \geo\]
then we have $I^\geo_{d,\chi}\subset I_{d,\chi}$.

\subsection{Mumford relations}

\label{sec: MR}

As explained in Section \ref{sec: intro_relations}, one general approach to proving geometric relations on the cohomology of moduli spaces is to construct some vector bundle $\CV$ out of the moduli data (in our case, the universal sheaf $\BF$) with a rank $r$ that we can compute by Hirzebruch--Riemann--Roch; then the vanishing of Chern classes
\[c_j(\CV)=0\quad\textup{for }j>r\,\]
may give non-trivial relations. Usually $\CV$ is first constructed as a complex, and then shown to be a vector bundle by some cohomological vanishing. Mumford relations follow this strategy applied to the following vanishing result: 

\begin{prop}\label{prop: vanishingMR}
    Let $F$ be a semistable sheaf of type $(d, \chi)$. Assume
    \[\chi\geq g_d\coloneqq \frac{(d-1)(d-2)}{2},\]
    then $H^1(F)=H^2(F)=0$.
\end{prop}

\begin{proof}
    This is proven in \cite[Proposition 2.1.3]{DM}, but we provide a proof for completeness. The vanishing $H^2(F)=0$ is obvious since $F$ is one-dimensional. Let $C'$ be the scheme-theoretic support of $F$, thus $0<d':=\deg(C')\leq d$. Suppose that $H^1(F) \neq 0$, we can pick an element $s\neq 0$ in the dual
    \[
    H^1(C', F)^\vee = \mathrm{Hom}(F, \omega_{C'}) = \mathrm{Hom}(F, \mathcal{O}_{C'}(-3+d')).
    \]
    This gives rise to a non-zero morphism
    \[
    s: F \to \mathcal{O}_{C'}(-3+d'),
    \]
    whose image we denote by $G(-3+d')$ with $G\subset \mathcal{O}_{C'}$. By the semistability of $F$, we have
    \begin{equation}\label{G_slope}
    \frac{\chi}{d}\leq \frac{\chi(G(-3+d'))}{\deg G} = \frac{\chi(G)}{\deg G} - 3 + d'.
    \end{equation}
    To estimate the RHS, we consider the diagram
    \[\begin{tikzcd}[ampersand replacement=\&]
	\&\&\& T \\
	0 \& G \& {\mathcal{O}_{C'}} \& {\mathcal{O}_{C''}} \& 0 \\
	0 \& {\widetilde{G}} \& {\mathcal{O}_{C'}} \& {\mathcal{O}_{C_\mathrm{pure}''}} \& {0.}
	\arrow[from=2-1, to=2-2]
	\arrow[from=2-2, to=2-3]
	\arrow[from=2-3, to=2-4]
	\arrow[from=2-4, to=2-5]
	\arrow[from=1-4, to=2-4]
	\arrow[two heads, from=2-4, to=3-4]
	\arrow[hook', from=2-2, to=3-2]
	\arrow[from=3-2, to=3-3]
	\arrow[Rightarrow, no head, from=2-3, to=3-3]
	\arrow[from=3-3, to=3-4]
	\arrow[from=3-4, to=3-5]
	\arrow[from=3-1, to=3-2]
\end{tikzcd}\]
Here $C''$ is a closed subvariety of $C'$, and $\mathcal{O}_{C''_\mathrm{pure}}$ is the quotient of $\mathcal{O}_{C''}$ by its torsion part $T$. Thus it is the structure sheaf of a planar curve of degree $0\leq d''<d'$. We have
\begin{equation*}
\frac{\chi(G)}{\deg G} \leq \frac{\chi(\widetilde{G})}{\deg G}=\frac{\chi(\mathcal{O}_{C'})-\chi(\mathcal{O}_{C''_\mathrm{pure}})}{d'-d''}.
\end{equation*}
Using the formula $\chi(\mathcal{O}_{C_d}) =1-g_d= 1-\frac{(d-1)(d-2)}{2}$ for any degree $d$ planar curve, an elementary computation shows that
\[
\max_{0\leq d''<d'\leq d} \left(\frac{\chi(\mathcal{O}_{C'})-\chi(\mathcal{O}_{C''_\mathrm{pure}})}{d'-d''}-3+d'\right) = \frac{d-3}{2}.
\]
It follows from \eqref{G_slope} that $\chi\leq \frac{d(d-3)}{2}$, but this contradicts our assumption $\chi \geq \frac{(d-1)(d-2)}{2}$.
\end{proof}

By the cohomology and base change theorem, it follows that for $(d, \chi)$ with $\chi \geq g$ the complex $Rp_\ast \CF\in D^b(\FM_{d, \chi})$ is actually a vector bundle with fibers $H^0(F)$ over $[F]\in \FM_{d,\chi}$. Moreover, the rank of this vector bundle is $\chi$. It follows that $c_j(Rp_\ast \CF)=0$ for every $j>\chi\geq g$. To write down this formula in terms of descendents, Grothendieck--Riemann--Roch yields%\wocomment{We can mention GRR and Newton in a separate sentence since they are used in different sentences. Also, we may consider writing down the Newton's identity once for readers (e.g. in Appendix) since we use multiple times.} \wecomment{Changed slightly; I don't think we need to write down Newton's identity since it's basically what's stated below} 
\[\ch(Rp_\ast \CF)=p_\ast\big(\ch(\CF)\td(T_p)\big)=p_\ast\big(\ch(\CF)q^\ast\td(\BP^2)\big)\,,\]
so $\ch_j(Rp_\ast \CF)$ is given by the geometrical realization of $\ch_j(\td(\BP^2))$. Thus by Newton's identity, the total Chern class $c(Rp_\ast \CF)$ is the geometric realization of
\[\exp\left(\sum_{j\geq 1}(-1)^{j-1}(j-1)!\ch_j(\td(\BP^2))\right).\]
The upshot of the above discussion is the following:

\begin{prop}\label{prop: MR}
Suppose that $j>\chi\geq g$. Then 
\[\left[\exp\left(\sum_{j\geq 1}(-1)^{j-1}(j-1)!\ch_j(\td(\mathbb{P}^2))\right)\right]_{2j}\in \FI_{d,\chi}\,.\]
\end{prop}

For a fixed $(d, \chi)$, we can use the isomorphisms $\FM_{d,\chi}\cong \FM_{d,\pm\chi+k d}$ to obtain relations by Proposition \ref{prop: MR} if we take sufficiently large $k$. Given an arbitrary $(d, \chi)$, a choice of signs $\pm$ and $j, k \in \mathbb{Z}$ such that $j>\pm \chi+kd\geq g$, we get a relation
\[\mr_{d,\chi}^{\pm, k, j}\in \FI_{d,\chi}\]
which corresponds to 
\[\left[\exp\left(\sum_{k\geq 1}(-1)^{k-1}(k-1)!\ch_k(\td(\BP^2))\right)\right]_{2j}\in \FI_{d,\pm \chi+k d}\,.\]
under the ismomorphism $\FI_{d,\chi}\cong \FI_{d,\pm \chi+k d}$.

\smallskip

We define the ideal
\[\FI^\mr_{d,\chi}:=\big\langle \mr_{d,\chi}^{\pm, k, j}\big \vert\, j>\pm \chi+kd\geq g\big\rangle\subset \FI_{d,\chi}\,.\]

\begin{rmk}
\label{rmk: dual MR}
    The relations of the form $\mr_{d,\chi}^{-, k, j}$ can also be understood as coming from the dual version of Proposition \ref{prop: vanishingMR}, namely
    \[H^0(F)=H^2(F)=0\quad \textup{if }\chi\leq -g.\]
\end{rmk}

\begin{example}
\label{M10example}
    Recall that $\FM_{1,0}$ is a $\BG_m$-gerbe over $M_{1,0} \simeq \check{\BP}^2$ where $\check{\BP}^2$ is the dual projective plane. Denote the dual hyperplane section by $\check{H}$. %\wocomment{Notation of $\check{()}$ explained.} 
    Then we have
    \begin{equation}
    \label{M10}
    H^*(\FM_{1,0}, \BQ) \simeq H^*(M_{1,0}) \otimes H^*(B\BG_m) =  \BQ[{\check{H}}]/{\check{H}}^3\otimes \BQ[u].
    \end{equation}
    We work this out using Mumford relations. Write
    \[
    \gamma_k = \mathrm{td}(\BP^2)\cdot e^{k H} \;\; \textrm{and} \;\; \CV_k = Rp_*\left(\CF\otimes q^* \CO_{\BP^2}(k)\right), \;\;\; \textrm{for } k\in \BZ.
    \]
    By the above discussions, we have
    \begin{align*}
    \label{M10_vanishing}
        & c_j(-\CV_{-1}) = 0 \quad \textrm{for } j>1,\quad c_j(\CV_{0}) = 0 \quad \textrm{for } j>0,\\
        & c_j(\CV_{1}) = 0 \quad \textrm{for } j>1, \quad c_j(\CV_{2}) = 0 \quad \textrm{for } j>2.
    \end{align*}
    Set further $\alpha = \ch_1(\gamma_1)$ and $\beta = \ch_1(\gamma_{-1})$. The second vanishing above implies $\mathrm{ch}_k(\gamma_0)=0$, and the first and third imply
    \begin{equation}
    \label{alphabeta}
    \ch_k(\gamma_1) = \frac{1}{k!}\alpha^k,\quad \ch_k(\gamma_{-1}) = \frac{(-1)^{k+1}}{k!}\beta^k \quad \textrm{for } k\geq 1.
    \end{equation}
    Since $\{\gamma_{-1}, \gamma_0, \gamma_1\}$ form a basis of $H^*(\BP^2)$, it follows from Theorem $\ref{thm: tautgenerated}$ that $\alpha, \beta$ generate $H^*(\FM_{1,0})$. Finally, the last vanishing gives
    \begin{equation}
        \label{M10_rel}
        c_3(\CV_2) = \frac{1}{6}\ch_1(\gamma_2)^3-\ch_1(\gamma_2)\ch_2(\gamma_2)+2\ch_3(\gamma_2) = 0.
    \end{equation}
    Using $\gamma_2 = 3\gamma_1+\gamma_{-1}-3\gamma_0$ and \eqref{alphabeta}, the relation  \eqref{M10_rel} simplifies into $(\alpha+\beta)^3 = 0$.  Thus
    \[
    H^*(\FM_{1,0}, \BQ) = \BQ[\alpha, \beta]/(\alpha+\beta)^3,
    \]
    matching the presentation \eqref{M10}. Note that 
    \[\alpha+\beta=\ch_1(\gamma_1+\gamma_{-1})=\ch_1(\gamma_1+\gamma_{-1}-2\gamma_0)=\ch_1(H^2)=c_0(2)
    \]
    is indeed the (pull-back of) hyperplane class $\check H$, cf. Proposition \ref{prop: c_kj_property} (i).  
\end{example}

\subsection{Generalized Mumford relations}

\label{sec: gmr}

Following the same strategy but using cohomological vanishing provided by stability, we can construct a generalized version of Mumford relations.

\begin{prop}\label{prop: vanishinggmr}
    Let $F, F'$ be semistable sheaves on $\mathbb{P}^2$ of topological types $(d,\chi)$ and $(d', \chi')$, respectively. Suppose that 
    \begin{equation}\label{eq: inequalitygmr}
    \frac{\chi'}{d'}<\frac{\chi}{d}<\frac{\chi'}{d'}+3\,.\end{equation}
    Then we have
    \[\Hom(F, F')=\Ext^2(F,F')=0\,.\]
\end{prop}

\begin{proof}
    The vanishing $\mathrm{Hom}(F, F') = 0$ follows from the first inequality in (\ref{eq: inequalitygmr}) and the semistability of $F$ and $F'$. To see that $\mathrm{Ext}^2(F, F') =0$, we use Serre duality
    \[
    \mathrm{Ext}^2(F, F') \simeq \mathrm{Hom}(F', F\otimes \mathcal{O}_{\mathbb{P}^2}(-3))^\vee,
    \]
    which vanishes due to the second inequality in (\ref{eq: inequalitygmr}).
\end{proof}

Thanks to Proposition \ref{prop: vanishinggmr}, whenever $(d,\chi)$ and $(d', \chi')$ satisfy the inequality \eqref{eq: inequalitygmr}, we can construct a vector bundle on the product $\FM_{d,\chi}\times \FM_{d',\chi'}$ of the moduli stacks . Indeed, let $\CF$ and $\CF'$ be the universal sheaves on $\FM_{d,\chi}\times \BP^2$ and $\FM_{d', \chi'}\times \BP^2$, respectively. Then the complex
\[\CV := \RHom_p(\CF, \CF')[1]\]
is a vector bundle, where $p\colon \FM_{d,\chi}\times \FM_{d',\chi'}\times \BP^2 \to \FM_{d,\chi}\times \FM_{d',\chi'}$ and we omit the pull-backs of $\CF, \CF'$ to the triple product. This vector bundle has fiber $\Ext^1(F, F')$ over a point $([F], [F'])\in  \FM_{d,\chi}\times \FM_{d',\chi'}$, and its rank is equal to
\[\dim \mathrm{Ext}^1(F,F') = -\chi(F, F')=d\cdot d'\,.\]
Thus we obtain a relation
\[c_j(-\RHom_p(\CF, \CF'))=0\quad\textup{for }j>dd'\,\]
which holds in $H^*(\FM_{d,\chi}\times \FM_{d', \chi'})$. To extract from this a relation on $\FM_{d,\chi}$, we take a homology class $A\in H_*(\FM_{d', \chi'})$ and integrate along it, i.e. we apply the slant product map
\[\int_A\colon H^*(\FM_{d,\chi}\times \FM_{d', \chi'})\to H^*(\FM_{d,\chi})\,.\]

To describe explicitly the relations in terms of descendents, we use Grothendieck--Riemann--Roch  again to find a representative of $-\RHom_p(\CF, \CF')$ in the descendent algebra. For this, we consider the realization map
\[\BD\otimes \BD\to H^*(\FM_{d,\chi})\otimes H^*(\FM_{d', \chi'})\,.\]
Define $C$ by
\[C=\exp\left(\sum_{\substack{a,b\geq 0\\(a,b)\neq (0,0)}}\sum_i (-1)^{b+d_i^L}(a+b-1)!\ch_a(\gamma_i^L)\otimes \ch_b(\gamma_i^R)\right)\in \BD\otimes \BD\]
where 
\begin{equation}\label{eqn: Kunneth}
    \sum_i \gamma_i^L\otimes \gamma_i^R=\Delta_\ast \td(\BP^2)
\end{equation}
is the Künneth decomposition of the push-forward of $\td(\BP^2)$ along the diagonal and $\gamma_i^L\in H^{2d_i^L}(\BP^2)$. More concretely, we take the decomposition
\[\Delta_\ast \td(\BP^2)=H^2\otimes 1+H\otimes H+1\otimes H^2+\frac{3}{2}H\otimes H^2+\frac{3}{2}H^2\otimes H+H^2\otimes H^2\,.\]
We let $C_j$ be the part of $C$ with cohomological degree $2j$ (recall that by definition $\ch_a(\gamma)$ has cohomological degree $2a$).

\begin{lem}\label{lem: geomrealizationC}
The geometric realization of $C_j$ equals $c_j(-\RHom_p(\CF, \CF'))$. 
\end{lem}

\begin{proof}
    This is a standard application of Grothendieck--Riemann--Roch, a diagonal trick and Newton's identity. We include a sketch of the proof for completeness, see \cite[Section 3.1]{Shen_JLMS} for details. By Grothendieck--Riemann--Roch, we have
\[
\mathrm{ch}(\RHom_p(\CF, \CF')) = p_*(\mathrm{ch}^\vee (\mathcal{F}') \cdot \mathrm{ch}(\mathcal{F})\cdot q^* \td(\BP^2)). 
\]
    Consider the following commutative diagram, where $\delta = \id \times \Delta$ and $p=\pi_{12} \circ \delta$:
    \[\begin{tikzcd}[ampersand replacement=\&]
	{\mathfrak{M}_{d,\chi}\times \mathfrak{M}_{d',\chi'} \times \mathbb{P}^2} \& {\mathfrak{M}_{d,\chi}\times \mathfrak{M}_{d',\chi'} \times \mathbb{P}^2\times \mathbb{P}^2} \& {\mathfrak{M}_{d,\chi}\times \mathfrak{M}_{d',\chi'}} \\
	{\mathbb{P}^2} \& {\mathbb{P}^2\times \mathbb{P}^2.}
	\arrow["\id\times \Delta", from=1-1, to=1-2]
	\arrow["{\pi_{12}}", from=1-2, to=1-3]
	\arrow["{q}"', from=1-1, to=2-1]
	\arrow["{\pi_{34}}"', from=1-2, to=2-2]
	\arrow["\Delta", from=2-1, to=2-2]
\end{tikzcd}\]
By projection formula and the Künneth decomposition \eqref{eqn: Kunneth}, we have
\begin{align*}
p_*(\mathrm{ch}^\vee (\mathcal{F}') \cdot \mathrm{ch}(\mathcal{F})\cdot q^* \td(\BP^2)) 
&=\sum_i\pi_{12,*}\left(\mathrm{ch}^\vee (\mathcal{F}') \cdot \mathrm{ch}(\mathcal{F}) \cdot \pi_{34}^* (\gamma_i^L\otimes \gamma_i^R)\right)\\
&=\sum_i\sum_{a,b\geq 0}(-1)^{a+2-d_i^L}\ch_a(\gamma_i^L)\otimes \ch_b(\gamma_i^R).
\end{align*}
The sign in the last equality is obtained by tracing back the degree of $\ch^\vee(\CF')$ used for $\ch_a(\gamma_i^L)$. The lemma follows by applying Newton's identity as in Section \ref{sec: MR}.
\end{proof}

Given $\alpha=(d,\chi),\, \alpha'=(d', \chi')$ satisfying \eqref{eq: inequalitygmr}, $j>dd'$ and $A\in H_*(\FM_{d',\chi'})$, we define the relation $\gmr_\alpha^{\alpha', j, A}\in \BD_\alpha$ to be the image of $C_j$ under
\[\BD\otimes \BD\to \BD_\alpha\otimes \BD_{\alpha'}\xlongrightarrow{\textup{id}\otimes \xi} \BD_\alpha\otimes H^*(\FM_{d', \chi'})\xlongrightarrow{\int_A}\BD_\alpha\,.\]

The upshot of the above discussion is the following:

\begin{prop}
    For every $\alpha=(d,\chi),\, \alpha'=(d', \chi')$ satisfying \eqref{eq: inequalitygmr}, $j>dd'$ and $A\in H_*(\FM_{d',\chi'})$, we obtain a tautological relation $\gmr_{\alpha}^{\alpha', j, A}\in \FI_{d,\chi}$.
\end{prop}

Thus, we define the ideals
\[\FI^{\gmr,\alpha'}_{\alpha}=\Big\langle\gmr_{\alpha}^{\alpha', j, A}\big \vert\, j>dd',\,A\in H_*(\FM_{d', \chi'})\Big\rangle\subset \FI_{d, \chi}\,.\]

%The isomorphisms $\FM_{d,\chi}\cong \FM_{d,\chi+kd}$ do not produce new relations since it can easily be checked that $I^{\gmr,d', \chi'}_{d,\chi}=I^{\gmr,d', \chi'+kd'}_{d,\chi+kd}$. However, using the isomorphism $\FM_{d,\chi}\cong \FM_{d,-\chi}$ we can define a new ideal
%\[I^{\gmr,d', \chi', -}_{d,\chi}\subset I^\Taut_{d, \chi}\]
%for any $d', \chi'$ that satisfy
%\begin{equation}\label{eq: GMR-inequality}\frac{\chi'}{d'}-3<\frac{\chi}{d}<\frac{\chi'}{d'}\end{equation}
%which corresponds to $I^{\gmr,d', -\chi', }_{d,-\chi}$ under the isomorphism $I^\Taut_{d, \chi}\cong I^\Taut_{d, -\chi}$. 

\begin{rmk}

There are several seemingly different ways to produce generalized Mumford relations. First, one can easily check that $\FI^{\gmr,\alpha'}_{\alpha}=\FI^{\gmr,d', \chi'+kd'}_{d,\chi+kd},$ so the isomorphisms $\FM_{d,\chi}\cong \FM_{d,\chi+kd}$ do not give new relations. One can also use the isomorphism $\FM_{d,\chi}\cong \FM_{d,-\chi}$ and look at $(d', \chi')$ satisfying
the dual version of the inequality (\ref{eq: inequalitygmr}). A third way is to use the vector bundle $\RHom_p(\CF', \CF)$ instead of $\RHom_p(\CF, \CF')$. The latter two approaches are actually equivalent by an application of Grothendieck--Verdier duality \cite[Section 3.4]{Huy}. Nevertheless, one can show that they also do not produce new relations.  
\end{rmk}

\begin{rmk}
    In the case $\gcd(d', \chi')=1$, it is often more efficient to use the good moduli space $M_{d', \chi'}$ to produce relations. We use $(-)^\dagger$ to denote the dual of an operator. Given $\overline{A}\in H^*(M_{d', \chi'})$, we can consider the relation $\gmr_{\alpha}^{\alpha', j, \overline{A}}\coloneqq \gmr_{\alpha}^{\alpha', j, A}$ where $A=\eta^\dagger([M_{d', \chi'}]\cap \overline{A})$. By Lemma \ref{lem: R-1quadraticidentity}, we have
    \[
    \gmr^{\alpha', j, \mathsf{R}_{-1}^\dagger(A)}_{\alpha}=(j-1-dd')\gmr^{\alpha', j-1, A}_\alpha.
    \]
    From this identity and Lemma \ref{lem: idealcoarse}, we conclude that the relations $\gmr_{\alpha}^{\alpha', j, a}$ generate $\FI_{\alpha}^{\gmr, \alpha'}$. In practice, we will indeed use the good moduli space $M_{d',\chi'}$ when possible, for example in the computations of Section \ref{sec: rings}. 
\end{rmk}

Because the construction of the ideal $\FI^{\gmr, \alpha'}_{\alpha}$ requires integration against homology classes of $\FM_{d', \chi'}$, to explicitly compute this ideal we need to understand $H^*(\FM_{d', \chi'})$ first. Thus it makes sense to determine the cohomology ring structure inductively for all $(d, \chi)$, by using the geometric relations with $d'<d$. It turns out that, when working with moduli stacks where $\mathrm{gcd}(d,\chi)>1$, it is also useful to take $d'=d$ and $\chi'$ with a smaller $\gcd$.\footnote{For example, the computation of $H^*(\FM_{3,0})$ in Section \ref{sec: M30} requires $\mathsf{GMR}$ from $M_{3,1}$ and $M_{3,2}$.} This motivates  

\begin{defn}
\label{def: order}
    We say that $(d', n')\prec (d,n)$ if either $d'<d$, or $d'=d$ and $\gcd(d', \chi')<\gcd(d,\chi)$. We define the ideal of generalized Mumford relations to be
    \[
    \FI^\gmr_{d,\chi} := \sum_{\substack{(d',\chi')\prec (d,\chi)\\\textrm{satisfying }\eqref{eq: inequalitygmr}}}\FI^{\gmr,\alpha'}_{\alpha} \subset \FI_{d,\chi}.
    \]
\end{defn}

\begin{rmk}
    \label{rmk: M52_M51}
As pointed out in Section \ref{sec: M51}, the ideal $\FI^\geo_{5,1}$ turns out to be strictly smaller than $\FI_{5,1}$. On the other hand, we have $\FI^\geo_{5,2}=\FI_{5,2}$. Thus from an algorithmic point of view, it could make sense to extend $\prec$ so that $(5,2)\prec (5,1)$ and include $\FI^{\gmr, 5,2}_{5,1}$ in the geometric relations. We have not been able to check if this produces the missing relations in $\FI^\geo_{5,1}$ due to the computational complexity involved. 
%One could speculate, more generally, if there is some total ordering $\prec$ (probably extending our partial ordering) for which we would get $I^\geo_{d,\chi}=I^\Taut_{d,\chi}$.
\end{rmk}

\begin{rmk}
\label{rmk: GMR_11}
    On the moduli space level, the generalized Mumford relations $I_{d,\chi}^{\gmr, 1,\chi'}$ already played a fundamental role in some recent developments understanding the cohomology rings. For example, one key step in Theorem \ref{generation} (i) consists of showing that these relations express all (normalized) tautological classes in terms of the $3d-7$ generators, and Theorem \ref{chi-dep} follows from a careful study of the three relations in Theorem \ref{generation} (ii), which can be produced by  $I_{d,\chi}^{\gmr, 1,\chi'}$. The current paper builds on these progresses and explore the full power of generalized Mumford relations; indeed, they make up the majority of all geometric relations we formulate.  
\end{rmk}

\subsection{Base relations}

\label{sec: br}

The set of relations that we discuss now have a different nature; they use the Hilbert--Chow morphism $h\colon \FM_{d,\chi}\to |d\cdot H|$ in a crucial way. Recall that $|d\cdot H|=\BP^b$ where $b=d(d+3)/2$ and that $\ch_1(H^2)=c_0(2)$ is the pull-back of the hyperplane class via $h$. It follows that $\mathrm{ch}_1(H^2)^{b+1}=0$. More generally, the following is true:

\begin{prop}
\label{prop: base relations}
    For any $i_1, \ldots, i_{b+1}\geq 1$, 
    \[\ch_{i_1}(H^2) \cdots \ch_{i_{b+1}}(H^2)\in \FI_{d,\chi}\]
    is a tautological relation.
\end{prop}

\begin{proof}
By the Grothendieck--Riemann--Roch theorem, the Chern character $\mathrm{ch}(\CF)$ is a cycle pushed forward from $\mathrm{supp}(\CF) \subset \FM_{d,\chi} \times \mathbb{P}^2$. Recall the geometric realization
\[\ch_i(H^2) =  \Big[p_\ast\big(\ch(\CF)\cdot q^\ast H^2)\Big]_{2i}\in H^{2i}(\FM_{d,\chi}).\]
If we pick a point $\pt \in \mathbb{P}^2$ to represent the class $H^2$, then $q^*H^2$ can be identified with $\FM_{d,\chi} \times \{\pt\}$, and the class $\mathrm{ch}_i(H^2)$ is a push-forward from the intersection $\mathrm{supp}(\mathcal{F}) \cap (\FM_{d,\chi} \times \{\pt\})$. In particular, a sheaf $[F]\in [\mathrm{ch}_i(H^2)]\subset \FM_{d,\chi}$ must be supported at the point $\pt$. Now we pick $b+1$ points $\pt_1, \pt_2,\ldots, \pt_{b+1} \in \mathbb{P}^2$ in general position, then there are no degree $d$ curves passing through all of them. It follows that
\[
\ch_{i_1}(H^2) \cdots \ch_{i_{b+1}}(H^2) = 0. \qedhere
\]
\end{proof}

We define the ideal of base relations
\[
\FI^\br_{d,\chi} := \left \langle \ch_{i_1}(H^2) \cdots \ch_{i_{b+1}}(H^2) \mid i_1, \ldots, i_{b+1} \geq 1 \right \rangle\subset \FI_{d,\chi}.
\]
Note that the base relations all have cohomological degrees $\geq 2(b+1)$.

\section{Virasoro representations}
\label{sec: Virasoro rep}

In this section, we introduce the Virasoro operators and explain how they can be used to study the cohomology ring of the moduli stacks and spaces. We explain the Virasoro constraints for moduli spaces $M_{d,\chi}$ with $\gcd(d,\chi)=1$ proven in \cite{lmquivers} and an important consequence: the ideal of tautological relations is preserved by the Virasoro operators, hence inducing a representation of $\Vir_{\geq -1}$ on $H^*(\FM_{d,\chi})$. In fact, such a property is proven in {\it loc. cit.} for any $\FM_{d,\chi}$ with $(d,\chi)$ not necessarily coprime. Second, we refine the preservation of the ideal under the Virasoro operators to the ideal of each type $\mr$/$\gmr$/$\br$/$\geo$. Furthermore, we prove that each ideal is generated by suitable primitive elements under the Virasoro operators. Working with the Virasoro operators and primitive relations is computationally much more effective than working directly with the entire ideals.

\subsection{Virasoro operators and Virasoro constraints}
\label{sec: virasoro operators}

Recall the definition of the descendent algebra $\BD$ from Section \ref{sec: descendent algebra}. The Virasoro operators $\{\bL_n\}_{n\geq -1}$ acting on $\BD$ is defined as a sum $\bL_n=\bR_n+\bT_n$ of the two operators: $\bR_n$ is a derivation operator such that
$$\bR_n(\ch_i(\gamma)):=i(i+1)\cdots (i+n)\ch_{i+n}(\gamma)
$$
and $\bT_n$ is a multiplication operator by the element
$$\bT_n
:=\sum_{a+b=n,\, a,b\geq 0} \sum_i a!b!\,(-1)^{2-d_i^L} \ch_a(\gamma_i^L)\cdot \ch_b(\gamma_i^R),
$$
where $\gamma_i^L$ and $\gamma_i^R$ are the same as in \eqref{eqn: Kunneth}. %with the same notations as in Section \ref{sec: gmr}.
%In the above expression, we used the shorthand notation
%$$(-1)^{2-d^L}\ch_a\ch_b(\td(\BP^2)):=\sum_{i\in I} (-1)^{2-d^L_i}\ch_a(\gamma_i^L)\ch_b(\gamma_i^R)
%$$
%where $\Delta_*\td(\td(\BP^2))=\sum_{i\in I}\gamma_i^L\otimes \gamma_i^R\in H^*(\BP^2)\otimes H^*(\BP^2)$ with $\gamma_i^L\in H^{2d_i^L}(\BP^2)$. 
Note that $\bL_{-1}=\bR_{-1}$ agrees with Definition \ref{def: R -1}. One can check that the operators satisfy the Virasoro bracket relations
$$[\bR_n,\bR_m]=(m-n)\bR_{n+m},\quad [\bL_n,\bL_m]=(m-n)\bL_{n+m}. 
$$
The weight zero Virasoro operator is defined as a combination of the Virasoro operators
$$\bL_{\inv}:=\sum_{n\geq -1}\frac{(-1)^n}{(n+1)!}\bL_{n}\circ \bL_{-1}^{n+1}:\BD\rightarrow \BD_\inv\,.
$$

Virasoro constraints in sheaf theory have been studied in \cite{BLM} building on the previous works \cite{MOOP, M, vanBree}. A proof of the Virasoro constraints for the moduli spaces $M_{d, \chi}$ with $\gcd(d,\chi)=1$ is given in \cite{lmquivers}.
\begin{thm}[{\cite{lmquivers}}]\label{thm: Virasoro constraints} Given $d$ and $\chi$ coprime, we have
    $$\int_{M_{d,\chi}}\bL_\inv (D)=0\quad \text{for all}\;\; D\in \BD\,.$$
\end{thm}
For concrete calculations, it is useful to record here the normalized formulation of the Virasoro constraints (cf. \cite[Conjecture 2.15]{BLM}) in terms of the generators $c_k(j)$. It has been explained in Section \ref{sec: normalized class} how the descendent algebra $\BD_\alpha$ is identified with the polynomial algebra $\BQ[c_0(2), c_1(1), c_2(0), c_1(2),\cdots]$. In particular, the operators $\bR_n$ and $\bT_n$ can be interpreted as operators acting on $\BQ[c_0(2), c_1(1), c_2(0), c_1(2),\cdots]$, and we have
\[\bR_n(c_k(j))=(k+j-1)\cdots(k+j-1+n)c_{k+n}(j)\,.\]
As discussed in Section \ref{sec: normalized class}, there is a normalized realization \[\BD_\alpha\simeq \BQ[c_0(2), c_1(1), c_2(0), c_1(2),\cdots]\to H^*(M_{d,\chi})\] that sends $c_1(1)$ to zero. The normalized Virasoro operators $\bL_n^\delta$ on $\BQ[c_0(2), c_1(1), c_2(0), c_1(2),\cdots]$  are given by
\[\bL_n^\delta=\bR_n+\bT_n-\frac{(n+1)!}{d}\bR_{-1}\circ c_{n+1}(1)\,.\]

\begin{cor}\label{cor: virasoro normalized}
   Given $d$ and $\chi$ coprime, we have
    $$\int_{M_{d,\chi}}\bL_n^\delta (D)=0\quad \text{for all}\quad n\geq 0,\;\, D\in \BQ[c_0(2), c_1(1), c_2(0), c_1(2),\cdots]\,.$$ 
\end{cor}
\begin{proof}
This follows from the standard form of the Virasoro constraints in Theorem \ref{thm: Virasoro constraints} and \cite[Proposition 2.16]{BLM}. Note that the definition of the realization of $c_k(j)$ also involves the twisting operators $\bF_\rho$ from Section \ref{sec: normalized class} and that these commute with the Virasoro operators by \cite[Lemma 2.19]{BLM}.
\end{proof}

\subsection{Virasoro representation}

Recall from the last section that we have a representation of the half of the Virasoro algebra $\Vir_{\geq -1}$ on $\BD$. On the other hand, we have an exact sequence 
$$0\rightarrow \FI_{d,\chi}\rightarrow \BD_{d,\chi}\rightarrow H^*(\FM_{d,\chi})\rightarrow 0\,. 
$$
It turns out that as a consequence of Virasoro constraints, one can show that the representation of $\Vir_{\geq -1}$ descends to the cohomology of the stack.

\begin{thm}[{\cite{lmquivers}}]\label{thm: repstack}
The representation of $\Vir_{\geq -1}$ on $\BD$ preserves the ideal of relations $\FI_{d,\chi}$. Hence, the operators $\bR_n$ descend to the cohomology $H^*(\FM_{d,\chi})$, i.e. there is a dashed arrow completing the diagram
\begin{center}
        \begin{tikzcd}\label{diagram: Rngeometric}
            \BD\arrow[r, "\bR_{n}"] \arrow[d, "\xi"]& \BD\arrow[d, "\xi"]\\
            H^*(\FM_{d, \chi})\arrow[r, dashed]&H^*(\FM_{d, \chi}).
        \end{tikzcd}
    \end{center}
\end{thm}
\begin{rmk}
    Note that $\Vir_{\geq -1}$ acts on $\BD$ via either $\{\bR_n\}_{n\geq -1}$ or $\{\bL_n\}_{n\geq -1}$. Since the multiplication operator $\bT_n$ trivially preserves the ideal $\FI_{d,\chi}$, the above theorem is true for either choice of a representation. In practice, however, we always work with the simpler one using $\{\bR_n\}_{n\geq -1}$. 
\end{rmk}

\begin{rmk}
The most natural way to prove a statement like Theorem \ref{thm: repstack} is to utilize the approximation of $H^*(\FM_{d,\chi})$ by the cohomology of moduli spaces of Joyce--Song pairs $P^N_{d,\chi}$, see the proof of Theorem \ref{thm: tautological generation}. If we knew the pair Virasoro constraints (cf. \cite[Conjecture 2.18]{BLM}) for Joyce--Song pairs, we could deduce that $\bR_n$ descends to the cohomology $H^*(P^N_{d,\chi})$, and hence to the cohomology of the stack by approximation. The Virasoro constraints for $P^N_{d,\chi}$ are equivalent to the Virasoro constraints for the Joyce invariant classes $[M_{d,\chi}]^\inva$ for every $(d, \chi)$, but currently we only know these unconditionally when $\gcd(d,\chi)=1$. This problem is worked around in \cite{lmquivers} by identifying $\FM_{d,\chi}$ with a stack of quiver representations and approximating it by framed quiver representations, which are the analogue of Joyce--Song pairs in the quiver context.
\end{rmk}

We can also make a statement for the cohomology of the good moduli spaces by using the normalized realization. We have an exact sequence 
\[0\to I_{d, \chi}\to \BQ[c_0(2), c_2(0), c_1(2),\cdots]\to H^*(M_{d,\chi})\to 0\,.\]
Define the derivation $\bR_n^\delta$ on $ \BQ[c_0(2), c_1(1), c_2(0), c_1(2),\cdots]$ by
\[\bR_n^\delta=\bR_n-\frac{(n+1)!}{d}c_{n+1}(1)\circ \bR_{-1}\,.\]
This derivation satisfies the property that $\bR_n^\delta(c_1(1))=0$ for every $n\geq 0$, so it induces an operator on $ \BQ[c_0(2), c_2(0), c_1(2),\cdots]$ which we denote in the same way. Note that $\bL_n^\delta=\bR_n^\delta+\bT_n^\delta$ where $\bT_n^\delta=\bT_n-\frac{(n+1)!}{d}c_n(1)$ is linear.

\smallskip

The following proposition is immediate: %\m{If we assume that $\bR_{-1}$ preserves $\FI$ (which is essentially automatic for any ideal of geometric relations) then $\bR_n^\delta$ preserving $I$ is equivalent to $\bR_n$ preserving $\FI$. By the way, I think that there is a 1-1 correspondence between ideals $I$ and ideals $\FI$ closed under $\bR_{-1}$ ($\FI$ is recovered as the smallest ideal containing $I\subseteq \BD_{\inv}\subseteq \BD$ closed under $\bR_{-1}$.}

\begin{prop}\label{prop: idealimagepreserved}
    Let $\FI\subseteq  \BD_\alpha \simeq \BQ[c_0(2), c_1(1), c_2(0), c_1(2),\cdots]$ be an ideal closed under the action of the operators $\bR_n$, $n\geq -1$. Then the ideal $I\subseteq \BQ[c_0(2), c_2(0), c_1(2),\cdots]$ obtained as the image of $\FI$ by the quotient setting $c_1(1)$ to 0 is closed under the operators $\bR_n^\delta$, $n\geq 0$. 
\end{prop}
\begin{proof}
    Since $\bR_n$ and $\bR_{-1}$ preserve the ideal $\FI$, the operator $\bR_n^\delta$ also preserves the ideal $\FI$, so the induced operator $\bR_n^\delta$ on the quotient $\BQ[c_0(2), c_2(0), c_1(2),\cdots]$ preserves the ideal $I$. 
\end{proof}

\begin{cor}
Let $d,\chi$ be coprime and $n\geq 0$. The derivation $\bR_n^\delta$ preserves the ideal $I_{d,\chi}$.
\end{cor}
\begin{proof}
    By Proposition \ref{prop: idealimagepreserved} above, this is a consequence of the stacky analogue (Theorem \ref{thm: repstack}) and Lemma \ref{lem: idealcoarse}. However, it is instructive to give a proof directly from the normalized form of the Virasoro constraints, which we do below.

    Suppose that $D\in I_{d,\chi}$ and let $E\in \BQ[c_0(2), c_2(0), c_1(2),\cdots]$ be arbitrary. By the Virasoro constraints in the form of Corollary \ref{cor: virasoro normalized},
    \[0=\int_{M_{d,\chi}} \bL_n^\delta(D\cdot E)=\int_{M_{d,\chi}} \bR_n^\delta(D)\cdot E+\int_{M_{d,\chi}} D\cdot \bL_n^\delta(E)=\int_{M_{d,\chi}} \bR_n^\delta(D)\cdot E\,.\]
    Since $E$ was arbitrary and the realization map $\BQ[c_0(2), c_2(0), c_1(2),\cdots]\to H^*(M_{d,\chi})$ is surjective, we conclude by Poincaré duality that $\bR_n^\delta(D)\in I_{d,\chi}$.
\end{proof}

\subsection{Preservation of geometric relations}
In this subsection, we study preservation of the ideal $\FI^\bullet_{d,\chi}$ under the $\Vir_{\geq -1}$-representation for each $\bullet=\mr, \gmr, \br, \geo$. The $\br$ case is trivial, and the $\geo$ case follows from the other three cases by definition. So we study the $\mr$ and $\gmr$ cases in the next two sections. 

\subsubsection{Mumford relations}

We prove here the following proposition:
\begin{prop}\label{prop: virasoroMR}
    The ideal $\FI_{d, \chi}^{\mr}$ is closed under the action of  $U(\Vir_{\geq -1})$. More is true: for each choice of sign $\pm$ and $k\in \BZ$, the ideal generated by $\mr_{d,\chi}^{\pm, k,j}$ for $j>\pm \chi+kd$ is closed under the action of $U(\Vir_{\geq -1})$.
\end{prop}
\begin{proof}%\wocomment{We do not explain what we mean by interact well here. I also realized that we do not introduce two symmetries making the statement of interacting well inherently vague.}
The Virasoro operators interact well with respect to the isomorphisms $\FI_{d,\chi}\cong \FI_{d,\pm \chi+kd}$, see Proposition \ref{prop: c_kj_property} (ii), so it is enough to treat the case $\pm=+$ and $k=0$. To simplify notation, we write $A_j=\mr_{d,\chi}^{+, 0, j}$. Recall that $A_j$ is the degree $2j$ term of 
\[A=\exp\left(\sum_{\ell\geq 1}(-1)^{\ell-1}(\ell-1)!\ch_\ell(\td(\BP^2))\right)\,.\]
To compute $\bR_n(A_j)$, we apply $\bR_n$ to the exponential:
\[\bR_n(A)=\left(\sum_{\ell\geq 1}(-1)^{\ell-1}(\ell+n)!\ch_{\ell+n}(\td(\BP^2))\right)A\,\]
hence
\[\bR_n(A_j)=\sum_{\ell=1}^j (-1)^{\ell-1}(\ell+n)!\ch_{\ell+n}(\td(\BP^2))A_{j-\ell}\,.\]
On the other hand, we also have
\[(j+n)A_{j+n}=\bR_0(A_{j+n})=\sum_{\ell=1}^{j+n} (-1)^{\ell-1}\ell!\ch_{\ell}(\td(\BP^2))A_{j+n-\ell}\,.\]
Assume first that $n\geq 0$. Comparing the two equations we find that
\begin{equation}\label{eq: Rnmr}\bR_n(A_j)=(-1)^n(j+n)A_{j+n}+\sum_{\ell=1}^{n}(-1)^{n+\ell} \ell! \ch_\ell(\td(\BP^2))A_{j+n-\ell}\,.\end{equation}
If $j>\chi$, then $j+n-\ell\geq j>\chi$, so clearly the right hand side is also in $\FI_{d, \chi}^{\mr}$. When $n=-1$, we obtain
\[\bR_{-1}(A_j)=\big(\ch_0(\td(\BP^2))-j+1\big)A_{j-1}=\big(\chi-j+1\big)A_{j-1}\,.\]
If $j>\chi+1$, then $j-1>\chi$ and $A_{j-1}\in \FI_{d, \chi}^{\mr}$; when $j=\chi+1$ the factor $\chi-j+1$ vanishes, so $\bR_{-1}(A_{\chi+1})=0$ is also in the ideal. 
\end{proof}

By the definition of $I_{d,\chi}^\mr$ and Proposition \ref{prop: idealimagepreserved} the corollary below follows:

\begin{cor}
    The ideal $I_{d,\chi}^\mr$ is closed under the action of the operators $\bR_n^\delta$ for $n\geq 0$.
\end{cor}

\subsubsection{Generalized Mumford relations}

We proceed now to prove that the generalized Mumford relations are also preserved by the action of the Virasoro operators.

\begin{thm}\label{thm: virasoroGMR}
    Let $\alpha=(d, \chi)$, $\alpha'=(d', \chi')$ satisfy \eqref{eq: inequalitygmr}. Then the ideal $\FI_{\alpha}^{\gmr, \alpha'}$ is closed under the action of $U(\Vir_{\geq -1})$. 
\end{thm}
\begin{proof}
Recall that the ideal $\FI_{\alpha}^{\gmr, \alpha'}$ is generated by relations of the form $\gmr^{\alpha', j, A}_\alpha$ where $j>dd'$ and $A\in H_*(\FM_{\alpha'})$. They are obtained by starting with the formal class $C_j\in \BD_\alpha\otimes \BD_{\alpha'}$, realizing the second factor $\BD_{\alpha'}\to H^*(\FM_\alpha')$ and then integrating against $A$. The crucial ingredient in the proof is Theorem \ref{thm: quadraticidentity}, which we prove in the appendix. This is a combinatorial identity which interacts $C$ with the action of the Virasoro operators; it states that we have an equality
\begin{align}\label{eq: quadraticinproof}    
(\bR_n\otimes \id)(C_j)&+\sum_{k=-1}^n \binom{n+1}{k+1}(\id\otimes \bR_{k})(C_{j+n-k})\\
&=\sum_{0\leq a+b\leq n}\sum_i (-1)^{d_i^L+1}\frac{a!(n-a)!}{(n-a-b)!}\big(\ch_a(\gamma_i^L)\otimes \ch_b(\gamma_i^R)\big)C_{j+n-a-b}\nonumber
\end{align}
in $\BD_\alpha\otimes \BD_{\alpha'}$. Since the operators $\bR_k$ descend to $H^*(\FM_{\alpha'})$ by Theorem \ref{thm: repstack}, the equality \eqref{eq: quadraticinproof} can be also viewed in $\BD_\alpha\otimes H^*(\FM_{\alpha'})$. Since $H_*(\FM_{\alpha'})$ is dual to $H^*(\FM_{\alpha'})$, we have dual operators $\bR_n^\dagger\colon H_*(\FM_{\alpha'})\to H_*(\FM_{\alpha'})$ on homology. Applying $\id\otimes \int_A$ to \eqref{eq: quadraticinproof} we obtain 
\begin{align}\label{eq: quadraticinproof2}\bR_n(\gmr^{\alpha', j, A}_\alpha)&+\sum_{k=-1}^n \binom{n+1}{k+1} \gmr^{\alpha', j+n-k, \bR_k^\dagger(A)}_\alpha=\\
&\sum_{0\leq a+b\leq n}\sum_i (-1)^{d_i^L+1}\frac{a!(n-a)!}{(n-a-b)!}\ch_a(\gamma_i^L)\gmr_{\alpha}^{\alpha'\!,\, j+n-a-b,\, \ch_b(\gamma_i^R)\cap A}\,.\nonumber
\end{align}
Note that for $k\leq n$ and $a+b\leq n$ we have $j+n-k,\, j+n-a-b\geq j>dd'$, so we conclude that $\bR_n(\gmr^{\alpha', j, A}_\alpha)\in \FI_{\alpha}^{\gmr, \alpha'}$.\qedhere
\end{proof}

Once again, a similar statement can be made at the level of the good moduli space:

\begin{cor}
    The ideal $I_{\alpha}^{\gmr, \alpha'}$ is closed under the action of the operators $\bR_n^\delta$ for $n\geq 0$.
\end{cor}

\subsection{Primitive relations}

Knowing that $\FI^\bullet_\alpha$ is closed under the action of $\Vir_{\geq -1}$ can provide a strong simplification on the computations. We now proceed to present and explain a considerably smaller set of generators of the ideals $\FI_\alpha^\bullet$ when we regard them as a $\BD_\alpha\otimes U(\Vir_{\geq-1})$-module, rather than just as an ideal in $\BD_\alpha$. Here, the algebra structure of $\BD_\alpha\otimes U(\Vir_{\geq-1})$ is induced from that of $\BD_\alpha$ and $U(\Vir_{\geq-1})$ together with the commutation rule $[\bR_k,D]=\bR_k(D)$. 

\begin{defn}
\label{def: primitive relations}
    We define the set $\mathfrak{P}^\bullet_\alpha$ of primitive relations of three types as follows:
\begin{enumerate}
    \item[(i)] $\mathfrak{P}_{\alpha}^\mr$ is the set of relations $\mr^{\pm, k, j}_{\alpha}$ with $\pm \chi+kd\geq g$ and $j=\pm \chi+kd+1$.
    
    \item[(ii)] $\mathfrak{P}_{\alpha}^{\gmr, \alpha'}$ is the set of relations $\gmr^{\alpha', j, A}_\alpha$ with either
    \[j=dd'+1\textup{ or }\big(j=dd'+2\textup{ and } \deg(A)=dd'+2\big)\,.\]
    \item[(iii)] $\mathfrak{P}_{\alpha}^\br$ consists of the single relation $\ch_1(H^2)^{b+1}$.
\end{enumerate}
\end{defn}

Clearly each of these sets are contained in the corresponding ideals including non-primitive relations. For example, $\FI_{\alpha}^\mr$ is generated by all relations $\mr^{\pm, k, j}_{\alpha}$ with $j\geq \pm \chi+kd+1$, but $\mathfrak{P}_{\alpha}^\mr$ consists only of the first Chern class beyond the rank. It turns out that the primitive relations are enough to generate all the relations once we know that the ideals are closed under the Virasoro action.

\begin{thm}
\label{thm: primitive relatiosn generate}
The ideals $\FI_\alpha^{\mr}, \FI_\alpha^{\gmr, \alpha'}, \FI_\alpha^{\br}$ are the smallest ideals containing $\mathfrak{P}_\alpha^{\mr}$, $\mathfrak{P}_\alpha^{\gmr, \alpha'}$, $\mathfrak{P}_\alpha^{\br}$, respectively, that are closed under the action of $U(\Vir_{\geq -1})$.
\end{thm}
\begin{proof}
We have already shown that the ideals $\FI_\alpha^{\mr}$, $\FI_\alpha^{\gmr, \alpha'}$, $\FI_\alpha^{\br}$ are closed under the action of $U(\Vir_{\geq -1})$ and they clearly contain the primitive relations, so it is enough to show that they are the smallest such ones.

We start with the ideal of Mumford relations. As we did in the proof of Proposition \ref{prop: virasoroMR}, it is enough to consider $k=0$ and $\pm=+$. Considering equation \eqref{eq: Rnmr} for $n=1$, it is clear that if $A_j$ is in some ideal closed under $\bR_{1}$, then $A_{j+1}$ is also in such ideal. Since $A_{\chi+1}$ is by definition in $\mathfrak{P}_\alpha^{\mr}$, it follows by induction on $j$ that $A_j$ is contained in the smallest ideal containing $\mathfrak{P}_\alpha^{\mr}$ and closed under the $U(\Vir_{\geq -1})$-action for every $j>\chi$.

Let us consider the case of generalized Mumford relations. First, we observe that we can modify \eqref{eq: quadraticinproof} by using Lemma \ref{lem: R-1quadraticidentity} to remove the $k=-1$ term, which gives
\begin{equation}\label{eq: quadraticinproof3}\bR_n(\gmr^{\alpha', j, A}_\alpha)+\sum_{k=0}^n \binom{n+1}{k+1} \gmr^{\alpha', j+n-k, \bR_k^\dagger(A)}_\alpha=(j+n)\gmr^{\alpha', j+n, A}_\alpha+\cdots\,
\end{equation}
where $\cdots$ denotes the remaining terms on the right hand side of \eqref{eq: quadraticinproof2} with $0<a+b\leq n$; note that all these terms are in the ideal generated by relations of the form $\gmr^{\alpha', i, A'}_\alpha$ with $j\leq i<j+n$.
The operator $\bR_{0}$ is given by $\bR_0(D)=\frac{1}{2}\deg(D)D$ where $\deg(D)\in \BZ$ is the cohomological degree, so $\bR_0^\dagger(A)=\frac{1}{2}\deg(A)A$ where $\deg(A)$ is the homological degree. Considering this, the ``leading term'' $\gmr_\alpha^{\alpha', j+n, A}$ appears twice in \eqref{eq: quadraticinproof3}: in the $k=0$ term of the sum, and in the right hand side. If
\begin{equation}
    \label{eq: leadingtermcoeff}\frac{1}{2}(n+1)\deg(A)\neq j+n
\end{equation}
then we can express $\gmr_\alpha^{\alpha', j+n, A}$ in terms of relations $\gmr_\alpha^{\alpha', j', A'}$ with $j\leq j'<j+n$. We now argue that every relation of the form $\gmr^{\alpha', i, A}_\alpha$ is contained in any ideal containing the primitive relations and closed under the Virasoro action. We proceed by induction on $i$. If $i=dd'+1$ then the relation $\gmr^{\alpha', i, A}_\alpha$ is primitive, so there is nothing to prove. If $i=dd'+2$ then either $\deg(A)=dd'+2$, in which case $\gmr^{\alpha', i, A}_\alpha$ is primitive by definition, or $(j, n)=(i-1,1)$ satisfies the inequality \eqref{eq: leadingtermcoeff}, in which case we conclude by induction. If $i\geq dd'+3$ then the inequality \eqref{eq: leadingtermcoeff} holds for either $(j,n)=(i-1,1)$ or $(j, n)=(i-2,2)$, so we conclude again by induction.

For base relations, we show by induction on $j\geq 0$ that any ideal containing $\ch_1(H^2)^{b+1}$ closed under the Virasoro action contains every element of the form
\begin{equation}\label{eq: brinduction}
\ch_{i_1}(H^2)\cdots\ch_{i_j}(H^2)\ch_1(H^2)^{b+1-j}
\end{equation}
for any $i_1, \ldots, i_j\geq 2$. Note that if we apply $\bR_{i_{j+1}-1}$ to a term of the form \eqref{eq: brinduction} we obtain
\[(b+1-j)i_{j+1}!\ch_{i_1}(H^2)\cdots\ch_{i_j}(H^2)\ch_{i_{j+1}}(H^2)\ch_1(H^2)^{b-j}+\Big(\textup{terms of the form \eqref{eq: brinduction}}\Big)\,.\]
This shows that the claim for $j$ implies the same claim for $j+1$, and concludes the proof.\qedhere
\end{proof}

\begin{rmk}
    The primitive generalized Mumford relations with $j=dd'+2$ are indeed necessary. For example, take $(d,\chi) = (2,1)$ and $(d',\chi') = (1,0)$, the two (primitive) relations in degrees $1$ and $2$ given by $j=dd'+1 = 3$ are
    \begin{equation}\label{pGMR_example}
    c_0(2)-2c_2(0), \quad 2c_2(0)c_1(1) - 2c_2(0)c_0(2) - c_1(1)c_0(2) + c_0(2)^2 + 2c_1(2) - 4c_3(0).
    \end{equation}
    Applying $\mathsf{L}_1$ to $c_0(2)-2c_2(0)$, we obtain a second relation $2c_1(2)-4c_3(0)$ in degree 2. Note that it lies in the ideal generated by \eqref{pGMR_example}. On the other hand, the degree 2 relation obtained by taking $j=4, \,\overline A=[M_{1,0}]$ is not in this ideal (by direct computation). Thus we do need to include the latter into the primitive relations.

\end{rmk}

\section{BPS integrality}
\label{sec: BPS_main}

In this section, we recall cohomological Hall algebras (CoHA in short) and BPS integrality of Mozgovoy--Reineke \cite{MozRei}. By establishing tautological generation and purity of the cohomology of the moduli stacks, we relate the Poincaré series of the moduli stacks to the intersection Poincaré polynomials of the good moduli spaces and vice versa. As applications, we prove structural formulas for the Poincar\'e series of the moduli stacks and free generation of cohomology rings in low degrees.

\subsection{Cohomological Hall algebra}
Denote the monoid of topological types of one-dimensional sheaves by
$$\Lambda:=\{(0,0)\}\sqcup \,\BZ_{\geq 1}\times \BZ. 
$$
For a fixed slope $\mu\in \BQ$, we define the submonoid
$$\Lambda_\mu:=\{(d,\chi)\in \Lambda\,|\,\chi=d\cdot \mu\}.
$$
Denote the set of nonzero elements by $\Lambda_\mu^{\times}:=\Lambda_{\mu}-\{(0,0)\}$. Recall that semistable one-dimensional sheaves of slope $\mu$ form an abelian category $\mathcal{A}_\mu$. This is a hereditary category in the sense that $\Ext^2(F_1,F_2)=0$ for any $F_1,F_2\in \mathcal{A}_\mu$ by semistability and Serre duality. Denote the corresponding smooth moduli stack by
$$\FM_\mu:=\bigsqcup_{(d,\chi)\in \Lambda_\mu}\FM_{d,\chi} 
$$
where $\FM_{0,0}$ is simply a point. Consider the stack of short exact sequences in $\mathcal{A}_\mu$ and forgetful morphisms 
\begin{equation}\label{eqn: SES diagram}
    \begin{tikzcd}
  & \text{SES}_\mu \arrow[ld, "\pi_1\times\pi_3"'] \arrow[rd, "\pi_2"] &   \\
\FM_\mu\times \FM_\mu &                                      & \FM_\mu. 
\end{tikzcd}
\end{equation}
Precisely, the stack $\text{SES}_\mu$ parametrizes short exact sequences $0\rightarrow F_1\rightarrow F_2\rightarrow F_3\rightarrow 0$ where all $F_i$'s are semistable one-dimensional sheaves of slope $\mu$ and the morphisms $\pi_1\times\pi_3$ and $\pi_2$ send this to $(F_1,F_3)$ and $F_2$, respectively. Since $\mathcal{A}_\mu$ is hereditary, we can use diagram \eqref{eqn: SES diagram} to define the cohomological Hall algebra multiplication
$$\star:=(\pi_2)_*\circ (\pi_1\times\pi_3)^*:H^*(\FM_\mu)\otimes H^*(\FM_\mu)\rightarrow H^*(\FM_\mu). 
$$
We remark that the push-forward $(\pi_2)_*$ is defined because $\pi_2$ is a proper morphism between smooth stacks. This indeed defines an associative algebra with a unit $1\in \BQ=H^*(\FM_{0,0})\subset H^*(\FM_\mu)$. 

\begin{defn}
    We call $(H^*(\FM_\mu),\star, 1)$ the slope $\mu$ semistable cohomological Hall algebra. 
\end{defn}

\subsection{BPS integrality}\label{sec: BPS integrality}

One of the most important aspects of CoHA's in various settings is the so-called BPS integrality. Roughly speaking, BPS integrality relates the cohomology of the moduli stacks and intersection cohomology of the moduli spaces in the setting of hereditary categories. The precise meaning of BPS integrality varies with the context. We use the BPS integrality theorem for hereditary categories, following Mozgovoy--Reineke. 
\begin{thm}[\cite{MozRei}]\label{thm: BPS}
    We have an equality
    \begin{equation}\label{eqn: BPS integrality in K theory}
        H^*(\FM_\mu,\BQ^\vir)=\Sym\left(
H^*({B}\BG_m,\BQ^\vir)\otimes \bigoplus_{(d,\chi)\in \Lambda_\mu^\times}H^*(M_{d,\chi},\IC^\vir)
\right)
    \end{equation}
in $\widehat{K}(\MHM(\Lambda_\mu))$, the $\Lambda_{\mu}$-graded localized $K$-theory of mixed Hodge modules.\footnote{The same equality holds in the derived category of monodromic mixed Hodge modules $D^+(\MMHM(\Lambda_\mu))$ if we consider semistable CoHA's for quivers without relations, see \cite{Davison-Meinhardt}. Equality in $D^+(\MMHM(\Lambda_\mu))$ for global hereditary categories as in our case seems to be unknown in the literature, even though this is expected to hold. }
\end{thm}

In the rest of the subsection, we explain the terminology in this theorem. The main language we use is Saito's theory of mixed Hodge modules, see \cite{Saito} for precise definitions. Let $X$ be a separated reduced scheme of finite type. Then there is an abelian category of mixed Hodge modules $\MHM(X)$ and its bounded derived category $D^b(\MHM(X))$. When $X = \pt$ is a point, $\MHM(\pt)$ recovers the category of polarizable mixed Hodge structures with rational coefficients. Important properties of mixed Hodge modules include:
\begin{enumerate}
    \item There is a faithful functor $\mathbf{rat}:D^b(\MHM(X))\rightarrow D^b_{\mathrm{cst}}(X)$ to the bounded derived category of constructible sheaves.
    \item The functor $\mathbf{rat}$ sends $\MHM(X)$ to the category $\mathrm{Perv}(X)$ of perverse sheaves. 
    \item There exists a six functor formalism with Verdier duality $\mathbf{D}$ for $D^b(\MHM(-))$.
    \item Every object $M\in \MHM(X)$ admits a weight filtration $W_\bullet M$.
\end{enumerate}

We say that $M\in D^b(\MHM(X))$ is pure of weight $w$ if $\mathrm{gr}^{W}_i\CH^j(M)=0$ unless $i=j+w$. We simply say that $M$ is pure if it is pure of weight $0$. The shift functor $[d]$ changes the weight of a complex of mixed Hodge modules by $d$. The weight filtration behaves nicely with respect to the six functors which in particular implies the following. 
\begin{enumerate}
    \item [(5)] (Derived) push-forward along proper morphisms preserves pure complexes. 
\end{enumerate}

The following fact, called semisimplicity, is perhaps the deepest result about mixed Hodge modules, which together with (5) recovers the decomposition theorem \cite{BBDG}. 
\begin{enumerate}
    \item [(6)] If $M\in D^b(\MHM(X))$ is pure, then there is a (non-canonical) isomorphism 
    \[
    M\simeq \bigoplus_{j \in \mathbb{Z}} \CH^j(M)[-j].
    \]
    
    %\footnote{The isomorphism is noncanonical. }
\end{enumerate}

Suppose moreover that $X$ is irreducible of dimension $d_X$. The construction of the intersection complex as a perverse sheaf naturally lifts to mixed Hodge modules. We also call such a lift the intersection complex and denote it by $\IC_X\in \MHM(X)$. Consider the elements
$$\BT\coloneqq \BQ(-1)\in \MHM(\pt),\quad \BL\coloneqq H^*_c(\BA^1) = \BQ(-1)[-2]\in D^b(\MHM(\pt))
$$
which are used for the Tate twist and Lefschetz twist, respectively. %\footnote{Note that $\BL=H^*_c(\BA^1)$.}
We use the same notations for the pull-backs of $\BT$ and $\BL$ via $X\rightarrow \pt$. When the dimension $d_X$ is even, the virtual intersection complex is defined by 
$$\IC_X^\vir:=\BT^{-d_X/2}\otimes \IC_X\in \MHM(X).
$$
This complex is pure of weight $0$, so it is preserved by the Verdier duality functor $\mathbf{D}$. If $X$ is smooth, the virtual intersection complex becomes
$$\IC_X^\vir=\BQ_X(d_X/2)[d_X]=\BL^{-d_X/2}. 
$$
which we simply denote by $\BQ^\vir_X$. When $d_X$ is odd, we take the above definition valued in $K(\MHM(X))[\BL^{1/2}]$ where the square root $\BL^{1/2}$, or equivalently $\BT^{1/2}$, is formally added.\footnote{An alternative approach to this is to use the derived category of monodromic mixed Hodge modules where the square root of the Tate twist exists, see \cite[Section 2.1.6]{Dav}.}

\medskip

We now explain how the theory of mixed Hodge modules extends to stacks, as required by Theorem \ref{thm: BPS}. Since the abelian category of mixed Hodge modules forms a sheaf with respect to the smooth topology \cite[Theorem 2.3]{Achar}, we can define $\MHM(\FX)$ for a locally finite Artin stack $\FX$. The six functor formalism, however, is not defined in full generality, hence requiring more care. This can be worked around for a certain class of Artin stacks, including global quotient stacks by reductive groups, see \cite{Dav2} for details. According to such a definition, for example, we have
$$H^*( B\BG_m,\BQ)=1\oplus \BL\oplus \BL^2\oplus\cdots \in D^+(\MHM(\pt))
$$
where $D^+(-)$ denotes the derived category bounded from below. The virtual cohomology of $B\BG_m$ defines an element 
$$H^*(B\BG_m,\BQ^\vir)=\frac{\BL^{1/2}}{1-\BL}\in \widehat{K}(\MHM(\pt))
$$
where $\widehat{K}(\MHM(\pt))$ denotes the localized $K$-theory of mixed Hodge modules over a point defined by
\begin{equation}\label{eqn: localized K theory}
    \widehat{K}(\MHM(\pt))\coloneqq K(\MHM(\pt))\otimes_{\BZ[\BL]}\BQ[\,\BL^{1/2},(1-\BL^n)^{-1}\,|\,n\geq 1].
\end{equation}

Finally, we explain the $\Lambda_\mu$-grading and symmetric product $\Sym(-)$ with respect to a certain symmetric monoidal structure. Recall that $(\Lambda_\mu,+)$ is a commutative monoid that is abstractly isomorphic to $\BZ_{\geq 0}$. We regard $\Lambda_\mu$ as a commutative monoid object in the category of schemes by considering it as a disjoint union of points, denoted by $\pt_{\alpha}$ for each $\alpha=(d,\chi)\in \Lambda_\mu$. Define the bounded below derived category of $\Lambda_\mu$-graded mixed Hodge modules as
$$D^+(\MHM(\Lambda_\mu)):=\prod_{\alpha\in \Lambda_\mu}D^+(\MHM(\pt_\alpha)). 
$$
This is equipped with a monoidal structure given by 
$$D^+(\MHM(\pt_{\alpha_1}))\times D^+(\MHM(\pt_{\alpha_2}))\rightarrow D^+(\MHM(\pt_{\alpha_1+\alpha_2})),\quad (M_1,M_2)\mapsto +_*(M_1\boxtimes M_2)
$$
where $+:\pt_{\alpha_1}\times \pt_{\alpha_2}\rightarrow \pt_{\alpha_1+\alpha_2}$. Moreover, this can be upgraded to a symmetric monoidal structure by \cite{MSS}. This induces a $\lambda$-ring structure and a symmetric product $\Sym(-)$ on $\widehat{K}(\MHM(\Lambda_\mu))$. 

\begin{rmk}
    The way we state Theorem \ref{thm: BPS} is not exactly the same as in the reference \cite{MozRei}. Our formulation follows easily by applying the Verdier duality $\mathbf{D}$ to Equation (12) combined with Theorem 5.4 of \textit{loc. cit.} Note that we need to use that the Verdier duality preserves the virtual intersection complexes $\BQ^\vir_{\FM_{d,\chi}}$ and $\IC^\vir_{M_{d,\chi}}$. The authors illustrated their theory in the case of semistable sheaves on curves, but it was explained that the same result applies to any hereditary categories with an exact framing functor such that the Euler form is symmetric, see \cite[Remark 3.12]{MozRei}. These assumptions are satisfied for the abelian category $\CA_\mu$ of semistable one-dimensional sheaves of slope $\mu$ on $\BP^2$.\footnote{The Euler form is symmetric for a trivial reason. We also have a left exact framing functor $\Hom(\CO(-N),-):\CA_\mu\rightarrow \textup{Vect}$. For any given $(d,\chi)\in \Lambda_\mu$, we can choose $N$ sufficiently large so that this functor is exact up to degree $d$, and the proof of \cite{MozRei} can be adapted to this setting. Alternatively, we can use the exact framing functor $\Ext^1(F',-):\CA_\mu\rightarrow \textup{Vect}$ where $F'$ is a fixed semistable sheaf such that $\mu<\mu'<\mu+3$.}
    
 %   \wo{I moved it to footnote and also explained this issue of left exact functor. I would like to say something weaker than ``proof of loc. cit. can be generalized to this setting" but don't know what other phrase I can use...}
\end{rmk}

\subsection{Tautological generation}\label{sec: tautological generation}

In this section, we prove the following properties of the cohomology of the moduli stacks, including tautological generation, cf. Theorem \ref{thm: tautgenerated} (ii). 

\begin{thm}\label{thm: tautological generation}
The cohomology ring $H^*(\FM_{d,\chi},\BQ)$ is tautologically generated and pure. Furthermore, the cycle class map from Chow to cohomology is an isomorphism. 
\end{thm}

To motivate our proof, we explain some of its main ingredients in the analogous statement for the classifying stack $B\GL_n\coloneqq [\pt/\GL_n]$. Denote by $\CV\rightarrow B\GL_n$ the universal rank $n$ vector bundle. For each $N\geq n$, we have a diagram 
\begin{equation}\label{eqn: diagram for BGL}
    \begin{tikzcd}
\HHom_n^N\coloneqq \HHom(\BC^N\otimes \CO_{B\GL_n},\CV) \arrow[d, "\pi"] &  \Gr(N,n)\arrow[l, "j"'] \\
B\GL_n             &                  
\end{tikzcd}
\end{equation}
where $\HHom_n^N$ is the total space of the vector bundle over $B\GL_n$ and $\Gr(N,n)$ is the Grassmannian. The stack $\HHom_n^N$ can be interpreted as a stack parametrizing all morphisms $\BC^N\rightarrow V$ with $\dim(V)=n$. Then  $j\colon \Gr(N,n)\rightarrow\HHom_n^N$ is an open embedding that corresponds to surjective morphisms $\BC^N\twoheadrightarrow V$. The complement of $j$ has codimension at most 
$$Nn-(N(n-1)+n-1)=N-n+1
$$
since it parametrizes morphisms $\BC^N\rightarrow V$ that factor through some $V'\subsetneq V$. The geometry of the above diagram thus implies
$$\mathrm{CH}^i(B\GL_n)
\xlongrightarrow[\,\raisebox{0.2 em}{\smash{\ensuremath{\sim}}}\,]{\pi^*}\mathrm{CH}^i(\HHom_n^N)
\xlongrightarrow[\,\raisebox{0.2 em}{\smash{\ensuremath{\sim}}}\,]{j^*}\mathrm{CH}^i(\Gr(N,n)),\quad N-n+1>i
$$
where the isomorphisms come from homotopy invariance and excision, respectively. On the other hand, a similar argument gives an isomrophism $H^{\leq 2i}(B\GL_n)\simeq H^{\leq 2i}(\Gr(N,n))$ for $N-n+1>i$. Consequently, the cohomology and Chow groups of $B\GL_n$ can be understood via those of $\Gr(N,n)$ if $N$ is sufficiently large. This reduces problems about tautological generation, purity and the cycle class map for $B\GL_n$ to those of $\Gr(N,n)$, which are better understood. 

%Purity holds for $\Gr(N,n)$ because it is smooth projective. Tautological generation and class map can be studied by cutting out the diagonal using the tautological section 
%\begin{equation}\label{eqn: diagonal for Gr}
%\begin{tikzcd}
 %                 & \HHom(\CS_1,\CV_2) \arrow[d]                    \\
%\Gr(N,n)=\mathrm{Zero}(\tau) \arrow[r, hook] & \Gr(N,n)\times \Gr(N,n) \arrow[u, "\tau"', bend right]
%\end{tikzcd}
%\end{equation}
%where $\CS_i$ and $\CV_i$ are the tautological subbundle and quotient of the $i$'th factor, respectively. We leave out the detail on how diagram \eqref{eqn: diagonal for Gr} can be used to study tautological generation and class map since this will be explained in full in the proof of Theorem \ref{thm: tautological generation}. 

\smallskip

We follow the same strategy to study $\FM_{d,\chi}$. We first introduce the moduli of limit stable pairs, also known as Joyce--Song pairs, that plays the role of Grassmannians above. Let $N$ be an integer such that 
\begin{equation}\label{eqn: range for N}
dN+\chi\geq g=\frac{(d-1)(d-2)}{2}.
\end{equation}
Then Proposition \ref{prop: vanishingMR} implies
\begin{equation}\label{eqn: Ext vanishing for N}
    \Ext^1(\CO(-N),F)=\Ext^2(\CO(-N),F)=0
\end{equation}
for any semistable sheaf $F$ of type $(d,\chi)$.
\begin{defn}\label{def: limit stable pair}
    Assume that $N$ and $(d,\chi)$ satisfy \eqref{eqn: range for N}. We say that a pair $(F,\phi)$ of a one-dimensional sheaf $F$ of type $(d,\chi)$ and a nonzero map $\phi:\CO(-N)\rightarrow F$ is \textit{limit stable} if 
    \begin{enumerate}
    \item [(i)] $F$ is semistable, and
    \item [(ii)] $\phi$ does not factor through any destabilizing subsheaf of $F$. 
\end{enumerate}

A morphism between two of such pairs $(F,\phi)$ and $(F',\phi')$ is a morphism $f:F\rightarrow F'$ such that $f\circ \phi=\phi'$. 
\end{defn}
\begin{rmk}\label{rmk: any morphism is iso}
    A morphism $f:F\rightarrow F'$ between two limit stable pairs $(F,\phi)$ and $(F',\phi')$ of identical topological type is necessarily an isomorphism. This is because if $f$ is not an isomorphism nor zero, then the image of $f$ defines a destabilizing subsheaf of $F'$ that $\phi'$ factors through.
\end{rmk}

The limit stable pairs form a projective moduli space, denoted by $P^N_{d,\chi}$, whose construction via geometric invariant theory is due to \cite{LePcoherent}. See also \cite{Lin} for a modern account in a generalized setting.

\begin{proof}[Proof of Theorem \ref{thm: tautological generation}]
% We prove Theorem \ref{thm: tautological generation} closely following the analogy with $\BGL_n$ case. 
Assume throughout the proof that $N$ and $(d,\chi)$ satisfy \eqref{eqn: range for N}. We first construct a diagram analogous to \eqref{eqn: diagram for BGL}. By the vanishing \eqref{eqn: Ext vanishing for N}, we obtain a vector bundle 
$$\HHom^N_{d,\chi}\coloneqq p_*\HHom(q^*\CO(-N),\CF)\overset{\pi}{\longrightarrow}\FM_{d,\chi}. 
$$
The total space $\HHom^N_{d,\chi}$ of this vector bundle is a stack parametrizing pairs of a semistable sheaf $F$ of type $(d,\chi)$ together with a morphism $\phi:\CO(-N)\rightarrow F$. Since limit stable pairs form an open subset among all such pairs, we obtain a diagram 
\begin{equation}\label{eqn: diagram for M}
\begin{tikzcd}
\HHom_{d,\chi}^N \arrow[d, "\pi"] &  P^N_{d,\chi}\arrow[l, "j"'] \\
\FM_{d,\chi}             &             
\end{tikzcd}
\end{equation}
where $\pi$ is a vector bundle and $j$ in an open embedding. We note that $P^N_{d,\chi}$ is smooth because it is an open subset inside the total space of a vector bundle over a smooth stack. 

We now estimate the codimension of the complement of $j$. For each semistable sheaf $F$ of type $(d,\chi)$, we need to study the closed inclusion of the unstable locus 
\begin{equation}\label{eqn: unstable locus}
    \Hom(\CO(-N),F)^{\textnormal{unstable}}\hookrightarrow \Hom(\CO(-N),F),
\end{equation}
i.e., the locus of morphisms that factor through some destabilizing subsheaf $F'\subsetneq F$. Note that the dimension of the space of destabilizing subsheaves $F'\subsetneq F$ can be bounded from above by a constant $c=c(d,\chi)$ that only depends on $(d,\chi)$. For example, such an upper bound can be obtained by the maximum dimension of Quot schemes $\mathrm{Quot}_{\BP^2}(F,(d',\chi'))$ for each $F$ and $1\leq d'<d$ with $\chi/d=\chi'/d'$. The precise bound will not be relevant to the proof. On the other hand, for each destabilizing subsheaf $F'\subsetneq F$ of type $(d',\chi')$ we have 
\begin{align*}
    \dim \Hom(\CO(-N),F')&=\chi(\CO(-N),F')\\
    &=\frac{d'}{d}(dN+\chi)\\
    &\leq \frac{d-1}{d}(dN+\chi).
\end{align*}
Combining these estimates, we obtain that
\begin{multline*}
    \dim \Hom(\CO(-N),F)-\dim \Hom(\CO(-N),F)^{\textnormal{unstable}}\\
    \geq (dN+\chi)-\left(
    c(d,\chi)+\frac{d-1}{d}(dN+\chi)
    \right)
    \eqqcolon N-\tilde{c}(d,\chi). 
\end{multline*}
In other words, the codimension of the complement of $j$ is bounded below by $N-\tilde{c}(d,\chi)$ where $\tilde{c}$ depends only on $(d,\chi)$. In particular, it grows arbitrarily large as $N$ does. 

By homotopy invariance and the excision sequence again, as in the case of $B\GL_n$, we conclude that the diagram \eqref{eqn: diagram for M} implies 
\begin{equation}\label{eqn: isomorphic range}
    \mathrm{CH}^{\leq i}(\FM_{d,\chi})\simeq \mathrm{CH}^{\leq i}(P^N_{d,\chi})\,,\quad H^{\leq 2i}(\FM_{d,\chi})\simeq H^{\leq 2i}(P^N_{d,\chi}),\quad \textnormal{if}\ \ N-\tilde{c}(d,\chi)>i.
\end{equation}
For any given $i\geq 0$, we can find a sufficiently large $N$ that realizes the above isomorphisms. So it suffices to prove the statement of the theorem for the moduli space of limit stable pairs $P^N_{d,\chi}$. This is done in the next proposition, finishing the proof. 
\end{proof}

\begin{prop}\label{prop: pair moduli}
Assume that $N$ and $(d,\chi)$ satisfy \eqref{eqn: range for N}. The cohomology ring of $P^N_{d,\chi}$ is tautologically generated and pure. Furthermore, the cycle class map from Chow to cohomology is an isomorphism. 
\end{prop}
\begin{proof}
Purity of the cohomology follows directly from smoothness and projectivity of $P=P^N_{d,\chi}$. Note that there exists a universal pair 
$$\Phi: q^*\CO(-N)\rightarrow \CF$$
over $P\times \BP^2$ by restricting that from $\HHom^N_{d,\chi}\times \BP^2$. By tautological classes of $P$ in either Chow or cohomology, we mean those classes obtained from $\CF$ just like in the case of moduli stack of sheaves. Since $\CF$ is indeed a pull-back from $\FM_{d,\chi}\times \BP^2$, the notion of tautological classes for $\FM_{d,\chi}$ and $P^N_{d,\chi}$ agrees in both Chow and cohomology in the range \eqref{eqn: isomorphic range} where they are isomorphic. 

To prove statements about tautological generation and the cycle class map, we shall construct an explicit Chow--Künneth decomposition of the diagonal 
$$\Delta\in \mathrm{CH}^k(P_1\times P_2),\quad k:=\dim(P).$$
The subscripts in $P_1$ and $P_2$ are just to distinguish the two copies of $P$. Denote the perfect complex coming from the universal pair by 
$$\CS^\bullet\coloneqq \Big[q^*\CO(-N)\xrightarrow{\Phi} \CF\Big]\in D^b(P\times \BP^2)
$$
which is placed in degrees $[0,1]$. %Similar to the Grassmanian case as in \eqref{eqn: diagonal for Gr}, 
We consider the complex $$\RHom_p(\CS^\bullet_1,\CF_2)\in D^b(P_1\times P_2)
$$
where $p:P_1\times P_2\times \BP^2\rightarrow P_1\times P_2$ is the projection map and $\CS^\bullet_i$ and $\CF_i$ are pull-backs from $P_i\times \BP^2$. We claim that $R\HHom_p(\CS^\bullet_1,\CF_2)$ is a vector bundle of rank $k$ admitting a tautological section $\tau$ that cuts out the diagonal $P=\mathrm{Zero}(\tau)\hookrightarrow P\times P$. Assuming the claim for the moment, we obtain the formula
\begin{align*}
    \Delta&=c_{k}(R\HHom_p(\CS^\bullet_1,\CF_2))\\
    &=\left[\exp\left(
    \sum_{i\geq 1}(-1)^{i-1}(i-1)!\sum_{a+b=i}\ch^{(\CS^\bullet_1)^\vee}_a\ch^{\CF_2}_b(\td(\BP^2))
    \right)\right]_k\,.
\end{align*}
Here $\ch^{(\CS^\bullet_1)^\vee}_a\ch^{\CF_2}_b(\td(\BP^2))$ denotes\footnote{For a $K$-theory class $\CV\in K^0(P\times \BP^2)$, the notation $\xi_{\CV}(-)$ means that we use $\ch(\CV)$ for the realization. }
\[\sum_{i}\xi_{(\CS^\bullet_1)^\vee}(\ch_a(\gamma_i^L))\xi_{\CF_2}(\ch_b(\gamma_i^R))\in H^\ast(P)\,,\]
where $\Delta_\ast \td(\BP^2)=\sum_{i}\gamma_i^L\otimes \gamma_i^R$, see Lemma \ref{lem: geomrealizationC} for a similar statement.
This yields an explicit Chow--Künneth decomposition of the diagonal in the tautological classes after expanding the formula. Note that this argument uses the Chow--Künneth decomposition of the diagonal of $\BP^2$. By standard arguments, see for example \cite{Beauville_diagonal}, this implies tautological generation and the cycle class map being an isomorphism for $P$. 

We finish the proof by establishing the above claim. To show the vector bundle property, it suffices to prove 
\begin{equation}\label{eqn: required vanishing}
    \Ext^i\big([\CO(-N)\xrightarrow{\phi_1}F_1]\,,\,F_2\big)=0
\end{equation}
for all $i\neq 0$ and $(F_j,\phi_j)\in P$ with $j=1,2$. This can be proved by  using the long exact sequence obtained by applying $\Hom(-,F_2)$ to the exact triangle 
$$ S^\bullet_1\xrightarrow{\ i_1\ } \CO(-N)\xrightarrow{\ \phi_1\ } F_1\xrightarrow{\ +1\ }.
$$
Both (i) and (ii) in Definition \ref{def: limit stable pair} are used to show the vanishing. The argument is standard, see for example \cite[Lemma 5.3]{BLM}. The only difference is that we need to use an extra vanishing $\Ext^2(F_1,F_2)=0$. The rank of the vector bundle can be computed by $\chi(S^\bullet,F)=\dim(P)$. 

\smallskip

Over $P_1\times P_2\times \BP^2$, consider the following morphism given by the composition
$$\tilde{\tau}:\CS^\bullet_1\xrightarrow{\ I_1\ } \CO(-N)\xrightarrow{\ \Phi_2\ } \CF_2. 
$$
Via adjunctions, this induces a tautological section $\CO_{P_1\times P_2}\xrightarrow{\tau}R\HHom_p(\CS^\bullet_1,\CF_2)$ of the vector bundle. We show that $Z\coloneqq \mathrm{Zero}(\tau)$ equals the diagonal $\Delta$ scheme-theoretically. Since $\tilde{\tau}\big|_{\Delta\times \BP^2}=\Phi\circ I=0$, we have $\Delta\hookrightarrow Z$. On the other hand, we have the following diagram
\begin{equation}\label{eqn: diagram with exact triangle}
\begin{tikzcd}
\CS^\bullet_1\big|_{Z\times \BP^2} \arrow[r, "I_1"] & \CO(-N) \arrow[r, "\Phi_1"] \arrow[d, equal] & \CF_1\big|_{Z\times \BP^2} \arrow[r, "+1"] & {} \\
\CS^\bullet_2\big|_{Z\times \BP^2} \arrow[r, "I_2"] & \CO(-N) \arrow[r, "\Phi_2"]                                & \CF_2\big|_{Z\times \BP^2} \arrow[r, "+1"] & {}
\end{tikzcd}
\end{equation}
over $Z\times \BP^2$. By definition of $Z$, we have $\Phi_2\circ I_1=0$ where we suppressed the restriction to $Z\times \BP^2$ from the notation. Since $R\HHom_p(\CS^\bullet_1,\CF_2)$ is a vector bundle, local-to-global spectral sequence implies that 
$$\Ext^{-1}_{Z\times \BP^2}(\CS^\bullet_1\big|_{Z\times \BP^2},\CF_2\big|_{Z\times \BP^2})=H^{-1}(R\HHom_p(\CS^\bullet_1,\CF_2)\big|_{Z\times \BP^2})=0.$$
Therefore, we have a part of the long exact sequence 
$$0\rightarrow \Hom(\CS^\bullet_1\big|_{Z\times \BP^2}, \CS^\bullet_2\big|_{Z\times \BP^2})
\xrightarrow{I_{2*}} \Hom(\CS^\bullet_1\big|_{Z\times \BP^2}, \CO(-N))
\xrightarrow{\Phi_{2*}} \Hom(\CS^\bullet_1\big|_{Z\times \BP^2}, \CF_2\big|_{Z\times \BP^2})\rightarrow\cdots
$$
where we used the above vanishing for the leftmost zero. Since $I_1$ in the middle term is mapped to $\Phi_2\circ I_1=0$, there exists unique $g\in \Hom(\CS^\bullet_1\big|_{Z\times \BP^2}, \CS^\bullet_2\big|_{Z\times \BP^2})$ such that $I_2\circ g=I_1$. In other words, the vertical map $g$ makes the left square in \eqref{eqn: diagram with exact triangle} commutes. By the axioms of triangulated category, this induces a morphism $h\in \Hom(\CF^\bullet_1\big|_{Z\times \BP^2}, \CF^\bullet_2\big|_{Z\times \BP^2})$ making the right square commutes in \eqref{eqn: diagram with exact triangle}. Such a morphism $h$ is necessarily an isomorphism by the relative version of Remark \ref{rmk: any morphism is iso}. The universal property of the moduli space $P=P^N_{d,\chi}$ then implies that a family of limit stable pairs $\CO(-N)\xrightarrow{\Phi_i}\CF_i\big|_{Z\times \BP^2}$ for $i=1, 2$ defines the same morphism $Z\rightarrow P_i=P$. In other words, $Z\rightarrow P_1\times P_2$ factors through the diagonal $\Delta$, completing the proof. 
\end{proof}

\begin{rmk}
    The results in this section, i.e. Theorem \ref{thm: tautological generation} and Proposition \ref{prop: pair moduli}, apply to more general settings. For example, let $S$ be a smooth del Pezzo surface, $H$ an ample divisor and $\alpha$ a topological type of either positive rank or one-dimensional sheaves. Then $\Ext^2_S(F_1,F_2)=0$ for all $F_1,F_2\in \FM_\alpha^{H{\textup{-ss}}}(S)$, which in particular implies that $\FM_\alpha^{H{\textup{-ss}}}(S)$ is smooth. By picking a positive enough line bundle $L$, we can construct the limit stable pair moduli space $P_{\alpha}^L(S)$ which is smooth projective. Then the same proof of Proposition \ref{prop: pair moduli} applies to show that $H^*(P^L_\alpha(S))$ is pure, tautologically generated and isomorphic to the Chow ring. The same properties hold for the moduli stack $\FM_\alpha^{H\textup{-ss}}(S)$ because we can identify $H^i(\FM_\alpha^{H\textup{-ss}}(S))\simeq H^i(P_{\alpha}^L(S))$ where $L$ is chosen sufficiently positive according to the cohomological degree $i$. 
\end{rmk}

\subsection{Applications}

\label{sec: applications}

In this subsection, we apply results from Sections \ref{sec: BPS integrality} and \ref{sec: tautological generation} to the structure of the cohomology ring of the moduli stacks. See \cite{MozRei} for the related discussion. 

We define the virtual Hodge polynomial of $M\in \MHM(\pt)$, viewed as a mixed Hodge structure, by 
$$E(M)=E(M,u,v):=\sum_{p,q\in \BZ}\dim h^{p,q}(\mathrm{gr}^{W}_{p+q}M)  u^p v^q\in \BZ[u^{\pm},v^{\pm}]. 
$$
The virtual Hodge polynomial satisfies the short exact sequence relation, hence extends to the $K$-theory $K(\MHM(\pt))$. In fact, this defines a ring homomorphism 
$$E(-):K(\MHM(\pt))\rightarrow \BZ[u^{\pm},v^{\pm}]. 
$$
We further extend this homomorphism to the localized $K$-theory \eqref{eqn: localized K theory}
$$E(-):\widehat{K}(\MHM(\pt))\rightarrow \BZ[u^{\pm},v^{\pm}, (uv)^{1/2}, (1-(uv)^n)^{-1} \mid n\geq 1]
$$
by letting
$$E(\BL^{1/2},u,v)=-(uv)^{1/2},\quad E((1-\BL^n)^{-1},u,v)=(1-(uv)^n)^{-1}.
$$
Denote by $\BQ\llbracket \Lambda_\mu\rrbracket$ the (completed) monoid ring associated to $\Lambda_\mu$ whose elements are of the form $\sum_{\alpha\in \Lambda_\mu}c_\alpha\cdot e^{\alpha}$ for a collection of $c_\alpha\in \BQ$.\footnote{We consider the completion to be compatible with our definition of $\widehat{K}(\MHM(\Lambda_\mu))$, which is defined as a product $\prod_\lambda\widehat{K}(\MHM(\pt_\lambda))$ rather than a direct sum. } The Hodge polynomial homomorphism naturally extends to the $\Lambda_\mu$-graded version by
$$E(-):\widehat{K}(\MHM(\Lambda_\mu))\rightarrow \BQ\llbracket \Lambda_\mu\rrbracket\otimes \BZ[u^{\pm},v^{\pm}, (uv)^{1/2}, (1-(uv)^n)^{-1} \mid n\geq 1]. 
$$
Denote the shifted (intersection) Hodge polynomial of the moduli stacks and spaces by
$$\overline{E}(\FM_{d,\chi}):=E(H^*(\FM_{d,\chi},\BQ^\vir),u,v),\quad 
\overline{\mathit{\mathit{IE}}}(M_{d,\chi}):=E(H^*(M_{d,\chi},\IC^\vir),u,v). 
$$
Note that we have 
$$\overline{E}(B\BG_m)=\frac{-(uv)^{1/2}}{1-uv}. 
$$

By applying the homomorphism $E(-)$ to the BPS integraity \eqref{eqn: BPS integrality in K theory}, we obtain the following

\begin{cor}\label{cor: BPS integrality for Hodge polynomial}
We have an equality
    \begin{equation}\label{eqn: BPS integrality for Hodge polynomial}
        \sum_{(d,\chi)\in \Lambda_\mu}\overline{E}(\FM_{d,\chi},u,v)\cdot e^{(d,\chi)}
        =\PE\left(\frac{-(uv)^{1/2}}{1-uv}\cdot \sum_{(d,\chi)\in \Lambda_\mu^\times}\overline{\mathit{\mathit{IE}}}(M_{d,\chi},u,v)\cdot e^{(d,\chi)}
\right)
    \end{equation}
in $\BQ\llbracket\Lambda_\mu\rrbracket\otimes \BZ[u^{\pm},v^{\pm}, (uv)^{1/2}, (1-(uv)^n)^{-1} \mid n\geq 1]$. 
\end{cor}

Here, the notation $\PE(-)$ denotes the plethystic exponential for the variables $u, v$ and $e^\alpha$. Explicitly, it is given by 
    $$\PE(f(u,v,e^\alpha)):=\exp\left(
    \sum_{n\geq 1} \frac{f(u^n,v^n,e^{n\alpha})}{n}
    \right)
    $$
    where $f\in \BQ\llbracket\Lambda_\mu\rrbracket\otimes \BZ[u^{\pm},v^{\pm}, (uv)^{1/2}, (1-(uv)^n)^{-1}\,|\,n\geq 1]$ has zero coefficient for the $e^0$ component. 
    
\begin{rmk}
Since $E(-,u,v)$ is defined using the graded factors with respect to the weight filtration, it does not necessarily remember the cohomologically graded dimensions. In the case of pure Hodge structures, however, taking graded factors does not lose anything, so $E(-,u,v)$ can recover the information on cohomological gradings. Note that the purity of $H^i(\FM_{d,\chi})$ is proved in Theorem \ref{thm: tautological generation} and the intersection cohomology of a proper variety is always pure. Therefore, Corollary \ref{cor: BPS integrality for Hodge polynomial} relates two sets of numerical data
\begin{center}
    Poincaré series of $\{\FM_{d,\chi}\}_{(d,\chi)\in \Lambda_\mu}$\: $\longleftrightarrow$\: intersection Poincaré polynomials of $\{M_{d,\chi}\}_{(d,\chi)\in \Lambda_\mu^\times}$. 
\end{center}

In fact, both sides of the formula in \eqref{eqn: BPS integrality for Hodge polynomial} depend only on the product $uv$ rather than the individual $u$ and $v$, because the Hodge structures involved are not just pure but also algebraic, in the sense that they are of $(p,p)$-type. This was proved in Theorem \ref{thm: tautological generation} for $H^*(\FM_{d,\chi})$, and the case for the intersection cohomology of the moduli spaces follows from \cite[Theorem 0.4.1]{Bou}. Therefore, we do not lose any information by substituting $u=v=q^{1/2}$ and $(uv)^{1/2}=q^{1/2}$ to \eqref{eqn: BPS integrality for Hodge polynomial}, which reads
\begin{equation}\label{eqn: BPS for Poincare series}
    \sum_{(d,\chi)\in \Lambda_\mu}\overline{E}(\FM_{d,\chi},q)\cdot e^{(d,\chi)}=\PE\left(\frac{-q^{1/2}}{1-q}\cdot \sum_{(d,\chi)\in \Lambda_\mu^\times}\overline{\mathit{\mathit{IE}}}(M_{d,\chi},q)\cdot e^{(d,\chi)}
\right)
\end{equation}
in $\BQ\llbracket\Lambda_\mu\rrbracket\otimes \BZ[q^{\pm},q^{1/2}, (1-q^n)^{-1} \mid n\geq 1]$. Note that $\overline{E}(-,q)$ and $\overline{\mathit{\mathit{IE}}}(-,q)$ are simply the shifted (intersection) Poincaré series 
\begin{align*}
    \overline{E}(\FM_{d,\chi},q)&=(-q^{1/2})^{-\dim \FM_{d,\chi}}\sum_{i\geq 0} \dim H^{2i}(\FM_{d,\chi})\cdot q^i,\\
    \overline{\mathit{\mathit{IE}}}(M_{d,\chi},q)&=(-q^{1/2})^{-\dim M_{d,\chi}}\sum_{i\geq 0} \dim \mathit{IH}^{2i}(M_{d,\chi})\cdot q^i
\end{align*}
by the purity and algebraicity of the (intersection) cohomology. We will use the notation $E(\FM_{d, \chi}, q)$ and $\mathit{IE}(M_{d, \chi}, q)$ for the non-shifted Poincaré series, i.e. without the $(-q^{1/2})^{-\dim}$ factors.
\end{rmk}

The BPS integrality formula \eqref{eqn: BPS for Poincare series}, combined with the $\chi$-independence \eqref{eqn: MS_chi_indep}, allows us to compute Poincaré series of the stacks using those of the coprime moduli spaces, as explained in Section \ref{sec: BPS intro}. This gives a key completeness criterion for geometric relations in our computation of their cohomology rings. 
% of shifted (intersection) Poincaré series as in \eqref{eqn: BPS for Poincare series} 
Furthermore, the BPS integrality formula implies strong structural results on the cohomology of $\FM_{d,\chi}$ as we explain next.

\begin{thm}\label{thm: structure of cohomology of stack}
The (non-shifted) Poincaré series $E(\FM_{d,\chi},q)$ depends on $(d,\chi)$ only through its degree $d$ and its multiplicity $m\coloneqq \gcd(d,\chi)$. More precisely, there exists a polynomial
$$A_{d,m}(q)=1+\cdots +(-1)^{m-1} q^{N},\quad N=d^2+\frac{m(m+1)}{2}
$$
such that 
$$E(\FM_{d,\chi},q)=\frac{A_{d,m}(q)}{(1-q)(1-q^2)\cdots (1-q^m)}.
$$
\end{thm}
\begin{proof}
The first statement, concerning the dependence of $E(\FM_{d,\chi}, q)$, follows immediately from the BPS integrality formula \eqref{eqn: BPS for Poincare series} and the $\chi$-independence theorem for intersection cohomology \cite[Theorem 0.1]{MS_chi-indep}, cf. Equation \eqref{eqn: MS_chi_indep}.

    By \eqref{eqn: BPS for Poincare series}, the (non-shifted) Poincaré series $E(\FM_{d,\chi},q)$ is a sum of terms of the form
    \begin{equation}\label{eq: termstructurepoincare}
        q^{d^2/2}\prod_{i=1}^\ell \frac{q^{k_i/2}}{1-q^{k_i}}q^{-k_i(d_i^2+1)/2}\mathit{IE}(M_{d_i, \chi_i}, q^{k_i})
    \end{equation}
    with
    \[\ell,\, k_i\geq 1\,,\quad\sum_{i=1}^\ell k_i \chi_i=\chi\,,\quad\sum_{i=1}^\ell k_i d_i=d\,,\quad(d_i, \chi_i)\in \Lambda_\mu\,.\]
    For any such term
    \[m=\gcd(d, \chi)=\sum_{i=1}^\ell k_i\gcd(d_i, \chi_i)\geq \sum_{i=1}^\ell k_i\,.\]
    Hence the denominator $\prod_{i=1}^\ell (1-q^{k_i})$ divides $\prod_{j=1}^m (1-q^j)$ (for example by using the fact that $q$-multinomial coefficients are polynomials in $q$), so we can write $E(\FM_{d,\chi},q)$ in the form stated by the theorem with $A_{d, m}(q)$ a Laurent polynomial in $q$. It remains to show that $A_{d,m}$ is actually a polynomial and that it has the specified degree, specified constant term and specified leading coefficient.

    We first observe that terms \eqref{eq: termstructurepoincare} only contribute to the $q^s$ coefficient of $A_{d,m}(q)$ for
    \begin{equation}\label{eq: ineqstructuralanalysis}
   s\geq \frac{d^2}{2}+\sum_{i=1}^\ell \left(\frac{k_i}{2}-\frac{k_i(d_i^2+1)}{2}\right)= \frac{d^2}{2}-\sum_{i=1}^\ell \frac{k_id_i^2}{2}\geq \frac{d^2}{2}-d\sum_{i=1}^\ell \frac{k_id_i}{2}=0\,.\end{equation}
    This shows that $A_{d,\chi}(q)$ is a polynomial. Equality in \eqref{eq: ineqstructuralanalysis} holds only if $\ell=1$, $k_1=1$ and $(d_1, \chi_1)=(d, \chi)$, so only this term (which we will refer to as the ``main term'') contributes to the constant coefficient. Similarly we can bound the degree of the rational function and find that only the main term contributes to the leading coefficient. The contribution of the main term to $E(\FM_{d, \chi}, q)$ is
    \[\frac{1}{1-q}\mathit{IE}(M_{d,\chi}, q)=\frac{\mathit{IE}(M_{d,\chi}, q)(1-q^2)(1-q^3)\ldots (1-q^m)}{(1-q)(1-q^2)\ldots (1-q^m)}\]
    which makes it clear that the constant coefficient is 1 and the top coefficient is $(-1)^{m-1}$.
\end{proof}
\begin{rmk}
    Note that degree of the rational function on the right hand side is equal to the dimension $\dim \FM_{d,\chi}=d^2$. This type of equality need not be true for arbitrary smooth quotient stacks $\FX=[X/G]$.
\end{rmk}

The BPS integrality formula provides control of the cohomology in small degrees, and can be used to prove a free generation result for the moduli stacks analogous to Theorem \ref{generation} (ii).

\begin{thm}\label{thm: free generation in low degree}
The homomorphism $\BD_{\alpha}\rightarrow H^*(\FM_{d,\chi})$ defines an isomorphism up to degree $2d-4$.
\end{thm}
\begin{proof}
    We already showed that the homomorphism $\BD_{\alpha}\to H^*(\FM_{d,\chi})$ is surjective in Theorem~\ref{thm: tautological generation}, so it is enough to prove that the dimensions agree up to degree $2d-4$. The dimension $\dim H^{2i}(\FM_{d,\chi})$ is the $q^i$ coefficient of $E(\FM_{d,\chi},q)$. We follow the proof of Theorem~\ref{thm: structure of cohomology of stack} and argue that only the main term contributes to the coefficient of $q^i$ with $i\leq d-2$ by strengthening the inequality \eqref{eq: ineqstructuralanalysis}. If either $\ell\geq 2$ or $k_1\geq 2$ it is easy to see that 
    \[s\geq \frac{d^2}{2}-\frac{(d-1)^2}{2}-\frac{1^2}{2}=d-1\,,\]
    showing that indeed the other terms only contribute in higher degrees. Hence, for $i\leq d-2$ we have the following chain of equalities that we will justify below:
    \begin{align*}\dim H^{2i}(\FM_{d,\chi})&=[q^i]E(\FM_{d,\chi},q)=[q^i]\frac{\mathit{IE}(M_{d,\chi},q)}{1-q}=\dim_{2i}H^*(M_{d,1})[u]\\
    &=\dim_{2i}\BD_{\alpha, \inv}[u]=\dim_{2i}\BD_{\alpha}\,.
    \end{align*}
    We have used $[q^i]$ to denote the $q^i$-coefficient in a series and $\dim_{2i}$ to denote the dimension of the $2i$-graded part of some graded algebra. The formal variable $u$ has degree 2. The first equality is by definition, the second comes from the previous observation that only the main term contributes for such degrees, the third uses the $\chi$-independence theorem  of Maulik--Shen (\cite[Theorem 0.1]{MS_chi-indep} or Equation \eqref{eqn: MS_chi_indep}), the fourth uses the freeness for the good moduli spaces (\cite[Theorem 0.2]{PS} or Theorem \ref{generation}), and finally the last equality uses the isomorphism $\BD_\alpha\simeq \BD_{\alpha, \inv}[u]$ from Proposition \ref{prop: isowt0}.
\end{proof}

We also have an analogue of Theorem \ref{generation} (i) for stacks by following essentially the same strategy as in \cite{PS}.
\begin{thm}\label{thm: smallgenerationstacks}
    The algebra $H^*(\FM_{d,\chi})$ is generated by the $3d$ classes
    \[c_{k+1}(0),\, c_{k}(1),\, c_{k-1}(2) \in H^{2k}(\FM_{d,\chi}),\;\; 1\leq k \leq d.\]
\end{thm}
\begin{proof}
    See \textit{Step 1} in \cite[Section 2.3]{PS}. Note that when $\chi=0$ we are not allowed to use the relations $R_1$ (in the notation of loc. cit.) as the inequality \eqref{eq: inequalitygmr} is not satisfied, but this can easily be avoided.
\end{proof}

Note that this result is weaker than Theorem \ref{generation} (i), in the sense that we require more generators and that we do not claim that this is a minimal generating set. This is due to the fact that the leading terms analyzed in \textit{Steps 2--4} of \cite[Section 2.3]{PS} are not necessarily non-zero if $\chi=d$ or $\chi=d/2$ (see for instance the formula for $\det(\mathsf{M})$ in loc. cit.).

\section{Computation of the cohomology rings}

\label{sec: computations of the rings}

In this section we compute all the cohomology rings in Theorem \ref{thm: rings}. By the symmetries before Theorem \ref{chi-dep}, it suffices to restrict our attention to $0\leq \chi \leq {d}/{2}$. We shall use the normalized classes $c_k(j)$ in Definition \ref{defprof: cij} to present the rings, as explained in Remark \ref{rmk: normalized_class}.

\subsection{Strategy of the computation.}
\label{sec: strategy}
We start by outlining a detailed strategy for the computation, combining the results and tools we developed in the previous sections. All our computations are implemented in the software \textsc{Macaulay2} \cite{M2}.

\smallskip

\begin{enumerate}[label = (\roman*)]
    \item By Theorem \ref{thm: tautgenerated}, the normalized tautological classes $c_k(j)$  generate the cohomology rings of $\FM_{d,\chi}$ and $M_{d,\chi}$. For the moduli spaces $M_{d,\chi}$, classes in cohomological degrees $\leq 2d-4$ generate according to Theorem \ref{generation} (i); for the stacks $\mathfrak{M}_{d,\chi}$ classes in degrees $\leq 2d$ generate the cohomology, see Theorem \ref{thm: smallgenerationstacks}. The other normalized classes can be expressed in terms of these generators via the relations in $\FI_{d,\chi}^{\gmr,1,\chi'}$.
    
    Let $R$ be a free polynomial ring in these generators, with correct cohomological degrees specified for each. We form an ideal $I$ to which we shall add geometric relations step by step and take its quotient. For now, we set $I=(0)$. 

    \smallskip

    \item Compute the Poincaré series of the quotient $R/I$ and compare it with the Poincaré series of the moduli space $M_{d,\chi}$ (known in small degrees, see for example \cite{lep93, CC2, CC}) or the moduli stack $\mathfrak{M}_{d,\chi}$ (obtained from BPS integrality). Identify the first cohomological degree, say $\ell$, where the two series do not match. This is where the first missing relations occur.

    \smallskip

    \item Apply Virasoro operators to generators of $I$ to obtain relations in degree $\ell$. Using Step (i), reduce the new relations to expressions in terms of the ring generators only, and add these to the ideal $I$.

    \smallskip

    \item If Step (iii) does not give all the missing relations in degree $\ell$, produce primitive relations in degree $\ell$. Note that we do not need primitive relations of degrees \textit{less than $\ell$} according to Theorem \ref{thm: primitive relatiosn generate}, since we have already performed Step (iii). Reduce the new relations and add them to the ideal $I$.

    \smallskip

    \item Repeat Steps (ii)-(iv) above, until the Poincaré series match in all degrees, meaning that we have obtained all the tautological relations\footnote{Note that $R$ is Noetherian and the ideal of tautological relations is finitely generated.}, or conclude that the geometric relations are not complete.

    \smallskip

    \item (Optional) Use the command \texttt{minimalPresentation} to obtain a minimal presentation of the cohomology ring and the ideal $I$.
    
\end{enumerate}

\begin{rmk}
    Although the strategy above makes use of the known Poincaré polynomials \cite{lep93, CC2, CC} of the moduli spaces $M_{d,\chi}$, we remark that our results have an independent nature and do not depend on these numerical data in any essential way. Indeed, since $M_{d,\chi}$ is smooth and projective when $\gcd(d,\chi)=1$, once we have an ideal $I$ of geometric relations such that the quotient ring $R/I$ is Gorenstein with the correct dimension, Poincaré duality would force an isomorphism $H^*(M_{d,\chi}) \simeq R/I$. The fact that we do have enough relations to obtain the cohomology rings for $d\leq 5$ is precisely the content of Theorem \ref{thm: rings} (for the moduli spaces). One then follows Steps (ii)-(v) to treat the moduli stacks. We chose to present our strategy as above since it is in practice much more algorithmic.
    
\end{rmk}

\subsection{The cohomology rings.}
\label{sec: rings}
In this section, we describe the computation of $H^*(\mathfrak{M}_{2,0}, \mathbb{Q})$ in detail following the above strategy. For the other moduli stacks (or spaces), we merely list key information such as Poincaré series, numbers of generators and relations, etc. for their cohomology rings, as the computations are much more involved. All presentations of the cohomology rings, including the partial one for $\mathfrak{M}_{5,0}$, are uploaded to the fourth author's website
\centerline{\url{https://github.com/Weite-Pi/weitepi.github.io}}
in a folder named \texttt{cohomology rings}.

\begin{rmk}
    The moduli space $M_{1,0}$ is isomorphic to $\mathbb{P}^2$, whose cohomology is well-known and can be obtained as in Example \ref{M10example}. The moduli spaces $M_{2,1}$ and $M_{3,1}$ are also classical \cite{lep93}. The cohomology of the moduli space $M_{4,1}$ is first computed in \cite{CM} using birational geometry, and translated to a presentation in terms of normalized classes $c_k(j)$ in \cite{KPS}. As explained in the previous remark, the cohomology rings of these moduli spaces can all be obtained directly using geometric relations, and we omit the details for them in the following.
    \end{rmk}

\subsubsection{The moduli stack $\mathfrak{M}_{2,0}$.} The cohomology of the stack $\mathfrak{M}_{2,0}$ is generated by the normalized classes in degrees $\leq 4$. We thus input the command
\[
\texttt{R = QQ[c02,c11,c20,c12,c21,c30, Degrees => \{2,2,2,4,4,4\}]}
\]
in \textsc{Macaulay2} to get a free polynomial ring. Its Poincaré series has first terms
\[
E(R, q) = 1 + 3q + 9q^2 + 19q^3 + 39q^4 + 69q^{5} + 119 q^{6} + \cdots.
\]
On the other hand, the correct Poincaré series for $\mathfrak{M}_{2,0}$ can be obtained by BPS integrality:
\begin{align*}
    E(\mathfrak{M}_{2,0}, q) &= \frac{1 + q + 2q^2 + 2q^3 + 3q^4 + q^{5} - q^{7}}{\left(1-q\right) \left(1-q^2\right)}\\
    &= 1 + 2q + 5q^2 + 8q^3 + 14q^4 + 18 q^{5} + 24 q^{6} + \cdots.
\end{align*}

Thus the first relation occurs in cohomological degree 2 (note that this is twice the exponent of $q$ since the latter records algebraic degrees). We produce a Mumford relation in this degree:
\[
\mathbf{r}_1=c_2(0) - \frac{1}{8}c_0(2).
\]
One can also use $\gmr$ with $M_{1,\chi'}$, which produces the same relation. Setting $I$ to be the ideal generated by this relation, we compute the Poincaré series
\[
E(R/I, q) = 1 + 2q + 6q^2 + 10q^3 + 20q^4 + 30q^{5} + 50q^{6} + \cdots.
\]

Comparing with $E(\mathfrak{M}_{2,0}, q)$, we see that the next relation occurs in (cohomological) degree 4. We apply the Virasoro operator $\mathsf{L}_1$ to the relation $c_2(0) - \frac{1}{8}c_0(2)$, which gives a new relation
\[
\mathbf{r}_2=-\frac{1}{4}c_1(2) + 2c_3(0).
\]
Adding this to the ideal $I$, we have
\[
E(R/I, q) = 1 + 2q + 5q^2 + 8q^3 + 14q^4 + 20q^{5} + 30q^{6} + \cdots.
\]

The next relations thus occur in degree 10, and there are two of them. We apply the Virasoro operators to the two relations in $I$; this time, however, they do not give anything new, i.e. the relations we obtain still lie in $I$. We proceed by producing $\gmr$ with $M_{1,\chi'}$, and this does give two new relations. After simplifying, they are
\begin{align*}
    &
    \mathbf{r}_3=c_1(1)c_2(0)^4-2c_2(0)^3c_3(0),\\
    & \mathbf{r}_4=c_1(1)^2 c_2(0)^3+16c_2(0)^5+2c_2(0)^3c_2(1)-6c_1(1)c_2(0)^2c_3(0)+6c_2(0)c_3(0)^2.
    \end{align*}
Adding these to $I$, we have
\[
E(R/I, q) = 1 + 2q + 5q^2 + 8q^3 + 14q^4 + 18q^{5} + 26q^{6} + \cdots.
\]

The next two missing relations lie in degree 12. We apply Virasoro operators to the four generators in $I$, and this gives two new relations
%\begin{align*}
%\mathsf{L_1}(\mathbf{r}_3) =4c_2(0)^3c_1(1)c_2(1)+16c_2(0)^4c_3(0)-3c_2(0)^2c_1(1)^2c_3(0)&\\+\:\! 6c_2(0)c_1(1)c_3(0)^2-6c_2(0)^2c_2(1)c_3(0)-2c_3(0)^3&,\\ \mathsf{L}_1(\mathbf{r}_4) = c_2(0)^4c_2(1)-2c_2(0)^3c_1(1)c_3(0)+3c_2(0)^2c_3(0)^2&,
%\end{align*}
$\mathbf{r}_5 = \mathsf{L}_1(\mathbf{r}_3)$ and $\mathbf{r}_6=\mathsf{L}_1(\mathbf{r}_4)$. Adding these to the ideal and computing the Poincaré series, we found that\footnote{\textsc{Macaulay2} is able to reduce Poincaré series to this rational form using the command \texttt{reduceHilbert}.}
\[
E(R/I, q) = \frac{1 + q + 2q^2 + 2q^3 + 3q^4 + q^{5} - q^{7}}{\left(1-q\right) \left(1- q^2\right)}, 
\]
matching the correct Poincaré series of $\mathfrak{M}_{2,0}$. Thus we have found all the relations and 
\[
H^*(\mathfrak{M}_{2,0}, \mathbb{Q}) \simeq R/I
\]
where $I=\langle \mathbf{r}_1, \mathbf{r}_2, \mathbf{r}_3, \mathbf{r}_4, \mathbf{r}_5, \mathbf{r}_6\rangle$.

\smallskip

We can further trim the quotient ring and obtain a minimal presentation of $H^*(\mathfrak{M}_{2,0}, \mathbb{Q})$. In the trimmed form, it has four tautological generators
\[
c_1(1),\: c_2(0),\: c_2(1),\: c_3(0),
\]
and four relations, two in degree 10 and two in degree 12.

\subsubsection{The moduli stack $\mathfrak{M}_{3,0}$.}
\label{sec: M30}

The Poincaré series of $\mathfrak{M}_{3,0}$ is as follows:
\begin{align*}
    E(\mathfrak{M}_{3,0}, q) &= \frac{1}{\left(1-q\right) \left(1-q^2\right) \left(1-q^3\right)}\cdot (1 + 2q + 3q^2 + 3q^3 + 3q^4 + 4q^{5} + 7q^{6} \\&\qquad \qquad + 6q^{7} + 5q^{8} + q^{9} - 4q^{10} - 6q^{11} - 2q^{12} + q^{13} + 2q^{14} + q^{15})\\& = 1 + 3q + 7q^2 + 13q^3 + 22q^4 + 35q^{5} + 55q^{6} + 79 q^{7} + 111 q^{8} +  \cdots. 
\end{align*}

We obtain $H^*(\mathfrak{M}_{3,0}, \mathbb{Q})$ following the same algorithm as for $\mathfrak{M}_{2,0}$. In the trimmed form, it has 7 generators
\[
c_0(2),\: c_1(1),\: c_2(0),\: c_1(2),\: c_2(1),\: c_2(2),\: c_3(1).
\]

There are in total 16 relations among these  generators:

\medskip

\begin{center}
\begin{tabular}{||c c c c c c c c||} 
 \hline
  degrees & $H^4$ & $H^6$ & $H^8$ & $H^{14}$ & $H^{16}$ & $H^{18}$ & $H^{20}$ \\ 
 \hline
 \# of rel. & 1 & 2 & 2 & 2 & 3 & 4 & 2\\
 \hline
\end{tabular}
\end{center}

\medskip

\subsubsection{The moduli stack $\mathfrak{M}_{4,2}$.} The Poincaré series of $\mathfrak{M}_{4,2}$ is as follows:
\begin{align*}
    E(\mathfrak{M}_{4,2},q) &= \frac{1}{(1-q) (1-q^2)}\cdot(1+2 q+5 q^2+8 q^3+9 q^4+6 q^{5}+5 q^{6}+4 q^{7} +5 q^{8} \\&\qquad \;\;+5 q^{9}+6 q^{10}+4 q^{11}+3 q^{12}-3 q^{14}-8 q^{15}-8 q^{16}-5 q^{17}-2 q^{18}-q^{19})\\
    & = 1 + 3 q + 9 q^2 + 19 q^3 + 34 q^4 + 50 q^{5} + 70 q^{6} + 90 q^{7} + 115 q^{8} + \cdots.
\end{align*}

In the trimmed form, the cohomology ring $H^*(\mathfrak{M}_{4,2}, \mathbb{Q})$ has 7 generators
\[
c_0(2),\: c_1(1),\: c_2(0),\:
c_1(2),\:
c_2(1),\: c_3(0),\: c_4(0).
\]

There are in total 19 relations among these generators:

\vspace{7pt}

\begin{center}
\begin{tabular}{||c c c c c c c c c c||} 
 \hline
  degrees & $H^6$ & $H^8$ & $H^{10}$ & $H^{12}$ & $H^{22}$ & $H^{24}$ & $H^{26}$ & $H^{28}$ & $H^{30}$ \\ 
 \hline
 \# of rel. & 1 & 5 & 5 & 1 & 2 & 2 & 1 & 1 & 1\\
 \hline
\end{tabular}
\end{center}

\medskip

\subsubsection{The moduli stack $\mathfrak{M}_{4,0}$.} The Poincaré series of $\mathfrak{M}_{4,0}$ is as follows:
\begin{align*}
    E(\mathfrak{M}_{4,0}, q) & = \frac{1}{(1-q)^2 (1-q^2) (1 - q^4)}\cdot (1+q+3 q^2+4 q^3+3 q^4+q^{5}+3 q^{6}+2 q^{7}\\&\quad+5 q^{8}+6 q^{9}+7 q^{10}+4 q^{11}+3 q^{12}-6 q^{13}-11 q^{14}-13 q^{15}-6 q^{16}+3 q^{17}\\&\quad+10 q^{18}+10 q^{19}+4 q^{20}-2 q^{21}-3 q^{22}-q^{23}-q^{24})\\& = 1 + 3 q + 9 q^2 + 20 q^3 + 39 q^4 + 65 q^{5} + 105 q^{6} + 158 q^{7} + 234 q^{8} + \cdots.
\end{align*}

In the trimmed form, the cohomology ring $H^*(\mathfrak{M}_{4,0}, \mathbb{Q})$ has 10 generators
\[
c_0(2),\: c_1(1),\: c_2(0),\:
c_1(2),\:
c_2(1),\: c_3(0),\:
c_3(1),\:
c_4(0),\:
c_4(1),\:
c_5(0).
\]

There are in total 40 relations among these generators:

\medskip

\begin{center}
\begin{tabular}{||c c c c c c c c c c c||} 
 \hline
  degrees & $H^6$ & $H^8$ & $H^{10}$ & $H^{12}$ &
  $H^{18}$ &
  $H^{20}$ &
  $H^{22}$ & $H^{24}$ & $H^{26}$ & $H^{28}$ \\ 
 \hline
 \# of rel. & 1 & 5 & 5 & 2 & 2 & 2 & 4 & 6 & 8 & 5 \\
 \hline
\end{tabular}
\end{center}

\medskip

\subsubsection{The moduli space $M_{5,1}$.}
\label{sec: M51}
The Poincaré polynomial of $M_{5,1}$ have been computed by various methods, see for example \cite{Yuan3, CC2, Mai13}:
\begin{align*}
 E(M_{5,1}, q) &= 1+2q + 6q^2 + 13 q^3 + 26q^4 + 45 q^5 + 68 q^6 + 87q^7 + 100 q^8 + 107 q^9 + 111 q^{10} \\&\quad+ 112 q^{11} + 113 q^{12} + 113 q^{13} + 113 q^{14} + 112 q^{15} + 111q^{16} +107 q^{17} + 100 q^{18} \\&\quad+ 87 q^{19} + 68 q^{20} + 45 q^{21} + 26 q^{22} + 13 q^{23} + 6 q^{24} +2 q^{25} + q^{26}.
 \end{align*}

As we mentioned in Section \ref{sec: intro_relations}, the geometric relations ${I}^\geo_{5,1}$ turn out to be \textit{not} complete. Indeed, our algorithm stops when we try to find a new relation in degree 36. On the other hand, imposing all geometric relations does reduce the top dimension to one:
\begin{align*}
E(R/{I}^\geo_{5,1}, q) &= 1+2q + 6q^2 + 13 q^3 + 26q^4 + 45 q^5 + 68 q^6 + 87q^7 + 100 q^8 + 107 q^9 + 111 q^{10} \\&\quad+ 112 q^{11} + 113 q^{12} + 113 q^{13} + 113 q^{14} + 112 q^{15} + 111q^{16} +107 q^{17} + 101 q^{18} \\&\quad+ 91 q^{19} + 78 q^{20} + \cdots + q^{26}. 
\end{align*}

Recall that $M_{5,1}$ is a smooth projective variety \cite{lep93}. In particular, the cohomology ring $H^*(M_{5,1}, \mathbb{Q})$ satisfies Poincaré duality. Thus once the top dimension is reduced to one, we can use the Poincaré pairing to find the remaining relations. This completely determines the ring structure.

\smallskip

The generators for $M_{5,1}$ are given by Theorem \ref{thm: tautgenerated}. There are in total 34 relations as follows. The cross signs indicate three relations that are \textit{not} geometric. 

\medskip

\begin{center}
\begin{tabular}{||c c c c c c c c c||} 
 \hline
  degrees & $H^{10}$ & $H^{12}$ & $H^{14}$ & $H^{30}$ & $H^{34}$ & $H^{36}$ & $H^{38}$ & $H^{40}$\\ 
 \hline
 \# of rel. & 3 & 12 & 13 & $1$ & $1$ & $1^\times$ & $2^{\times}$ & $1$\\
 \hline
\end{tabular}
\end{center}

\medskip

\subsubsection{The moduli space $M_{5,2}$.} The Poincaré polynomial and ring generators of $M_{5,2}$ are the same as those of $M_{5,1}$. This time the geometric relations \textit{are} complete, namely $I^\geo_{5,2} = I_{5,2}$ and also $\FI^\geo_{5,2} = \FI_{5,2}$ on the stack level. There are in total 41 relations for the moduli space:

\medskip

\begin{center}
\begin{tabular}{||c c c c c c c c c c c c c||} 
 \hline
  degrees & $H^{10}$ & $H^{12}$ & $H^{14}$ & $H^{16}$ & $H^{18}$ & $H^{28}$ & $H^{30}$ & $H^{32}$ & $H^{34}$ & $H^{36}$ & $H^{38}$ & $H^{40}$\\ 
 \hline
 \# of rel. & 3 & 12 & 13 & 2 & 1 & 1 & $1$ & $1$ & 2 & $3$ & $1$ & $1$ \\
 \hline
\end{tabular}
\end{center}

\medskip

\begin{rmk}
    Our results on the cohomology of $M_{5,1}$ and $M_{5,2}$ give a different and explicit proof for the first non-trivial case of Theorem \ref{chi-dep}. Also as we discussed in Remark \ref{rmk: M52_M51}, one can ask whether $I^{\gmr,5,2}_{5,1}$ will produce the missing tautological relations for $M_{5,1}$.
\end{rmk}

%\we{Conjecture: relation of largest degree has codimension exactly $g=\frac{1}{2}(d-1)(d-2)$? This is true for $d\leq 5$.}

\subsubsection{The moduli stack $\mathfrak{M}_{5,0}$.} \label{sec: M50}
Finally, we determine the cohomology ring of the stack $\mathfrak{M}_{5,0}$ up to cohomological degree 32. We did not complete this for a purely computational reason, as the current ring structure is already quite complicated, and it takes an exceedingly long time to compute Poincaré series or to reduce the relations. We hope this partial result would still be interesting and useful for future research. 

\smallskip

First, we list the first terms in the Poincaré series of $\mathfrak{M}_{5,0}$:
\begin{align*}
    E(\mathfrak{M}_{5,0}, q) & = 1 + 3 q + 9 q^2 + 22 q^3 + 49 q^4 + 98 q^5 + 180 q^6 + 302 q^7 + 
 476 q^8\\& + 713 q^9 + 1032 q^{10} + 1449 q^{11} + 1994 q^{12} + 
 2690 q^{13} + 3575 q^{14} \\&+ 4679 q^{15} + 6047 q^{16} + 7706 q^{17} + 
 9696 q^{18} + 12027 q^{19} + 14731 q^{20} + \cdots
\end{align*}

Our computation shows that the geometric relations are complete up to cohomological degree 32. The cohomology ring $H^*(\mathfrak{M}_{5,0})$ has 13 generators 
\[
c_0(2),\: c_1(1),\: c_2(0),\:
c_1(2),\:
c_2(1),\: c_3(0),\: c_2(2),\:
c_3(1),\:
c_4(0),\:
c_3(2),\:
c_4(1),\:
c_4(2),\:
c_5(1).
\]

There are in total 56 relations up to degree 32:

\medskip

\begin{center}
\begin{tabular}{||c c c c c c c c c c c c c c||} 
 \hline
  degrees & $H^8$ & $H^{10}$ & $H^{12}$ & $H^{14}$ & $H^{16}$ & $H^{18}$ & $H^{22}$ & $H^{24}$ & $H^{26}$ & $H^{28}$ & $H^{30}$ & $H^{32}$ & $\cdots$\\ 
 \hline
 \# of rel. & 1 & 3 & 8 & 11 & 5 & 1 & 2 & 2 & 2 & 4 & 7 & 10 & $\cdots$ \\
 \hline
\end{tabular}
\end{center}

\medskip

The next relations occur in degree 34, and there are 14 of them. To find these relations, one can start by applying the Virasoro operators to the existing ones.

\section{\texorpdfstring{$P=C$}{P=C} and other conjectures}
\label{sec: conjectures}
In this final section, we formulate the strengthened version of the $P=C$ conjecture and verify several predictions on the cohomology rings of the moduli spaces (mostly $M_{5,\chi}$) in Theorem \ref{thm: rings}. Again, it suffices to verify each of the conjectures for $0\leq \chi \leq {d}/{2}$.

\subsection{The $P=C$ conjecture.}
\label{sec: P=C}

\subsubsection{The perverse filtration}
\label{sec: perverse filtration}

We start by recalling the construction of the perverse filtration \cite[Lecture 5]{dC_pcmi}; see also \cite[Section 1.1]{KPS}. Let $f: X\to Y$ be a proper morphism between irreducible varieties. For simplicity, we assume that $X$ and $Y$ are both smooth and projective. The perverse truncation functors \cite{BBDG} filter the derived push-forward complex $Rf_* \mathbb{Q}_X$, thus inducing an increasing filtration
\[
P_0 H^*(X) \subset P_1 H^*(X) \subset \cdots \subset H^*(X, \mathbb{Q})
\]
called the perverse filtration. Intuitively, it measures the complexity of the fibers of the proper map $f: X\to Y$.

\smallskip

We recall the following useful properties of the perverse filtration:

\begin{enumerate}[label=(\roman*)]
    \item Let $\eta$ be an ample class on $Y$, and denote by $\xi:= f^* \eta$ its pull-back to $X$. Then the perverse filtration on $H^*(X, \mathbb{Q})$ is completely determined by the multiplication by $\xi$ operator; see \cite[Proposition 5.2.4]{dCM} and \cite[Proposition 1.1]{KPS} for the explicit formula.
    
    \smallskip

    \item Let $r$ be the \textit{defect of semismallness} of $f$:
    \[
    r:= \dim X\times_Y X -\dim X
    \]
    When $f$ is equidimensional, this is simply the relative dimension $\dim X - \dim Y$. Then the perverse filtration $P_\bullet H^*(X, \mathbb{Q})$ stabilizes in $2r$ steps, i.e. $P_{2r}H^*(X) = H^*(X)$.
    
    \smallskip

    \item The graded pieces $\mathrm{gr}^P_\bullet H^*(X)$ of the perverse filtration satisfy two Hard Lefschetz type symmetries \cite[Theorem 2.1.4]{dCM}. They arise respectively from the pull-back of an ample class on $Y$ and a relative ample class on $X$. 
\end{enumerate}

\smallskip

Consider the moduli space $M_{d,\chi}$ with $\gcd(d,\chi)=1$ and the Hilbert--Chow map
\[
h: M_{d,\chi} \longrightarrow |d \cdot H|.
\]
This is a proper map with generic fibers being abelian varieties. Its topology is encoded in the perverse filtration $P_\bullet H^*(M_{d,\chi}, \mathbb{Q})$. %\footnote{When $\gcd(d,\chi)\neq 1$, the perverse filtration is defined on the intersection cohomology $P_\bullet \mathit{IH}^*(M_{d,\chi}, \mathbb{Q})$. } 
This is one side of the $P=C$ conjecture. 

\subsubsection{The Chern filtration}
\label{sec: chern filtration} We turn to the Chern side next. Recall that the (local) Chern filtration $C_\bullet H^{\leq 2d-4}(M_{d,\chi})$ was defined in \cite[Section 0.2]{KPS} for coprime $M_{d,\chi}$.
%\begin{equation}\label{eqn: Chern filtration}
 %   C_k H^{\leq 2d-4}(M_{d,\chi}) \coloneqq \mathrm{span}\left\{ \prod_{i=1}^s c_{k_i}(j_i) \,\Bigg|\, \sum_{i=1}^s k_i \leq k \right\} \subset H^{\leq 2d-4}(M_{d,\chi}, \mathbb{Q}).
%\end{equation}
By definition, it involves only the normalized tautological \textit{generators} in Theorem \ref{generation} (i). In order to extend this to a filtration on the entire cohomology rings, it is natural to include all normalized tautological \textit{classes.} We thus define the (global) Chern filtration 
\begin{equation}\label{eqn: global Chern filtration}
    C_k H^{*}(M_{d,\chi}) \coloneqq \mathrm{span}\left\{ \prod_{i=1}^s c_{k_i}(j_i) \,\Bigg|\, \sum_{i=1}^s k_i \leq k \right\} \subset H^*(M_{d,\chi}, \mathbb{Q})
\end{equation}
where we impose no restriction on the cohomological degree, and $c_{k_i}(j_i)$ is allowed to be any normalized class.\footnote{This corrects the mismatch between perversity and the Chern degrees in \cite[Remark 2.10]{KPS}.} Equivalently, this is the smallest multiplicative filtration where each class $c_k(j)$ lies in the $k$-th step of the filtration.

\medskip

The definition above has the advantage of being explicit and computable, but apparently relies on the normalization in Section \ref{sec: normalized class}. Recall that the normalized classes are closely related to the weight zero descendent algebra via the identification $\BQ[c_0(2), c_2(0), c_1(2),\cdots]\simeq \BD_{\alpha,\inv}$, where $\alpha = (d, \chi)$ as usual. We now explain how the latter can be used to give a more canonical definition of the Chern filtration, without any choice of normalization. To motivate this, we note that the previous definition \eqref{eqn: global Chern filtration} can be viewed as being induced from a parallel Chern filtration $C_\bullet\BQ[c_0(2), c_2(0), c_1(2),\cdots]$ on the free polynomial ring, via the surjective map
$$\BQ[c_0(2), c_2(0), c_1(2),\cdots]\twoheadrightarrow H^*(M_{d,\chi}).$$
We can also define the Chern filtration $C_\bullet H^*(\FM_{d,\chi})$ for the stacks in the same way, using a surjective map from $C_\bullet\BQ[c_0(2),c_1(1), c_2(0), c_1(2),\cdots]$ for any $(d,\chi)$.

\begin{defn}\label{def: chern filtration on descendent}
The Chern filtration $\widetilde 
C_\bullet \BD_\alpha \subset \mathbb{D}_\alpha$ is the minimal multiplicative filtration such that 
$\ch_k(H^j)\in \widetilde C_{k-j+1}\,\BD_\alpha$ for any $k>0$ and $j=0,1,2$. %Note that $\mathrm{ch}_k(H^j)$ corresponds to the normalized class $c_{k-j+1}(j)$, explaining the filtration degree. 
The Chern filtration on the weight zero descendent algebra $\BD_{\alpha,\inv}$ is defined by restriction, i.e.
\[\widetilde C_\bullet  \BD_{\alpha, \inv}=\widetilde C_\bullet \BD_\alpha\cap \BD_{\alpha, \inv}\,.\]
\end{defn}

\begin{rmk}
    Since $\xi(\ch_k(H^j))=p_*\big(\ch_{k-j+2}(\CF)\cdot q^*H^j\big)$, the Chern filtration index $(k-j+1)$ is one less than the Chern character degree of the universal sheaf we take in the realization. The shift by one can be motivated by the codimension of the sheaves we parametrize.
\end{rmk}

\begin{rmk}\label{rmk: chern filtration twisting}
    The Chern filtrations are preserved under the isomorphism 
$$\widetilde C_\bullet\BD_{e^\rho\cdot\alpha}\xrightarrow{\bF_\rho}\widetilde C_\bullet\BD_{\alpha} 
    $$
    from Section \ref{sec: twist by P2}. Indeed, we have $\bF_\rho(\ch_k(H^j))=\ch_k(H^j\cdot e^\rho)\in \widetilde C_{k-j+1}\BD_\alpha$, and the multiplicativity shows the inclusion $\bF_\rho(\widetilde C_\bullet)\subseteq \widetilde C_\bullet$. The other inclusion follows from the same argument applied to $\bF_{-\rho}$. The same statement holds for the weight zero descendent algebra.  
\end{rmk}

%\begin{rmk}\wo{This remark is probably misleading and we can just erase.}
%    By definition, $C_\bullet\,\BD_\alpha$ admits a multplicative splitting. Since the $\bR_{-1}$ operator does not preserve this splitting in degree $1$ generators,  this does not induce a multiplictive splitting for $C_\bullet\,\BD_{\alpha,\inv}$.
%\end{rmk}

% Using the surjective realization homomorphisms, we define the Chern filtration on the cohomology groups of moduli stacks and coprime moduli spaces:
\begin{defn}\label{def: chern filtration on cohomology}
We define the Chern filtration on the cohomology of the moduli stacks and spaces to be
$$\widetilde C_\bullet H^*(\FM_{d,\chi})\coloneqq \textup{Im}(\widetilde C_\bullet\BD_{\alpha}\overset{\xi}{\longrightarrow} H^*(\FM_{d,\chi})),\quad 
\widetilde C_\bullet H^*(M_{d,\chi})\coloneqq \textup{Im}(\widetilde C_\bullet\BD_{\alpha,\inv}\overset{\xi}{\longrightarrow} H^*(M_{d,\chi}))
$$
where $\gcd(d,\chi)=1$ is assumed for the moduli space. 
\end{defn}

The next proposition justifies the same choice of terminology for the above definitions. 

%For moduli stacks, we also expect $\widetilde C_\bullet H^*(\FM_{d,\chi})$ to be important in the study of stacky perverse filtrations in the sense of \cite{Dav}, see Theorem \ref{thm: stacky P=C} for strong numerical evidence. 

%\we{Removed two sentences that refers to stacky $P=C$}\wo{I think it is good to keep some sentence about stacky P=C because this is where the Chern filtration for stacks is first defined. So referring this would explain why we define such a thing. But perhaps we can place it after the proposition below?} \we{Added a sentence below.}

\begin{prop}
The two definitions of the Chern filtration match, i.e.
$$C_\bullet H^*(\FM_{d,\chi})=\widetilde C_\bullet H^*(\FM_{d,\chi}) 
\quad\textup{and}\quad 
C_\bullet H^*(M_{d,\chi})=\widetilde C_\bullet H^*(M_{d,\chi})$$
where $\gcd(d,\chi)=1$ is assumed for the moduli space. 
\end{prop}
\begin{proof}

It suffices to show that the two filtrations match under the identifications of the formal algebras as in Section~\ref{sec: relations between classes}. Consider the commutative diagram of ring homomorphisms
\begin{center}
    \begin{tikzcd}
\BQ[c_0(2),c_1(1), c_2(0),c_1(2), \cdots] \arrow[r, "\sim"] \arrow[d, "{\eta}'"] & \BD_\alpha \arrow[d, "\eta"] \\
\BQ[c_0(2), c_2(0), c_1(2), \cdots] \arrow[r, "\sim"]                & \BD_{\alpha,\inv}. 
\end{tikzcd}
\end{center}
The left and the right columns are equipped with the $C_\bullet$ and $\widetilde C_\bullet$ filtrations, respectively. The top row isomorphism identifies the two Chern filtrations $C_\bullet$ and $\widetilde{C}_\bullet$ by Remark \ref{rmk: chern filtration twisting}. This proves the statement for $\FM_{d,\chi}$. 

Now we consider the moduli space $M_{d,\chi}$ which corresponds to the bottom row. On the left hand side, it is clear that ${\eta}'(C_\bullet)=C_\bullet$ since ${\eta}'$ simply sets $c_1(1)=0$. Therefore, we are left with showing that $\eta(\widetilde C_\bullet)=\widetilde C_\bullet$ on the right hand side. Since $\eta$ restricts to the identity on $\BD_{\alpha, \inv}$, we have $\eta(\widetilde C_\bullet)\supseteq \widetilde C_\bullet$. On the other hand, recall the formula
\[\eta=\sum_{j\geq 0}\frac{(-1)^j}{j!\cdot d^j}\ch_1(H)^j\bR_{-1}^j\,.\]
It is clear from the definition of $\bR_{-1}$ and the $\widetilde C_\bullet$ filtration that $\bR_{-1}(\widetilde C_\bullet)\subseteq \widetilde C_{\bullet-1}$ and $\ch_1(H)\widetilde C_{\bullet}\subseteq \widetilde C_{\bullet+1}$, since $\ch_1(H)\in \widetilde C_1$. We conclude that $\eta(\widetilde C_\bullet)\subseteq \widetilde C_\bullet$, which completes the proof.
\end{proof}

We expect the Chern filtration $C_\bullet H^*(\FM_{d,\chi})$ to play a role in a \textit{stacky} version of the $P=C$ conjecture; see Theorem \ref{thm: stacky P=C} for the numerical evidence.

\subsubsection{The conjecture and the result}

We are now ready to state the (strengthened) $P=C$ conjecture. 

\begin{conj}
\label{conj: global P=C}
    For coprime $d\geq 1$ and $\chi\in \mathbb{Z}$, we have
    \[
    P_\bullet H^*(M_{d,\chi}, \mathbb{Q}) = C_\bullet H^*(M_{d,\chi}, \mathbb{Q}).
    \]
\end{conj}

It was shown in \cite[Theorem 0.6]{MSY} that the Chern filtration is contained in the perverse filtration in small cohomological degrees, i.e.
\[P_\bullet H^{\leq 2d-4}(M_{d,\chi}) \supset C_\bullet H^{\leq 2d-4}(M_{d,\chi}) \,.\]
In addition, the equality between the two filtrations up to cohomological degree 4 is shown in \cite{Yuan4}. 

\begin{rmk}
\label{rmk: P=C}
Before presenting more evidence, we discuss a number of structural implications of Conjecture \ref{conj: global P=C} on both sides of the equality.

\begin{enumerate}[label=(\roman*)]
    \item The perverse filtration is in general \textit{not} multiplicative, meaning that we do not necessarily have
    \[
    P_k H^*(X, \mathbb{Q}) \times P_{k'}H^{*}(X, \mathbb{Q}) \overset{\cup}{\longrightarrow} P_{k+k'} H^{*}(X, \mathbb{Q}).
    \]
    See \cite[Exercise 5.6.8]{dC_pcmi} for a counterexample, and \cite[Theorem 0.6]{dCMS} on how the $P=W$ conjecture can be reduced to a multiplicativity statement. On the other hand, the (global) Chern filtration is multiplicative by definition. Thus Conjecture \ref{conj: global P=C} would imply that the perverse filtration on $M_{d,\chi}$ is multiplicative as well. 
    
    \smallskip

    \item The Hilbert--Chow morphism is equidimensional, and thus $P_\bullet H^*(M_{d,\chi})$ stabilizes in $2g$ steps, where $g=\frac{1}{2}(d-1)(d-2)$ is the relative dimension. This implies that the Chern filtration should also stabilize, i.e. $C_{2g}H^*(M_{d,\chi}) = H^*(M_{d,\chi})$. Explicitly, every monomial in the normalized classes $c_k(j)$, possibly with Chern degrees $> 2g$, should be expressed by those with Chern degrees $\leq 2g$ in the cohomology ring $H^*(M_{d,\chi}, \mathbb{Q})$.

    \smallskip

    \item According to \cite[Theorem 0.4]{MS_chi-indep}, the perverse numbers (i.e. dimensions of graded pieces of the perverse filtration) of $M_{d,\chi}$ are independent on the Euler characteristic $\chi$, and satisfy two Hard Lefschetz type symmetries \cite[Theorem 2.1.4]{dCM}. Thus the graded dimensions of the Chern filtration on $M_{d,\chi}$ should also be independent on $\chi$ and satisfy these symmetries. This prediction is particularly non-trivial given Theorem \ref{chi-dep}, which states that the cohomology ring structure of $M_{d,\chi}$ is \textit{as $\chi$-dependent as possible.}

    \smallskip

    \item Also from the equidimensionality of the Hilbert--Chow map and the two Hard Lefschetz symmetries, the perverse filtration $P_\bullet H^*(M_{d,\chi})$ is not `full' when the cohomological degree is large. For example, it only has the highest perversity piece 
    \[
    0 = P_{2g-1}H^{\mathrm{top}}(M_{d,\chi}) \subset P_{2g} H^{\mathrm{top}}(M_{d,\chi}) = H^{\mathrm{top}}(M_{d,\chi}, \mathbb{Q})
    \]
    in top cohomological degree. More generally, we have 
    \[
    P_{2\ell-1} H^{2(b+\ell)}(M_{d,\chi}, \mathbb{Q})=0 \quad \textrm{for \,}\ell\geq 1.
    \]
    Conjecture \ref{conj: global P=C} thus predicts that monomials in $c_k(j)$ that lie in $C_{2\ell-1} H^{2(b+\ell)}(M_{d,\chi})$ should also vanish, providing a wealth of non-trivial relations in the cohomology ring $H^*(M_{d,\chi}, \mathbb{Q})$. For example, the base relations (cf. Proposition \ref{prop: base relations}) are a special case of this prediction since
    \[\prod_{i=1}^{b+1}c_{k_i}(2)\in C_k H^{2b+2k+2}(M_{d,\chi})\subseteq C_{2(k+1)-1}H^{2(b+k+1)}(M_{d,\chi})\]
    where $k=\sum_{i=1}^{b+1} k_i$. It remains to be studied how these (conjectural) relations determine the ring structure.

\end{enumerate}

\end{rmk}

\smallskip

Here is our main evidence for the $P=C$ conjecture:

\begin{thm}
    \label{thm: global P=C}
    Conjecture \ref{conj: global P=C} holds for all moduli spaces $M_{d,\chi}$ with $d\leq 5$ and $\gcd(d,\chi)=1$. In particular, the perverse filtrations on these moduli spaces are multiplicative.
\end{thm}

\begin{proof}
    According to \cite[Corollary 1.3]{KPS}, the perverse filtration on $M_{d,\chi}$ is characterized by the cup product with the class $\xi = c_0(2)$. Thus the cohomology ring structure completely determines the perverse filtration. The theorem follows by straightforward verifications of the $P=C$ equality as explained in \cite[Section 2.4]{KPS}.
\end{proof}

We also record the following result towards a stacky $P=C$ conjecture, without going into details on how the perverse filtration on the stacks \cite{Dav} is defined. The proof is again a direct verification in \textsc{Macaulay2}.\footnote{See Remark \ref{rmk: stacky perverse numbers} on how the stacky perverse numbers can be computed via BPS integrality.}

\begin{thm}
    \label{thm: stacky P=C}
    For the moduli stacks $\FM_{2,0},\, \FM_{3,0},\, \FM_{4,2},\, \FM_{4,0}$, the graded dimensions of the perverse filtration and the Chern filtration coincide in all cohomological degrees.
\end{thm}

The same numerical match holds for $\FM_{5,0}$ up to cohomological degree $32$ where we are able to determine the ring structure, cf. Section \ref{sec: M50}.

\subsection{Refined GV/PT/Nekrasov correspondence}\label{sec: GV/PT}
As we have discussed in the overview, one major motivation to study $M_{d,\chi}$, and in particular the perverse filtration $P_\bullet H^*(M_{d,\chi})$, arises from its connection to the refined Gopakumar--Vafa invariants of local $\mathbb{P}^2$, which we write simply $K_{\BP^2}$. We organize the numerical information of the perverse filtration in the Laurent polynomial in two variables
\[\Omega_d(q,t)=(-1)^{d^2+1}q^{-g}t^{-b}\sum_{i,j\geq 0}\dim \mathrm{gr}_i^P \mathit{IH}^{i+j}(M_{d,\chi})q^i t^j\,,\]
where $b=\frac{d(d+3)}{2}$ and $g=\frac{(d-1)(d-2)}{2}$ are the dimension of the base and fibers of the Hilbert--Chow map, respectively. The right hand side does not depend on $\chi$ by the isomorphism in \eqref{eqn: MS_chi_indep}. The two hard Lefschetz symmetries imply that $\Omega_d$ is invariant under the two symmetries $q\leftrightarrow q^{-1}$ and $t\leftrightarrow t^{-1}$.

We define the refined Gopakumar--Vafa invariants $N_d^{j_L, j_R}$ for $d\geq 1$ and $j_L, j_R\in \frac{1}{2}\BZ_{\geq 0}$ to be the unique integers satisfying the equality
\begin{equation}\label{eq: refinedbps}
    \Omega_d(q, t)=\sum_{j_L, j_R}(-1)^{2j_L+2j_R}N_{d}^{j_L, j_R}\chi_{j_L}(q)\chi_{j_R}(t)\end{equation}
where 
\[ \chi_j(x)= \frac{x^{2j+1}-x^{-2j-1}}{x-x^{-1}}=x^{-2j}+x^{-2j+2}+\ldots+x^{2j-2}+x^{2j}\,.\]
The numbers $N_d^{j_L, j_R}$ are supposed to be the numbers of BPS particles with charge $d$ and quantum spin numbers $(j_L, j_R)$ in a certain 5D gauge theory; see for example the discussion in the introductions of \cite{IKV} or \cite{CKK}. These numbers are a refinement of the standard Gopakumar--Vafa invariants, which conjecturally control the enumerative geometry (Gromov--Witten or Pandharipande--Thomas invariants) of $K_{\BP^2}$. Indeed, the definition of Maulik--Toda \cite{MT} of the (unrefined) Gopakumar--Vafa invariants $n_{g,d}$ of $K_{\BP^2}$ is 
\[(-1)^{d^2+1}\Omega_d(q, -1)=\Omega_d(-q,1)=\sum_{g\geq 0}n_{g,d}\big(q^{1/2}+q^{-1/2}\big)^{2g}\,.\]

Choi, Katz and Klemm propose in \cite{CKK} to use refined stable pairs to define $N_d^{j_L, j_R}$. Comparing their proposal to our definition of $N_d^{j_L, j_R}$ leads to a refined GV/PT correspondence, which we now formulate precisely.

The (unrefined) Pandharipande--Thomas invariants are integers $\PT_{d, n}$ which (virtually) count stable pairs on $K_{\BP^2}$ with discrete invariants $(d, n)$. Pandharipande--Thomas invariants admits a refinement\footnote{See for example \cite{ON} or \cite{CKK}. Our formal variable $t$ corresponds to $\kappa^{1/2}$ in \cite{ON} or to $\BL^{1/2}$ in \cite{CKK}.}
\[\PT_{d,n}^{\mathsf{ref}}\in \BZ[t^{\pm 1}]\]
symmetric under $t\leftrightarrow t^{-1}$ which specializes to $\PT_{d, n}$ when $t=1$. The refined PT vertex is a generating series in formal variables $q, t, Q$:
\[Z^{\mathsf{ref}}_{\mathsf{PT}}=\sum_{d\geq 0}Q^d\sum_{n\in \BZ}\PT_{d,n}^{\mathsf{ref}}\cdot (-q)^n\,.\]
The refined PT vertex $Z^{\mathsf{ref}}_{\mathsf{PT}}$ can be combinatorially calculated (at least up to $Q^{14}$ in \cite{KPS}) via the refined topological vertex \cite{IKV, IK}. Alternatively, it can be read from the Nekrasov partition function of rank 2 instantons, as explained in \cite[Section 3.3]{KPS}, which also admits a combinatorial description. The refined GV/PT correspondence is the following expression for the refined topological vertex in terms of $\Omega_d$:
\begin{equation}
    \label{eq: gvpt}Z^{\mathsf{ref}}_{\mathsf{PT}}=\mathrm{PE}\left(-\frac{q}{(1-qt)(1-q/t)}\sum_{d\geq 1}\Omega_d(q,t)Q^d\right)\,.
\end{equation}
This conjecture is an equivalent formulation of \cite[Conjecture 3.1]{KPS}. The specialization $t=1$ of the conjecture is precisely the (also conjectural) unrefined GV/PT correspondence for $K_{\BP^2}$. On the other hand, the specialization $q=t$, known as the Nekrasov--Shatashvili limit (which forgets the information of the perverse filtration, namely $\Omega_d(q,q)=\overline{\mathit{IE}}(M_{d,\chi},q))$, was shown in \cite{BouTak} to be related to the relative Gromov--Witten theory of the pair $(\BP^2, E)$ where $E\subseteq \BP^2$ is a smooth elliptic curve.

\begin{rmk}
    The numbers $N_d^{j_L, j_R}$ are defined in \cite[(8.1)]{CKK} from $Z^{\mathsf{ref}}_{\mathsf{PT}}$; see Table 1 in loc. cit. for these numbers up to $d=7$. Conjecture \eqref{eq: gvpt} is the statement that our definition \eqref{eq: refinedbps} matches with theirs.  
\end{rmk}

\begin{rmk}

\label{rmk: stacky perverse numbers}

    It is quite interesting to notice the similarity between \eqref{eq: gvpt} and the BPS integrality formulas, for instance \eqref{eqn: BPS for Poincare series}. Indeed, the $Q^d$ coefficient of
    \begin{equation}
        \label{eq: perversebpsintegrality}
    \mathrm{PE}\left(-\frac{q}{1-qt}\sum_{d\geq 1}\Omega_d(q,t)Q^d\right)\end{equation}
    is the generating function keeping track of the numbers $\dim \mathrm{gr}^P_i H^{i+j}(\FM_{d, 0})$ where $P_\bullet$ is a stacky perverse filtration as in \cite{Dav}. We do not have any good explanation for the apparent similarity between \eqref{eq: gvpt} and \eqref{eq: perversebpsintegrality}. 
\end{rmk}

Our calculation of the cohomology rings allow us to compute $\Omega_d$ for $d\leq 5$, and thus check the refined GV/PT conjecture up to $d=5$. Previously, this was known only up to $d=4$ \cite[Theorem 0.7]{KPS}, using the calculation of the cohomology ring in \cite{CM}. 

\begin{thm}
\label{thm: GV/GW/Nekrasov}
The refined GV/PT correspondence \eqref{eq: gvpt} holds up to order $Q^5$.
\end{thm}

%\m{I commented out the older version of this section}
%\subsection{GV/GW/Nekrasov correspondence.} As we have discussed in the overview, one major motivation to study $M_{d,\chi}$, and in particular the perverse filtration $P_\bullet H^*(M_{d,\chi})$, arises from its connection to refined BPS invariants. We recall that these are defined as \cite[Section 0.1]{KPS}
%\begin{equation}
%\label{eq: refined BPS}
%    n_d^{i,j}:= \dim \mathrm{Gr}_i^P H^{i+j}(M_{d,\chi}, \mathbb{Q}).
%\end{equation}

%A second proposal, due to Nekrasov--Okounkov \cite{ON}, predicts that the refined BPS invariants can be computed by certain Nekrasov partition functions. This proposal as well as its connection to Pandharipande--Thomas theory is explained in detail in \cite[Section 3]{KPS}. The Nekrasov partition function can be easily computed, giving the combinatorial version of refined BPS invariants for $d\leq 14$, see \textit{loc. cit.} Section 3.5.

%\begin{thm}
%\label{thm: GV/GW/Nekrasov}
%    For degree $d=5$, the invariants (\ref{eq: refined BPS}) defined by the perverse filtration match with those defined by Nekrasov partition functions. This yields the GV/GW correspondence for $d=5$.
%\end{thm}

%The proof of Theorem \ref{thm: GV/GW/Nekrasov} is completely parallel to that of \cite[Theorem 0.7]{KPS}, where the same statement is proved for $d\leq 4$. We compute the invariants (\ref{eq: refined BPS}) directly using the cohomology ring structure of $M_{5,\chi}$, and match them with their combinatorial counterpart. \qed

\subsection{Euler characteristics of line bundles.} \label{sec: CHung-Moon}
We finish with a conjecture of Chung--Moon:

%\footnote{The class $\alpha_d$ in their notation equals the tautological class $c_0(2)$.}

\begin{conj}[{\cite[Conjecture 8.2]{CM}}]
    \label{conj: CM} For $d\geq 1$ and coprime $\chi$, we have
    \begin{equation}
    \label{euler_char}
    \chi\big(M_{d, \chi},\, m\cdot c_0(2)\big) = \binom{m+3d-1}{m}.
    \end{equation}
    
\end{conj}
Here we write $c_0(2)$ also to denote its associated line bundle, i.e. the pull-back of $\mathcal{O}(1)$ under the Hilbert--Chow map. The Euler characteristics in (\ref{euler_char}) are known as \textit{generalized $K$-theoretic Donaldson numbers,} and are relevant to the strange duality conjecture \cite[Section 8]{CM} for $\mathbb{P}^2$. Note that this is formulated only for $M_{d,1}$ (and verified for $d\leq 4$) in \cite{CM}. The next proposition justifies our extension of the conjecture to arbitrary $\chi$ coprime to $d$. %\wecomment{Ignorant question: at which point in this subsection do we really need $\mathrm{gcd}(d,\chi)= 1$? For singular $M_{d,\chi}$, I think the trivial $D$-module $\mathcal{O}_M$ still correspond to $\mathrm{IC}_M$, right? So if we rewrite $m\cdot c_0(2)$ as $h^*\mathcal{O}(m)$ then Proposition 6.14 seems to work for singular $M_{d,\chi}$ as well. We just need smooth $M_{d,1}$ to check the conjecture holds.} \wecomment{Answer: $\mathrm{IC}_M$ only corresponds to $\mathcal{O}_M$ on the smooth locus. In general it corresponds to an intermediate extension of $\mathcal{O}_M$ which can be quite complicated. So we do need coprimality here.}

\begin{prop}
\label{prop: CM_chi-indep}
    For a fixed $d$, the Euler characteristic in (\ref{euler_char}) is independent of the choice on $\chi$. In particular, Conjecture \ref{conj: CM} holds for $M_{d,1}$ if and only if it holds for all coprime $M_{d,\chi}$.
\end{prop}

\begin{proof}
    We first observe that the Euler characteristics in question are computed by the Hilbert polynomial of the derived push-forward $Rh_* \mathcal{O}_{M_{d,\chi}}$, where
    \[
    h: M_{d,\chi} \longrightarrow |d\cdot H|
    \]
    is the Hilbert--Chow map. Thus it suffices to show that $Rh_* \mathcal{O}_{M_{d,\chi}}$ is $\chi$-independent over the base $|d\cdot H| = \BP^b$. This follows from \cite[Theorem 0.4]{MS_chi-indep} combined with the theory of Hodge modules. Indeed, for any $\chi$ coprime to $d$, let $\pi: \mathcal{C} \to U$ be the universal smooth curve over $U\subset \mathbb{P}^b$; recall that \cite{MS_chi-indep} we have an isomorphism
    \begin{equation}\label{CM_chi-indep}
    Rh_* \mathbb{Q}_{M_{d,\chi}}[n] \simeq \bigoplus_{i=0}^{2g} \mathrm{IC}\left(\wedge^i R^1\pi_* \mathbb{Q}_\mathcal{C}\right)[-i+g]
    \end{equation}
    in the bounded derived category $D^b(\mathrm{MHM}(\BP^b))$, where $n = \dim M_{d,\chi}$, and
    \[
    \mathbb{Q}_{M_{d,\chi}}[n] = (\CO_{M_{d,\chi}}, F_\bullet), \quad 0 = F_{-1} \CO_{M_{d,\chi}} \subset F_0 \CO_{M_{d,\chi}} =\CO_{M_{d,\chi}} 
    \]
    is the trivial Hodge module. Consider its associated de Rham complex
    \[
    \mathrm{DR}(\mathcal{O}_{M_{d,\chi}}) = \left[\mathcal{O}_{M_{d,\chi}} \to  \Omega_{M_{d,\chi}}^1 \to \cdots \to  \Omega_{M_{d,\chi}}^n\right][n]
    \]
    with the induced Hodge filtration $F_\bullet \mathrm{DR}(\mathcal{O}_{M_{d,\chi}})$. A direct computation yields
    \[
    \mathrm{gr}^F_{-\ell} \mathrm{DR}(\mathcal{O}_{M_{d,\chi}}) = \Omega_{M_{d,\chi}}^\ell [n-\ell].
    \]
    Taking $\ell=0$, we have 
    \begin{align*}
        Rh_*\CO_{M_{d,\chi}} &= Rh_*\left(\mathrm{gr}^F_0\mathrm{DR}(\CO_{M_{d,\chi}})\right) = \mathrm{gr}^F_0 \mathrm{DR}\left(Rh_*\mathbb{Q}_{M_{d,\chi}}[n]\right)\\
        & \simeq \mathrm{gr}^F_0\mathrm{DR}\left(\bigoplus_{i=0}^{2g} \mathrm{IC}\left(\wedge^i R^1\pi_* \mathbb{Q}_\mathcal{C}\right)[-i+g]\right)
    \end{align*}
    is independent of $\chi$, where we used in the second equality the commutativity of $\mathrm{gr}_\bullet^F \mathrm{DR}(-)$ with derived push-forward \cite[Section 2.3.7]{Saito}. This finishes the proof.
\end{proof}

\begin{thm}
    \label{thm: CM}
    Conjecture \ref{conj: CM} holds for all coprime moduli spaces $M_{5,\chi}$.
\end{thm}

\begin{proof}
    By Proposition \ref{prop: CM_chi-indep}, it suffices to consider $M= M_{5,1}$. We first determine the point class in $H^*(M, \mathbb{Q})$. Recall that $M_{d,\chi}$ has only $(p,p)$-classes \cite[Theorem 0.4.1]{Bou}. It follows that
    \begin{equation}
    \label{eq: thm5.7_1}
    \chi(M, \mathcal{O}_M) = \sum_{k\geq 0} (-1)^k\dim H^k(M, \mathcal{O}_M) = \sum_{k\geq 0} (-1)^k h^{0,k}(M) = 1.
    \end{equation}
    On the other hand, the Hirzebruch--Riemann--Roch theorem gives
    \begin{equation}
    \label{eq: thm5.7_2}
    \chi(M, \mathcal{O}_M) = \int_M \mathrm{ch}(\mathcal{O}_M)\cdot \mathrm{td}(M) = \int_M \mathrm{td}(M).
    \end{equation}
    The Todd class $\mathrm{td}(M)$ can be written in terms of the normalized tautological classes $c_k(j)$ using the expression of the tangent bundle
    \[
    \mathcal{T}_M = -Rq_* R\mathcal{H}\kern -1pt \mathit{om}(\mathbb{F}, \mathbb{F}) + \mathcal{O}_M \in K(M),
    \]
    see for example \cite[Section 2.3.3]{KPS}. Combining Equations (\ref{eq: thm5.7_1}) and (\ref{eq: thm5.7_2}) gives the point class in $H^*(M, \mathbb{Q})$ in terms of the normalized classes\footnote{Explicitly, it is represented by $\frac{7812500000}{62868347} c_0(2)^8 c_2(2)^6$.}. The proof of Theorem \ref{thm: CM} is then a straightforward verification in the cohomology ring, using the Hirzebruch--Riemann--Roch formula
    \[
    \chi\big(M, m\cdot c_0(2)\big) = \int_M \mathrm{ch}(m\cdot c_0(2))\cdot \mathrm{td}(M).
    \]
    Note that the LHS of (\ref{euler_char}) is a polynomial in the variable $m\in \mathbb{Z}$, so we only need to check the equality for a finite number of values.
\end{proof}

\appendix

\section{An identity in the quadratic descendent algebra}

In this appendix we state and prove an identity involving the interaction between the Virasoro derivations $\bR_n$ and a special class $C$ in (the completion of) the quadratic descendent algebra $\BD\otimes \BD$. This is a purely formal identity and it holds for the descendent algebra $\BD=\BD^X$ of an arbitrary target variety $X$ (in the main text $X$ is always $\BP^2$); for simplicity, we will assume that $X$ does not have odd cohomology and that $\dim(X)$ is even, but this can be adapted with minor changes in the signs. This identity is used in the proof that the Virasoro operators preserve the ideal of $\gmr$ and we hope that it can be of use in other contexts in the study of Virasoro constraints.

\smallskip

Define $C$ by
\[C=\exp\left(\sum_{\substack{a,b\geq 0\\(a,b)\neq (0,0)}}\sum_i (-1)^{b+d_i^L}(a+b-1)!\ch_a(\gamma_i^L)\otimes \ch_b(\gamma_i^R)\right)\]
where
\[\sum_i \gamma_i^L\otimes \gamma_i^R=\Delta_\ast \td(X)\]
is the Künneth decomposition of the push-forward of $\td(X)$ along the diagonal and $\gamma_i^L\in H^{2d_i^L}(X)$. Let $C_j\in \BD\otimes \BD$ be the degree $2j$ part of $C$. Recall that $C_j$ is a natural lift to the quadratic descendent algebra of the Chern class $c_j(-\RHom_p(\CF, \CF'))$.

We define also the operators \[\partial_n\colon \BD\otimes \BD\to \BD\otimes \BD\]
by the formula
\[\partial_n=\bR_n\otimes \id+\sum_{k=-1}^n\binom{n+1}{k+1}\id\otimes \bR_k\,.\]
Note that both the class $C$ and the operators $\partial_n$ make sense in $\BD_\alpha\otimes \BD_{\alpha'}$ as well. 

\begin{thm}\label{thm: quadraticidentity}
    For $n\geq 0$, we have
    \[\partial_n(C)=\left(\sum_{0\leq a+b\leq n}\frac{a!(n-a)!}{(n-a-b)!}\sum_i(-1)^{d_i^L+1}\ch_a(\gamma_i^L)\otimes \ch_b(\gamma_i^R)\right)C\]
    in (the completion of) $\BD_\alpha\otimes \BD_{\alpha'}$.
\end{thm}

Note that for $n=0$, the formula reads (see for example \eqref{eq: Thm A a=b=0} below)
\[
(\bR_0 \otimes \id + \id \otimes \bR_0 +\id \otimes \bR_{-1})( C) = -\chi(\alpha, \alpha') C.
\]
We illustrate this special case first. Applying the derivation $\bR_0\otimes \id+\id\otimes \bR_0$ to the definition of $C$ yields
\begin{equation}
    \label{eq: newtonidentity}
\sum_{j\geq 1}jC_j=\left(\sum_{\substack{a,b\geq 0\\(a,b)\neq (0,0)}}\sum_i (-1)^{b+d_i^L}(a+b)!\ch_a(\gamma_i^L)\otimes \ch_b(\gamma_i^R)\right) C\,,\end{equation}
which is just Newton's identity for $-\RHom_p(\CF, \CF')$. The $n=0$ case of Theorem \ref{thm: quadraticidentity} then follows from the following lemma.

\begin{lem}\label{lem: R-1quadraticidentity}
    We have
    \[(\id\otimes \bR_{-1})(C_j) = - (\bR_{-1}\otimes \id)(C_j)=-(j-1+\chi(\alpha, \alpha'))C_{j-1}\,\]
    in $\BD_{\alpha}\otimes \BD_{\alpha'}$.
\end{lem}
\begin{proof}
We have
\begin{align*}(\id\otimes \bR_{-1})(C)&=\left(\sum_{\substack{a,b\geq 0\\(a,b)\neq (0,0)}}\sum_i (-1)^{b+d_i^L}(a+b-1)!\ch_a(\gamma_i^L)\otimes \ch_{b-1}(\gamma_i^R)\right) C\\
&=\left(\sum_{\substack{a,b\geq 0}}\sum_i (-1)^{b+1+d_i^L}(a+b)!\ch_a(\gamma_i^L)\otimes \ch_{b}(\gamma_i^R)\right) C\end{align*}
where in the second line we replaced $b-1$ by $b$. Now the term in the sum with $a=b=0$ is
\begin{equation}
\label{eq: Thm A a=b=0}
-\sum_i (-1)^{d_i^L}\ch_0(\gamma_i^L)\otimes \ch_{0}(\gamma_i^R)=-\sum_i(-1)^{d_i^L}\int_{X}\ch(\alpha)\cdot \gamma_i^L\int_{X}\ch(\alpha')\cdot \gamma_i^R=-\chi(\alpha, \alpha')\,.
\end{equation}
The terms with $(a,b)\neq (0,0)$ are, up to a minus sign, the ones in \eqref{eq: newtonidentity}, so we conclude the proof of the lemma for the $\id\otimes \bR_{-1}$. The statement for $\bR_{-1}\otimes \id$ is similar; it can also be deduced from the fact that $-\RHom_p(\CF, \CF')$ is unaffected when we twist both $\CF$ and $\CF'$ by a line bundle, which formally implies that its Chern classes are annihilated by $\bR_{-1}\otimes \id+\id\otimes \bR_{-1}$.
\end{proof}

\begin{proof}[Proof of Theorem \ref{thm: quadraticidentity}]
Observe that $\partial_n$ is a derivation. When we apply $\partial_n$ to $C$, we obtain a sum of the form
\[\left(\sum_{\substack{a,b\geq 0}}\sum_i (-1)^{b+d_i^L}K_{a,b}\ch_a(\gamma_i^L)\otimes \ch_{b}(\gamma_i^R)\right) C\]
where $K_{a,b}$ are some constants. These constants are a sum of contributions of each summand of $\partial_n$:
\begin{enumerate}[label=(\roman*)]
    \item When $\bR_{n}\otimes \id$ hits $(-1)^{b+d_i^L}(a+b-n-1)!\ch_{a-n}(\gamma_i^L)\otimes\ch_b(\gamma_i^R)$ we get
    \[(-1)^{b+d_i^L}(a+b-n-1)!(a)_{n+1}\ch_a(\gamma_i^L)\otimes \ch_b(\gamma_i^R)\]
    where $(a)_{m}$ denotes the falling factorial
    \[(a)_{m}=a(a-1)\ldots(a-m+1)\,.\]
    Thus, we get a contribution of $(a+b-n-1)!(a)_{n+1}$ to the constant $K_{a,b}$ coming from this summand. Note that if $a< n$ then $(a)_n=0$, so the formula for the contribution makes sense as long as $a+b\geq n+1$. If $a+b\leq n$ then the contribution of this term is also set to 0.
    \item When $\id\otimes \bR_{k}$ hits $(-1)^{b-k+d_i^L}(a+b-k-1)!\ch_{a}(\gamma_i^L)\otimes\ch_{b-k}(\gamma_i^R)$ we get
    \[(-1)^{b-k+d_i^L}(a+b-k-1)!(b)_{k+1}\ch_a(\gamma_i^L)\otimes \ch_b(\gamma_i^R)\,.\]
    Hence, the summand $\binom{n+1}{k+1}\id\otimes \bR_{k}$ contributes to $K_{a,b}$ with
    \[
(-1)^k\binom{n+1}{k+1}(a+b-k-1)!(b)_{k+1}\,.\]
As before, this formula holds as long as $a+b\geq k+1$.
\end{enumerate}
Now for $(a,b)$ with $a+b>n$, we have 
\[K_{a,b}=(a+b-n-1)!(a)_{n+1}+\sum_{k=-1}^{n} (-1)^k\binom{n+1}{k+1}(a+b-k-1)!(b)_{k+1}=0\,,\]
where the last equality results in an application of Lemma \ref{lem: binomialidentity} with $m=n+1,\, \ell=k+1$, which we will prove next.

\smallskip

When $a+b\leq n$, we have
\begin{align*}K_{a,b}&=\sum_{k=-1}^{a+b-1} (-1)^k\binom{n+1}{k+1}(a+b-k-1)!(b)_{k+1}\\
&=a!b!\sum_{k=-1}^{b-1} (-1)^k\binom{n+1}{k+1}\binom{a+b-k-1}{a}=-a!b!\binom{a+b-n-1}{b}
\\
&=(-1)^{b-1}\frac{a!(n-a)!}{(n-a-b)!}
\end{align*}
where the penultimate identity is \eqref{eq: binomialidentity}. This finishes the proof.
\qedhere
\end{proof}

\begin{lem}\label{lem: binomialidentity}
Let $m\in \BZ_{\geq 0}$ and $a,b\in \BC$. Then the following identity holds: 
\[(a)_m=\sum_{\ell=0}^m (-1)^\ell \binom{m}{\ell}(a+b-\ell)_{m-\ell}(b)_\ell\,.\]
\end{lem}
\begin{proof}
    This is a polynomial identity in $a$ and $b$, so it is enough to prove it for $a,b\in \BZ$ with $a,b>m$. In this case, we may replace the falling factorials by a quotient of factorials, and the identity we want can be rewritten as
    \begin{equation}
        \label{eq: binomialidentity}
   \binom{a+b-m}{b}=\sum_{\ell=0}^m(-1)^\ell\binom{m}{\ell}\binom{a+b-\ell}{a}\,. \end{equation}
    This follows by taking the $z^{a+b}$ coefficient in the identity of generating series
    \[\frac{z^a}{(1-z)^{a+1-m}}=(1-z)^m\frac{z^a}{(1-z)^{a+1}}.\qedhere
    \] 
\end{proof}

%\bibliography{reference.bib}
%\bibliographystyle{alpha}

\end{document}